\documentclass[11pt]{amsart}

\usepackage[utf8]{inputenc}

\usepackage[margin=1in]{geometry}
\usepackage{amsmath,amssymb}

\usepackage{hyperref}

\usepackage{bbm}
\usepackage[T1]{fontenc}
\usepackage{enumitem}
\usepackage{mathabx}
\usepackage{xcolor}
\usepackage{mathtools}
\mathtoolsset{showonlyrefs}

\theoremstyle{plain}
\newtheorem{thm}{Theorem}[section]
\newtheorem{cor}[thm]{Corollary}
\newtheorem{lemma}[thm]{Lemma}
\newtheorem{prop}[thm]{Proposition}

\theoremstyle{definition}
\newtheorem{defi}[thm]{Definition}

\newtheorem{rem}[thm]{Remark}
\newtheorem{ass}[thm]{Assumption}

\newtheorem{cons}[thm]{Consequence}

\numberwithin{equation}{section}

\usepackage[colorinlistoftodos,prependcaption,color=yellow,textsize=tiny]{todonotes}

\usepackage{amsmath}
\allowdisplaybreaks[1]

\DeclareMathOperator{\loc}{loc}
\DeclareMathOperator{\vp}{\varphi}
\DeclareMathOperator{\PP}{\mathbb{P}}
\DeclareMathOperator{\EE}{\mathbb{E}}
\DeclareMathOperator{\NN}{\mathbb{N}}
\DeclareMathOperator{\MM}{\mathcal{M}}

\DeclareMathOperator{\TT}{\mathbb{T}}
\DeclareMathOperator{\RR}{\mathbb{R}}
\DeclareMathOperator{\ZZ}{\mathbb{Z}}
\DeclareMathOperator{\Z}{\mathcal{Z}}

\DeclareMathOperator{\X}{\mathcal{X}}
\DeclareMathOperator{\XBM}{\X_{\text{BM}}}
\DeclareMathOperator{\Xc}{\X_{\text{cont}}}
\DeclareMathOperator{\XL}{\X_{\text{Lebesgue}}}
\DeclareMathOperator{\XS}{\X_{\text{Sobolev}}}
\DeclareMathOperator{\Xp}{\X_{\text{power}}}
\DeclareMathOperator{\Xind}{\X_{\text{ind}}}
\DeclareMathOperator{\Xii}{\X_{\text{ind}}^\infty}
\DeclareMathOperator{\XBMi}{\X_{\text{BM}}^\infty}
\DeclareMathOperator{\Xpi}{\X_{\text{power}}^\infty}

\title{Solutions to the stochastic thin-film equation for  initial values with non-full support}
\date{\today}
\author{Konstantinos Dareiotis\textsuperscript{1}, Benjamin Gess\textsuperscript{2}, Manuel V. Gnann\textsuperscript{3}, Max Sauerbrey\textsuperscript{4}}
		\thanks{\textsuperscript{1}University of Leeds, United Kingdom (\href{mailto:k.dareiotis@leeds.ac.uk}{k.dareiotis@leeds.ac.uk})}
	\thanks{\textsuperscript{2}Bielefeld University and MPI MiS Leipzig, Germany (\href{mailto:benjamin.gess@math.uni-bielefeld.de}{benjamin.gess@math.uni-bielefeld.de})}
	\thanks{\textsuperscript{3}Delft University of Technology, Netherlands (\href{mailto:m.v.gnann@tudelft.nl}{m.v.gnann@tudelft.nl})}
		\thanks{\textsuperscript{4}Delft University of Technology, Netherlands (\href{mailto:maxsauerbrey97@gmail.com}{maxsauerbrey97@gmail.com})}
\renewcommand\rightmark{Solutions to the STFE for  initial values with non-full support}

\begin{document}
	\maketitle
	\begin{abstract}The stochastic thin-film equation with mobility exponent $n\in [\frac{8}{3},3)$ on the one-dimensional torus with multiplicative Stratonovich noise is considered. We show that martingale solutions exist for non-negative initial values. This advances on existing results in three aspects: (1) Non-quadratic mobility with not necessarily strictly positive initial data, (2) Measure-valued initial data, (3) Less spatial regularity of the noise. This is achieved by carrying out a compactness argument  based solely on the control of the $\alpha$-entropy dissipation and the conservation of mass. 
	\end{abstract}
	\section{Introduction}
	The stochastic thin-film equation
	\begin{equation}\label{Eq201}
	\partial_t u \,=\, -\,\partial_x(u^n\partial_x^3 u )\,+\, \partial_x (u^\frac{n}{2} W)
	\end{equation}
	with  noise $W$
	describes the evolution of the height $u(t,x)$ of a thin liquid film  driven by surface tension and thermal fluctuations. Equation \eqref{Eq201} was derived in \cite{DMES2005} and \cite{GruenMeckeRauscher2006} using the fluctuation dissipation relation and a lubrication approximation based on the Navier-Stokes equations with thermal noise, respectively.

	The construction of solutions to this equation based on the stochastic compactness method has been subject to recent research initiated by \cite{fischer_gruen_2018}. In \cite{fischer_gruen_2018} solutions to \eqref{Eq201} are constructed for the case of a quadratic mobility exponent $n=2$ and the additional consideration of interaction forces between the molecules of the liquid and the underlying substrate. The case of quadratic mobility is special, since then the noise term becomes linear in the film height $u$. Moreover, the  interaction forces expressed through an additional interface potential in \eqref{Eq201} leads to strictly positive film heights for all times. 
	Subsequently, in \cite{GessGann2020} solutions to \eqref{Eq201} with quadratic mobility and with noise interpreted in the Stratonovich sense were constructed, based on a time-splitting scheme. This construction allowed to treat also the case of initial data with non-full support. A related result was derived in \cite{KleinGruen22} by letting the interface potential in \cite{fischer_gruen_2018} tend to zero. Additionally, in  \cite{KleinGruen22} so called $\alpha$-entropy estimates for the solutions were derived, implying that the contact angle between  the substrate and the liquid is zero for almost all times. The case of non-quadratic mobility has only been treated in \cite{dareiotis2021nonnegative}, where solutions to \eqref{Eq201} are constructed for the range of mobility exponents $n\in [\frac{8}{3},4)$. The obtained solutions as well as the given initial condition were required to be positive almost everywhere, since the construction is based on the entropy estimate, which controls the smallness of $u$. We mention also the higher dimensional counterparts \cite{metzger2022existence} and \cite{Sauerbrey_2021} of \cite{fischer_gruen_2018} and \cite{GessGann2020}, respectively. Lastly, we refer to \cite{GGKO21} for numerical simulations of a spatially discretized version of \eqref{Eq201}. 
	
	The aim of this article is to complement the existing literature by constructing solutions to \eqref{Eq201} for  initial data without full support and non-quadratic mobility exponents. 
	
	The construction presented in this work applies for the range of mobility exponents $n\in [\frac{8}{3},3)$. The restriction $n\in (2,3)$ stems from the approach, while the additional requirement $n\ge \frac{8}{3}$ is solely required to ensure that suitable approximations are provided by \cite[Theorem 2.2]{dareiotis2021nonnegative}. This restriction may be alleviated in the future, since we expect  this requirement in \cite{dareiotis2021nonnegative} to be of technical nature.
	
	The applications of the stochastic compactness method in the mentioned articles rely on closing a-priori estimates, known for the deterministic thin-film equation 
	\begin{equation}\label{Eq203}
		\partial_t u \,=\, -\,\partial_x(u^n\partial_x^3 u ),
	\end{equation}
	for suitable approximations of \eqref{Eq201}. In the deterministic setting, these a-priori estimates are given by 
	the energy estimate
	\begin{equation}\label{Eq204}
		\frac{1}{2}\partial_t \int (\partial_x u)^2\, dx\,\le \, -\,\int (u^n \partial_x^3 u)^2\, dx,
	\end{equation}
	the entropy estimate
	\begin{equation}\label{Eq205}
		\partial_t \int G(u)\, dx\,\le\,-\,\int (\partial_x^2 u)^2\, dx,
	\end{equation}
	and the $\alpha$-entropy estimate
	\begin{equation}\label{Eq202}
		\partial_t \int G_\alpha(u)\, dx \,\lesssim_{\alpha,n,\theta}\, -\,\int u^{\alpha+n-2\theta+1}(\partial_x^2u^\theta)^2\, dx\,
		-\,\int u^{\alpha+n-3}(\partial_x u)^4 \,dx,
	\end{equation}
	where $G(u)=\int^u\int^t s^{-n}\,ds\, dt$ and more generally $G_\alpha(u)=\int^u\int^t s^{\alpha-1}\,ds\, dt$. We point out that $G_\alpha=G$ for $\alpha = 1-n$. Moreover, \eqref{Eq202}  holds, if
	\begin{equation}\label{Eq246}
	\tfrac{1}{3}(\alpha+n-2)(2\theta-1-(\alpha+n))-(\theta-1)^2\,>\,0.
	\end{equation}
	A parameter $\theta\in (0,\infty)$ subject to \eqref{Eq246} exists, if $\alpha\in (\frac{1}{2}-n,2-n)$. If $\alpha\in (1-n,2-n)$, the particular choice $\theta=1$ satisfies \eqref{Eq246}, while for the boundary cases $\alpha\in \{\frac{1}{2}-n, 2-n\}$ a version of \eqref{Eq202} applies too, see \cite[Proposition 2.1]{beretta1995nonnegative} for details. 
	
	In the stochastic setting, the time increments of these quantities consist next to the negative dissipation terms from the thin-film operator and the martingale part, also of possibly positive terms arising from the It\^o-correction of the multiplicative noise term. The proofs in \cite{fischer_gruen_2018,metzger2022existence} rely on balancing these additional terms in a combined energy-entropy functional. When using Stratonovich noise,  cancellations of these additional terms appear, which can be used to close the energy estimate independently, at least in the case of quadratic mobility $n=2$, as demonstrated in \cite{GessGann2020}. In the case of Stratonovich noise, also the $\alpha$-entropy estimates can be closed, which is the key in the proofs of \cite{KleinGruen22,Sauerbrey_2021}. Moreover, the use of Stratonovich calculus allows to close the entropy estimate also for non-quadratic mobilities, which was used in \cite{dareiotis2021nonnegative} to afterwards estimate the additional terms in the It\^o expansion of the energy functional. However, the use of the entropy estimate requires the solution to be positive almost everywhere, at least if $n\ge 2$. 
	
	The aim of this article is to provide solutions with possibly non-full support also in the case of a non-quadratic mobility $n\ne 2$. The main problem in this case is that the It\^o expansion of the energy contains terms, which can explode for $u=0$, see \cite[Eq. (4.10)]{dareiotis2021nonnegative} and hence a control on the smallness of the solution is required. The strategy of this article is to let go of the energy estimate and to base the whole analysis on the $\alpha$-entropy estimates as well as the conservation of mass $\partial_t \int u\,dx=0$, which was carried out for the deterministic thin-film equation in \cite[Section 6]{weak_trace}. By restricting ourselves to the range $n\in (2,3)$, we can take $\alpha\in (-1,2-n)$, such that \eqref{Eq202} holds true with $\theta=1$ and $\int G_\alpha(u)\,dx<\infty$ also for functions $u$ without full support. Since the $\alpha$-entropy estimate yields less control on the spatial derivatives of the solution, a weaker notion of solutions than in the cases in which the energy estimate is available is necessary for our analysis. Therefore, we use a weak form of the thin-film operator  introduced for the higher dimensional setting \cite[Eq. (3.2)]{Passo98ona} and  allow the solutions $u$ take values in the space of measures. However, based on the $\alpha$-entropy estimate, we derive that $u(t)$ admits a density with respect to the Lebesgue measure for almost all $t$, which lies in a first order Sobolev space. 
	\subsection{Main result}
	In this article, we prove the existence of weak martingale solutions to \eqref{Eq201}. To this end, we write $m(u)= |u|^n$ and $q(u)=\sqrt{m(u)}=|u|^\frac{n}{2}$ for the mobility function and its square-root, respectively.	We restrict ourselves throughout this article to mobility exponents from the following range.
	\begin{ass}\label{Condition_n}We assume that  $n\in [\frac{8}{3},3)$.
	\end{ass}
	Moreover, we assume $W$ in \eqref{Eq201} to be spatiotemporal noise, which is white in time and colored in space. Specifically, we assume that the noise $W$ is given by the time derivative of the Wiener process $B$, defined by 
	\begin{equation}\label{Eq_Def_B}
		B(t)\,=\, \sum_{k\in \ZZ} \sigma_k \beta^{(k)}
	\end{equation}
	for a family of independent $\mathfrak{F}$-Brownian motions $(\beta^{(k)})_{k\in \ZZ}$ on a  probability space $(\Omega, \mathfrak{A}, \PP)$ with filtration $\mathfrak{F}$.
	Here, we assume that
	\begin{equation}\label{Eq_Def_sigma}\sigma_k
		\,=\, \lambda_k f_k,
	\end{equation}
	where $\Lambda= (\lambda_k)_{k\in \ZZ}$ is a sequence of real numbers and
	\begin{equation}\label{Eq80}
		f_k(x)\,=\, \begin{cases}
			\cos(2\pi k x), &k\ge 1,\\
			\sin(2\pi k x), &k\le -1,\\
			1, &\text{else}
		\end{cases}
	\end{equation}
	are the eigenfunctions of the periodic Laplace operator. 
	We  impose the following smoothness condition on the process $B$ in terms of the sequence $\Lambda$.
	\begin{ass}\label{Condition_sigma}
		It holds
		$\sum_{k\in \ZZ} (k\lambda_k)^2\,<\,\infty$.
	\end{ass}
	As initial values to \eqref{Eq201}, we allow for non-negative Borel measure-valued random variables. To be precise, we introduce the sigma-field $\mathcal{Z}$ on $\MM(\TT)$ as the sigma-field generated by the pre-dual space $C(\TT)$, i.e. an $\MM(\TT)$-valued random variable $X$ is $\Z$-measurable, iff $\langle X, \vp \rangle$ is measurable for each $\vp\in C(\TT)$.
	\begin{ass}\label{Condition_IV} The initial value  $
		u_0\colon \Omega \to \mathcal{M}(\TT)$ is $\mathfrak{F}_0$-$\Z$ measurable and $u_0\ge 0$  almost surely. 
	\end{ass}
	Interpreting \eqref{Eq201} in Stratonovich form, using the notation $q(u)=|u|^\frac{n}{2}$ and the description of the noise \eqref{Eq_Def_B}, we obtain the equivalent It\^o formulation
	\begin{align}\label{STFE_Integral_form}
		d u \,=\, 
		-\, \partial_x( u^n  \partial_x^3 u)\,dt\,
		+ \,\tfrac{1}{2}\sum_{k\in \ZZ}	\partial_x (\sigma_k q'(u) \partial_x (\sigma_k q(u)) )\, dt 
		\,+	\,\sum_{k\in \ZZ} \partial_x (\sigma_k q(u)) \,d\beta^{(k)}
	\end{align}
	of the stochastic thin-film equation. 
	In order to obtain a sufficiently weak formulation for the case of a possibly compactly supported initial value $u_0$, we
	test \eqref{STFE_Integral_form} with a smooth function $\vp\in C^\infty(\TT)$ in the dual pairing $\langle \cdot, \cdot \rangle$ on $\TT$ and rewrite the thin-film operator in the weak form introduced in \cite[Eq. (3.2)]{Passo98ona}. We obtain the formulation
	\begin{align}\begin{split}
			\label{Eq_STFE_rewritten}
			d \langle{u}, \vp\rangle\,=\,& \Bigl[\tfrac{n(n-1)}{2} \langle
			{u}^{n-2} (\partial_x {u})^3, \partial_x\vp\rangle \,+\,
			\tfrac{3n}{2} \langle {u}^{n-1}(\partial_x{u})^2, \partial_x^2\vp\rangle\,+\,
			\langle {u}^{n}\partial_x {u}, \partial_x^3\vp\rangle\Bigr]\,dt 
			\\&-\,\tfrac{1}{2}\sum_{k\in \ZZ}\langle
			\sigma_k q'({u}) \partial_x (\sigma_k q({u})) ,\partial_x \vp\rangle\, dt
			\,-\,\sum_{k\in \ZZ} 
			\langle
			\sigma_k q({u}), \partial_x \vp
			\rangle
			\, d{\beta}^{(k)},
		\end{split}
	\end{align}
	which gives rise to the following notion of martingale solutions to the stochastic thin-film equation  on a fixed time horizon $T\in (0,\infty)$.
	\begin{defi}\label{Defi_sol}
		A martingale solution to \eqref{Eq_STFE_rewritten} consists out of a probability space $(\tilde{\Omega}, \tilde{\mathfrak{A}}, \tilde{\PP})$ with a filtration $\tilde{\mathfrak{F}}$ satisfying the usual conditions, a family of independent $\tilde{\mathfrak{F}}$-Brownian motions $(\tilde{\beta}^{(k)})_{k\in \ZZ}$ and a non-negative, vaguely continuous, $\tilde{\mathfrak{F}}$-adapted  $(\MM(\TT), \Z)$-valued process $\tilde{u}$ defined on $[0,T]$ such that $\tilde{\PP}\otimes dt$-almost everywhere $\tilde{u}\in W^{1,1}(\TT)$ with 
		\begin{equation}\label{Eq255}
			\tilde{u}^{n-2}(\partial_x \tilde{u})^3, \, \tilde{u}^{n-1}(\partial_x \tilde{u})^2, \, \tilde{u}^{n}\partial_x \tilde{u} , \, \tilde{u}^{n-2}\partial_x \tilde{u},\, \tilde{u}^n \,\in\,  L^1([0,T]\times \TT)
		\end{equation}
	almost surely and for every $\vp\in C^\infty(\TT)$, $t\in [0,T]$ we have 
		\begin{align}\begin{split}
			\label{Eq_1}
			\langle\tilde{u}(t), \vp\rangle\,-\, \langle\tilde{u}(0), \vp\rangle\,=\,& \tfrac{n(n-1)}{2}\int_0^t \langle
			\tilde{u}^{n-2} (\partial_x \tilde{u})^3 ,\partial_x\vp\rangle\, ds
			\,+\,
			\tfrac{3n}{2} \int_0^t\langle \tilde{u}^{n-1}(\partial_x\tilde{u})^2 ,\partial_x^2\vp\rangle \, ds\\&+\,
			\int_0^t\langle  \tilde{u}^{n}\partial_x \tilde{u}, \partial_x^3\vp \rangle \,ds
			\,-\,\tfrac{1}{2}\sum_{k\in \ZZ}\int_0^t \langle
			\sigma_k q'(\tilde{u}) \partial_x (\sigma_k q(\tilde{u})) ,\partial_x \vp\rangle \, ds
			\\&-\,\sum_{k\in \ZZ} \int_0^t 
			\langle
			\sigma_k q(\tilde{u}), \partial_x \vp
			\rangle
			\, d\tilde{\beta}^{(k)}_s.
		\end{split}
	\end{align}
	\end{defi}
	\begin{rem}\begin{enumerate}[label=(\roman*)]
		\item 
		The requirement \eqref{Eq255} ensures that the deterministic and stochastic integrals in \eqref{Eq_1} converge.
		\item
		We demand in Definition \ref{Defi_sol}  that for every $\vp\in C^\infty(\TT)$, $t\in [0,T]$ the identity \eqref{Eq_1} holds outside of some $\tilde{\PP}$-null-set. Since the vague continuity of $\tilde{u}$ implies that all the processes in \eqref{Eq_1} are continuous in time, this null-set can be chosen independently of $t$.
		\end{enumerate}
	\end{rem}

	In the course of this article, we prove the existence of weak martingale solutions to the stochastic thin-film equation in the sense of Definition \ref{Defi_sol} under the previous assumptions.
	\begin{thm}\label{thm_main}
		Under the Assumptions \ref{Condition_n}, \ref{Condition_sigma}, \ref{Condition_IV}   and $\sigma_k$ given by \eqref{Eq_Def_sigma}, there exists a martingale solution $(\tilde{\Omega}, \tilde{\mathfrak{A}}, \tilde{\PP})$, $\tilde{\mathfrak{F}}$, $(\tilde{\beta}^{(k)})_{k\in \ZZ}$, $\tilde{u}$ to \eqref{Eq_STFE_rewritten} in the sense of Definition \ref{Defi_sol} such that  $\tilde{u}(0)$ has the same distribution as $u_0$ on $(\MM(\TT), \Z)$. Moreover, $\tilde {u}$ admits the following properties.
		\begin{enumerate}[label=(\roman*)]
			\item  \label{Main_Item_Mass} Mass is conserved, i.e. almost surely   $\langle\tilde{u}(t), \mathbbm{1}_{\TT}\rangle=\langle\tilde{u}(0), \mathbbm{1}_{\TT}\rangle$ for all $t\in [0,T]$.
			\item \label{Main_Item2} Almost surely, $ \tilde{u}\in  L^p([0,T]\times \TT)\cap L^r(0,T;W^{1,r}(\TT)) $ for each $p\in  (n+4,7)$ and $r\in(\frac{n+4}{2},\frac{7}{2})$ with 
			\begin{align}&\label{Eq224}
				\tilde{\EE}\bigl[
				\|\tilde{u}\|_{L^p([0,T]\times \TT)}^p
				\bigr]\, \lesssim_{n,p,\Lambda, T}\,
				\EE\bigl[
				\|u_0\|_{\MM(\TT)}^{p-n} +
				\|u_0\|_{\MM(\TT)}^{p} 
				\bigr] ,\\&\label{Eq225}
					\tilde{\EE}\bigl[
				\|\partial_x \tilde{u}\|_{L^r([0,T]\times \TT)}^r
				\bigr]\, \lesssim_{n,r,\Lambda, T}\,
				\EE\bigl[
				\|u_0\|_{\MM(\TT)}^{r-n} +
				\|u_0\|_{\MM(\TT)}^{r} 
				\bigr] ,
			\end{align}
			whenever the respective right-hand side is finite.
			\item \label{Main_Item6} Almost surely, we have $\tilde{u}\in C([0,T];H^{\kappa}(\TT))$	for $\kappa\in (-\infty, \frac{-1}{2}) $ and if $\gamma\in(0,\frac{1}{2})$, $\mu\in (\frac{n+4}{n+2},\frac{7}{n+2})$ and $\nu\in (1,\frac{7}{n+4})$, it holds
			\begin{align}\begin{split}\label{Eq243}&
					\tilde{\EE}\Bigl[
					\left\|
					\tilde{u}
					\right\|^{\nu}_{W^{\gamma, \frac{2\nu}{2-\nu}}(0,T;W^{-3,\mu}(\TT))}
					\Bigr]\,\lesssim_{\gamma, n, \mu,\nu, \Lambda, T}\, \EE\Bigl[
					\|u_0\|_{\MM(\TT)}^{(n-1-\frac{n^2}{p_{\mu,\nu}})\nu}\,+\,
					\|u_0\|_{\MM(\TT)}^{(n+1)\nu} \Bigr],
				\end{split}
			\end{align}
			where $p_{\mu,\nu}=\max\{ \mu(n+2), \nu(n+4) \}$, whenever the right-hand side is finite.
		
			\item \label{Main_Item5} For $\alpha \in (-1, 2-n)$
			 we have almost surely 
			 \begin{equation}\label{Eq241} \tilde{u}^\frac{\alpha+n+1}{4}\in L^4(0,T;W^{1,4}(\TT)),\quad \tilde{u}^\frac{\alpha+n+1}{2}\in L^2(0,T;H^2(\TT))
			  \end{equation}
			 and
			it holds the $\alpha$-entropy type estimate
			\begin{equation}\label{Eq210}
			\tilde{\EE}\Bigl[
			\bigl\|\partial_x \tilde{u}^\frac{\alpha+n+1}{4}\bigr\|_{L^4([0,T]\times \TT)}^4\,+\, 
			\bigl\|\partial_{x}^2 \tilde{u}^\frac{\alpha+n+1}{2}\bigr\|_{L^2([0,T]\times \TT)}^2
			\Bigr]\,\lesssim_{\alpha, n,\Lambda,T}\,
			\EE\bigl[\|u_0\|_{\MM(\TT)}^{\alpha+1} + \|u_0\|_{\MM(\TT)}^{\alpha+n-1}
			\bigr],
			\end{equation}
			if the right-hand side is finite.
		\end{enumerate}
	\end{thm}

\begin{rem}
	The reason to restrict ourselves to the range of mobility exponents from Assumption \ref{Condition_n} is  twofold. Firstly, the techniques used in the present proof require the assumption $n\in (2,3)$ to work out. Moreover, the additional restriction $n\ge \frac{8}{3}$ ensures that \cite[Theorem 2.2]{dareiotis2021nonnegative} applies and yields the existence of approximating solutions. Hence, proving a version of \cite[Theorem 2.2]{dareiotis2021nonnegative} for $n\in (2, \frac{8}{3})$ would also lead to an extension of the present result to the canonical range $n\in (2,3)$.
\end{rem}
\begin{rem}\label{Rem_Intbility}We convince ourselves, that the regularity statements from Theorem \ref{thm_main} suffice to deduce the most restrictive integrability assumption from Definition \ref{Defi_sol}, namely that $\tilde{u}^{n-2}(\partial_x \tilde{u})^3\in L^1([0,T]\times \TT)$. As a consequence of Theorem \ref{thm_main} \ref{Main_Item2}, we have that almost surely
	\[
	\tilde{u}\in L^{n+4}([0,T]\times \TT),\quad \partial_x\tilde{u}\in L^\frac{n+4}{2}([0,T]\times \TT).
	\]
	Hence, using H\"older's inequality, we can indeed conclude that
	\[
	\|\tilde{u}^{n-2}(\partial_x \tilde{u})^3\|_{L^1([0,T]\times \TT)}\, \le \,  \|\tilde{u}^{n-2}\|_{L^\frac{n+4}{n-2}([0,T]\times \TT)}	\|(\partial_x \tilde{u})^3\|_{L^\frac{n+4}{6}([0,T]\times \TT)}
	\]
	is almost surely finite.
\end{rem}
	\begin{rem}Let  $\tilde{u}$ be the martingale solution to \eqref{Eq_STFE_rewritten} obtained from Theorem \ref{thm_main}. 
		For $\alpha\in (-1, 2-n)$, the $\alpha$-entropy type estimate from Theorem \ref{thm_main} \ref{Main_Item5} yields the following additional properties of $\tilde{u}$.
		\begin{enumerate}[label=(\roman*)]
			\item \label{Item_Rem1} We have that $\tilde{\PP}\otimes dt$-almost everywhere $\tilde{u}^\frac{\alpha+n+1}{2}\in H^2(\TT)$. In particular, by the Sobolev-embedding $\tilde{u}\in C(\TT)$ and hence uniformly supported away from zero and bounded on compact subsets of $\{\tilde{u}>0\}$. Thus, $\tilde{u}\in H^2_{\loc}(\{\tilde{u}>0\})$ and we have 
			\begin{align}&
				\partial_x\tilde{u}^\frac{\alpha+n+1}{4}\,=\, \tfrac{\alpha+n+1}{4}\tilde{u}^\frac{\alpha+n-3}{4}\partial_x \tilde{u},\\&
				\partial_{x}^2\tilde{u}^\frac{\alpha+n+1}{2}\,=\, \tfrac{\alpha+n+1}{2}\tilde{u}^\frac{\alpha+n-1}{2}\partial_{x}^2\tilde{u}\,+\, \tfrac{(\alpha+n+1)(\alpha+n-1)}{4}\tilde{u}^\frac{\alpha+n-3}{2}(\partial_{x}\tilde{u})^2
			\end{align}
			on $\{\tilde{u}>0\}$.	Hence, we conclude that 
			\begin{align}
				\int_0^T\int_{\{\tilde{u}>0\}} \tilde{u}^{\alpha+n-3}(\partial_x\tilde{u})^4\, dx\, dt\,\lesssim_{\alpha,n} \|\partial_x \tilde{u}^\frac{\alpha+n+1}{4} \|_{L^4([0,T]\times \TT)}^4
			\end{align}
		and consequently
		\begin{align}&
			\int_0^T \int_{\{\tilde{u}>0\}}  \tilde{u}^{\alpha +n-1}(\partial_x^2 \tilde{u})^2 \, dx\, dt\,\lesssim_{\alpha,n} \, \int_0^T \int_{\{\tilde{u}>0\}}(\partial_x^2 \tilde{u}^\frac{\alpha+n+1}{2})^2+ \tilde{u}^{\alpha+n-3}(\partial_x\tilde{u})^4\, dt 
			\\&\quad \lesssim_{\alpha,n} \,  \| \partial_x^2 \tilde{u}^\frac{\alpha+n+1}{2}\|_{L^2([0,T]\times \TT)}^2\,+\,\|\partial_x \tilde{u}^\frac{\alpha+n+1}{4} \|_{L^4([0,T]\times \TT)}^4 .
		\end{align}
	Therefore, \eqref{Eq210} implies that 
	\begin{align}\begin{split}\label{Eq242}
			&
		\tilde{\EE}\biggl[
			\int_0^T\int_{\{\tilde{u}>0\}} \tilde{u}^{\alpha+n-3}(\partial_x\tilde{u})^4+ \tilde{u}^{\alpha +n-1}(\partial_x^2 \tilde{u})^2\, dx\, dt
		\biggr]\\&\quad  \lesssim_{\alpha, n, \Lambda, T} \,
			\EE\bigl[\|u_0\|_{\MM(\TT)}^{\alpha+1} + \|u_0\|_{\MM(\TT)}^{\alpha+n-1}
		\bigr]
		\end{split}
	\end{align}
	appealing to the classical form of the $\alpha$-entropy estimate \eqref{Eq202} with $\theta=1$.
			
			\item We demonstrate, that as a consequence of \eqref{Eq210} one can also recover an estimate in the spirit of \eqref{Eq202} for $\theta\ne 1$.
			Arguing as in  \ref{Item_Rem1}, we conclude that $\tilde{\PP}\otimes dt$-almost everywhere  $\tilde{u}^\theta\in H^2_{\loc}(\{\tilde{u}>0\})$ for any $\theta>0$, with 
			\[
			\partial_x^2 \tilde{u}^\theta\,=\, 
			\theta {\tilde{u}}^{\theta-1} \partial_x^2\tilde{u}\,+\, \theta(\theta-1)\tilde{u}^{\theta-2}(\partial_x \tilde{u} )^2,
			\]
			By taking the square on both sides, we obtain that
			\[
			(
			\partial_x^2 \tilde{u}^\theta)^2\,\lesssim_\theta\, 
			\tilde{u}^{2\theta-2} 
			(
			\partial_x^2\tilde{u})^2\,+\,\tilde{u}^{2\theta-4}(\partial_{x}\tilde{u})^4.
			\]
			Hence, using \eqref{Eq242}  we infer 
			\[
				\EE\biggl[
			\int_0^{T}\int_{\TT} \tilde{u}^{\alpha+n-2\theta+1}
			(\partial_x^2\tilde{u}^\theta )^2 \,dx\,ds	\biggr]\, \lesssim_{\alpha,  n,\theta,\Lambda,T} \,\EE\bigl[\|u_0\|_{\MM(\TT)}^{\alpha+1} + \|u_0\|_{\MM(\TT)}^{\alpha+n-1}
			\bigr].
			\]
			\item In \cite[Corollary 3.2]{KleinGruen22} it is shown that the solutions to the stochastic thin-film equation constructed in \cite{KleinGruen22} admit a zero contact angle almost everywhere based on the finiteness of the $\alpha$-entropy dissipation. Following the proof of \cite[Corllary 3.2]{KleinGruen22} we obtain the same statement as a consequence of Theorem \ref{thm_main} \ref{Main_Item5}, namely that   $\tilde{\PP}\otimes dt$-almost everywhere, $\tilde{u}$ admits $0$ as its classical derivative at every point from its zero set.
		\end{enumerate}
	\end{rem}

	\subsection{Outline of the proof}\label{SubSec_Proof}
	As pointed out in the introduction, the main innovation of this article is to provide solutions to \eqref{Eq201} for initial data without full support in the case of a non-quadratic mobility $n\ne 2$. The main difficulty in the analysis of \eqref{Eq201} is to close the deterministic a-priori estimates \eqref{Eq204}, \eqref{Eq205} and \eqref{Eq202} in the stochastic setting, where at least in the case of Stratonovich noise, the entropy and $\alpha$-entropy estimate seem to hold for a wide range of $n$, see for example \cite[Lemma 4.3]{dareiotis2021nonnegative}. However, the additional energy production due to the stochastic term in \eqref{Eq201} seems to require a control on the smallness of the solution, which we are unable to provide in the case of initial data without full support. Since at least for the case $n\ge 2$, the entropy estimate also fails for such initial values, we decided to  rely the whole analysis only on the remaining $\alpha$-entropy estimate \eqref{Eq202} for $\alpha>-1$ as well as the conservation of mass. The key observation for the proof is that the use of the weak formulation \eqref{Eq_STFE_rewritten} and an interpolation of the $\alpha$-entropy dissipation and conservation of mass is sufficient to conclude compactness in suitable spaces. 
	
	To elaborate on the interpolation argument, we confine ourselves to the deterministic setting and set $w=u^\frac{\alpha+n+1}{4}$, where $u$ is a solution to \eqref{Eq203} and $\alpha\in(-1,2-n)$. Then, by the chain rule, integration of \eqref{Eq202} provides an estimate on 
	\begin{equation}\label{Eq206}
	\int_0^T \| \partial_x w \|_{L^4(\TT)}^4 \, dt\,\eqsim_{\alpha,n}\, \int_0^T \int_{\TT} \bigl(u^\frac{\alpha+n-3}{4}\partial_x u\bigr)^4 \,dx\,dt\,=\,
	\int_0^T \int_{\TT} u^{\alpha+n-3}(\partial_x u)^4\,dx\,dt
	\end{equation}
	and the conservation of mass on
	\[
	\sup_{t\in[0,T]} \|w\|_{L^\frac{4}{\alpha+n+1}(\TT)}^\frac{4}{\alpha+n+1}\,=\,
	\sup_{t\in [0,T]} \|u\|_{L^1(\TT)},
	\]
	for non-negative $u$. An application of the Gagliardo-Nirenberg interpolation inequality  allows to find between this estimate in $L^4(0,T;\dot{W}^{1,4}(\TT))$ and $L^\infty(0,T;L^\frac{4}{\alpha+n+1}(\TT))$ an estimate on $w$, which has the same integrability in space and time, see Lemma \ref{lemma_int_p}. Using the identity 
	\[\|u\|_{L^p([0,T]\times \TT)}\,=\, \|w\|_{L^\frac{4p}{\alpha+n+1}([0,T]\times \TT)}^\frac{\alpha+n+1}{4},\]
	 this can be translated to an estimate on $u$, if $p<7$. Together with the estimate on \eqref{Eq206} and the application of H\"older's inequality
	\[
	\|\partial_x u\|_{L^r([0,T]\times \TT)}\,\eqsim_{\alpha,n}\, \bigl\|u^\frac{3-(\alpha+n)}{4}\partial_x w \bigr\|_{L^r([0,T]\times \TT)}\,\le\, 
	\bigl\|u^\frac{3-(\alpha+n)}{4}\bigr\|_{L^{\frac{4p}{3-(\alpha+n)}}([0,T]\times \TT)}\|\partial_x w\|_{L^4([0,T]\times \TT)}
	\]
	a space-time integral estimate on $\partial_x u$ for $r<\frac{7}{2}$ is obtained in Lemma \ref{Lemma_int_r}. Then, for example, the first term on the right-hand side of \eqref{Eq_STFE_rewritten} can be estimated via H\"older's inequality 
	\[
	\|u^{n-2}(\partial_x u)^3\|_{L^{\frac{7}{n+4}-}([0,T]\times \TT)}\,\le \,
	\|u^{n-2}\|_{L^{\frac{7}{n-2}-}([0,T]\times \TT)}\|(\partial_x u)^3\|_{L^{\frac{7}{6}-}([0,T]\times \TT)},
	\]
	where $\frac{7}{n+4}>1$ by Assumption \ref{Condition_n}. This turns out to be enough to conclude some temporal regularity of $u$ in Lemma \ref{Lemma_reg_tf_int_New} and also identify the term $u^{n-2}(\partial_x u)^3$ in the limit using Vitali's convergence theorem, as demonstrated in the proof of Lemma \ref{Lemma_Limits} \ref{Item2}.
	
	\subsection{Discussion of the result}
	The fact that the result from \cite[Section 6]{weak_trace} can be generalized  to the stochastic setting is essentially due to the use of Stratonovich noise, which is compatible with the $\alpha$-entropy estimates. This
	allows us to construct solutions to \eqref{Eq201} for non-negative initial values from the space of measures, 
	including the interesting case of the Dirac distribution.
	Moreover, only closing the $\alpha$-entropy estimate requires less spatial regularity of the noise compared to cases in which also the energy estimate is used. Indeed, Assumption \ref{Condition_sigma} essentially expresses that $B$ is a $Q$-Wiener process in $H^1(\TT)$, while the reviewed results in \cite{dareiotis2021nonnegative,fischer_gruen_2018, GessGann2020,KleinGruen22,metzger2022existence,Sauerbrey_2021}  all require that $B$ converges in $H^2(\TT^d)$.	
	
	Although it would be preferable to have solutions satisfying also the energy estimate, our result is the only one so far providing solutions to the stochastic thin-film equation, which allows for initial values without full support in the case of a non-quadratic mobility. It is a natural question, whether  the $\alpha$-entropy estimate alone is sufficient to verify qualitative statements e.g. on the propagation speed of the solution.
	
	\subsection{Notation} 
		Let $\X$ be a Banach space and $\nu$ a non-negative measure on a  measurable space $S$. Then we write $L^p(S, \mathcal{X})$, $p\in [1,\infty]$, for the Bochner space on $S$ with values in $\mathcal{X}$, equipped with the norm
	\[
	\|f\|_{L^p(S,\mathcal{X})}^p\,=\, \int_S   \|f\|_{\mathcal{X}}^p \, d\nu,  \quad p\in [1,\infty),
	\]
	and the usual modification for $p=\infty$. In the case $\mathcal{X}=\RR$, we simply write $L^p(S)$ or, if $S$ is equipped with the counting measure, $l^p(S)$. 
	
	If $[0,T]$ is an interval, we use the notation $L^p(0,T; \mathcal{X})$ for $L^p([0,T]; \X)$. Moreover, we write $C([0,T],\mathcal{X})$ for the space of continuous,  $\X$-valued functions equipped with the norm
	\[
	\|f\|_{C([0,T];\X)}\,=\, \sup_{0\le t\le T} \|f(t)\|_{\X}.
	\]
	We write $W^{\kappa, p}(0,T;\mathcal{X})$, $\kappa\in(0,1)$, $p\in[1,\infty)$, for the Sobolev-Slobodeckij space equipped with the norm 
	\[
	\|f\|_{W^{\kappa, p}(0,T;\mathcal{X})}^p\,=\, \|f\|_{L^p(0,T; \X)}^p\,+\, \int_{0}^T\int_{0}^T \frac{\|f(t)-f(s)\|_{\X}^p}{|t-s|^{1+\kappa p}}\,dt\, ds.
	\]
	For $l\in \NN$ and $p\in[1,\infty]$, $W^{l,p}(0,T;\X)$ denotes the usual Sobolev space with norm 
	\[
	\|f\|_{W^{k,p}(0,T;\X)}^p\,=\, \sum_{j=0}^{l} \|\partial_t^{l} f\|_{L^p(0,T; \X)}^p.
	\]
	
	We write $\TT$ for the flat torus, i.e. the interval $[0,1]$ with its endpoints identified. We write $C(\TT)$ and $C^l(\TT)$, $l\in \NN$, for the continuous and $l$-times continuously differentiable functions on $\TT$, respectively, equipped with the norms
	\[
	\|f\|_{C(\TT)}\,=\, \sup_{x\in \TT}|f(x)|,\quad \|f\|_{C^l(\TT)}\,=\, \sum_{j=0}^l \|\partial_x^j f\|_{C(\TT)}.
	\]
	For the smooth functions on $\TT$ we write $C^\infty(\TT)$.
	We write  $W^{l,p}(\TT)$, $l\in \NN$, $p\in [1,\infty]$, for the  Sobolev space equipped with the norm
	\[ \|f\|_{W^{l,p}(\TT)}^p\,=\, \sum_{j=0}^l \|\partial_x^j f\|_{L^p(\TT)}^p
	\]
	and for the case $p=2$ we use the notation $H^l(\TT)=W^{l,2}(\TT)$. If $p\in (1,\infty)$, we write $W^{-l,p}(\TT)$ for the dual space of $W^{l,p'}(\TT)$ under the duality pairing $\langle \cdot, \cdot \rangle$ in $L^2(\TT)$, where $p'$ is the H\"older conjugate of $p$, and equip it with the norm 
	\[
	\|f\|_{W^{-l,p}(\TT)}\,=\, \sup_{\|g\|_{W^{l,p'}(\TT)}\le 1}|\langle f,g\rangle |.
	\]
	Denoting by $	\hat{f}(k) \,=\, \langle f, e^{2\pi i k \cdot}\rangle $, $k\in \ZZ$, the $k$-th Fourier coefficient of a distribution $f$ on $\TT$, we denote by $H^\kappa(\TT)$ for $\kappa\in \RR\setminus \NN$ the Bessel-potential space consisting of all distributions with
	\[
	\|f\|_{H^\kappa(\TT)}^2\,=\,\sum_{k\in \ZZ}|\hat{f}(k)|^2(1+|2\pi k|^{2})^\kappa
	\]
	being finite. Moreover, we write $\mathcal{M}(\TT)$ for the space of Radon measures on $\TT$ and equip it with the total variation norm $\|\nu \|_{\MM(\TT)}= |\nu|(\TT)$. 	

	Lastly, if $G$ and $H$ are separable Hilbert spaces, we denote the space of Hilbert-Schmidt operators between $G$ and $H$ by $L_2(G,H)$, which carries the norm
	\[
	\|\Psi\|_{L_2(G,H)}^2\,=\, \sum_{k\in \NN} \|\Psi g_k\|_{H}^2
	\]
	for an orthonormal basis $(g_k)_{k\in \NN}$  of $G$. If $(\Omega, \mathfrak{A}, \PP)$ is a probability space, we write $\EE[\,\cdot\,]$ and $\EE[\,\cdot\,| \mathfrak{H}]$ for the expectation and conditional expection with respect to a sub-sigma-field $\mathfrak{H}\subset \mathfrak{A}$, respectively.  For two quantities $A$ and $B$, we write $A\lesssim B$, if there exists a universal constant $C$ such that $A\le CB$. If this constant depends on parameters $p_1,\dots$, we write $A\lesssim_{p_1,\dots}B$ instead. We write $A\eqsim_{p_1,\dots} B$, whenever $A\lesssim_{p_1,\dots} B$ and $B\lesssim_{p_1,\dots}A$.
	\section{Approximate solutions} 
	In the course of this article, we prove the existence of solutions to \eqref{Eq_STFE_rewritten} under Assumptions \ref{Condition_n}, \ref{Condition_sigma} and \ref{Condition_IV}. For this purpose, we use solutions to the stochastic thin-film equation with a strictly positive, and regularized initial value and spatially smooth noise as approximations. To regularize $u_0$, we
	let $(\eta_\epsilon)_{\epsilon>0}$ be a family of non negative, smooth functions $\eta_\epsilon\colon \TT\to \RR$ sufficing $\|\eta_\epsilon\|_{L^1(\TT)}=1$ and 
	\[
	\int_{\TT\setminus B_\delta(0)} |\eta_\epsilon|\, dx\,\to \, 0
	\]
	as $\epsilon\searrow 0$ for every $\delta >0$.
	Then, we define 
	\begin{equation}\label{Eq99}
		u_{0,\epsilon}\,=\, \mathbbm{1}_{\{
			\|u_0\|_{\mathcal{M}(\TT)}\, <\, \frac{1}{\epsilon}
			\}}( u_0*\eta_\epsilon)\,+\, \epsilon,
	\end{equation}
	which is strictly positive and smooth by \cite[Theorem 2.3.20]{grafakos2014classical} with 
	\begin{equation}\label{Eq101}\partial_x^l u_{0,\epsilon} \,=\,  
		\mathbbm{1}_{\{
			\|u_0\|_{\mathcal{M}(\TT)}\, <\, \frac{1}{\epsilon}
			\}}( u_0*\partial_x^l\eta_\epsilon)
	\end{equation}
	for $l\in \NN$. By the convolution inequality
	\begin{equation}\label{Eq102}
		\|\nu *\eta\|_{L^1(\TT)}\,\le \, \|\eta\|_{L^1(\TT)} \|\nu\|_{\mathcal{M}(\TT)}
	\end{equation}
	 for general $\nu\in \mathcal{M}(\TT)$, $\eta\in C^\infty(\TT)$, see \cite[Proposition 8.49]{folland_real_analysis},
	we conclude with \eqref{Eq101} that the $\mathfrak{F}_0$-measurable $u_{0,\epsilon}$ lies in $L^\infty(\Omega, H^1(\TT))$. 
	Moreover, since $\vp*(\eta_\epsilon(-\cdot ))$ converges uniformly to $ \vp$ for $\vp\in C(\TT)$ by \cite[Theorem 1.2.21]{grafakos2014classical}, we have that
	\[
	\langle 
	u_0 *\eta_\epsilon, \vp
	\rangle\,=\, \int_{\TT} \int_{\TT} \eta_\epsilon(x-y)\, du_0 (y) \vp(x)\, dx \,=\, 
	\int_{\TT} \int_{\TT} \eta_\epsilon(x-y)\vp(x) \, dx \, d u_0 (y)\,\to\, 	\langle 
	u_0 , \vp
	\rangle,
	\]
	so that $u_{0,\epsilon}$ converges almost surely to $u_0$ in the vague topology of $\MM(\TT)$ as $\epsilon\searrow 0$. Moreover, we introduce the cut-off weights
	\begin{equation}\label{Eq87}
		\lambda_{k, \epsilon}\,=\, \begin{cases}
			\lambda_k, &|k|<\frac{1}{\epsilon},\\
			0, &\text{else}.
		\end{cases}
	\end{equation}
	and define correspondingly
	\begin{equation}\label{Eq88}
		\sigma_{k,\epsilon}\,=\, \lambda_{k, \epsilon} f_k.
	\end{equation}
	Assumption \ref{Condition_n} in conjunction with  \cite[Theorem 2.2]{dareiotis2021nonnegative} implies that for each $\epsilon\in (0,1)$ there exists a solution in the sense of \cite[Definition 2.1]{dareiotis2021nonnegative} to 
	\[
	d u \,=\, 
	-\, \partial_x( u^n  \partial_x^3 u)\,dt
	\,+\, \tfrac{1}{2}\sum_{k\in \ZZ}	\partial_x (\sigma_{k,\epsilon} q'(u) \partial_x (\sigma_{k,\epsilon} q(u)) )\, dt 
	\,+	\,\sum_{k\in \ZZ} \partial_x (\sigma_{k,\epsilon} q(u)) \,d\beta^{(k)}
	\] with initial value $u_{0,\epsilon}$. To be precise we list the following consequences of Assumption  \ref{Condition_n}, which we use in the proof of Theorem \ref{thm_main}. We make this distinction to make it transparent that an extension of the admissible $n$ in \cite[Theorem 2.2]{dareiotis2021nonnegative} would also lead to a relaxation of Assumption \ref{Condition_n}.
	\begin{cons} \label{Consequences_Ass_1} Assumption \ref{Condition_n} and  \cite[Theorem 2.2]{dareiotis2021nonnegative} imply the following.
	\begin{enumerate}[label=(\roman*)]
		\item \label{Item_Consequences_1} We have $n\in (2,3)$.
		\item Let $\epsilon>0$. Then there exists a probability space $(\tilde{\Omega}_\epsilon,\tilde{\mathfrak{A}}_\epsilon , \tilde{\PP}_\epsilon)$ with a filtration $\tilde{\mathfrak{F}}_\epsilon$ satisfying the usual conditions, a family of independent $\tilde{\mathfrak{F}}_\epsilon$-Brownian motions $(\tilde{\beta}^{(k)}_\epsilon)_{k\in \ZZ}$, an $\tilde{\mathfrak{F}}_\epsilon$-adapted, weakly continuous $H^1(\TT)$-valued process $\tilde{u}_\epsilon$ and an $\tilde{\mathfrak{F}}_{\epsilon,0}$-measurable random variable $\tilde{\xi}_\epsilon$,  subject to the following properties.
		\item  $(\tilde{u}_\epsilon(0),\tilde{\xi}_\epsilon)$ has the same distribution as $\bigl(u_{0,\epsilon}, \|u_0\|_{\MM(\TT)}\bigr)$.
		\item \label{Item_Consequences_4} We have $\tilde{\PP}_\epsilon\otimes dt\otimes dx$-almost everywhere $\tilde{u}_\epsilon>0$.
		\item \label{Item_Consequences_5} It holds 
		\[
		\tilde{\EE}_\epsilon\biggl[
		\sup_{0\le t\le T} \|\tilde{u}_\epsilon(t)\|_{H^1(\TT)}^{2n}\,+\, \|\mathbbm{1}_{\{\tilde{u}_\epsilon> 0\}}q(\tilde{u}_\epsilon)\partial_x^3 \tilde{u}_\epsilon\|_{L^2([0,T]\times \TT)}^4\,+\,
		\|\tilde{u}_\epsilon\|_{L^2(0,T;H^2(\TT))}^4
		\biggr]\,<\, \infty.
		\]
		\item \label{Item_Consequences_6} For all $\vp\in C^\infty(\TT)$ and $t\in [0,T]$, it holds
		\begin{align}\begin{split}
				\label{Eq55_New}
				\langle \tilde{u}_\epsilon(t), \vp \rangle\,-\,
				\langle \tilde{u}_\epsilon(0), \vp \rangle\,=\,&
				\int_0^t\int_{\{\tilde{u}_\epsilon>0\}} m(\tilde{u}_\epsilon) \partial_x^3\tilde{u}_\epsilon \partial_x\vp \,dx\,ds
				\\&-\,\tfrac{1}{2}\sum_{k\in \ZZ}\int_0^t \langle
				\sigma_{k,\epsilon} q'(\tilde{u}_\epsilon) \partial_x (\sigma_{k,\epsilon} q(\tilde{u}_\epsilon)) ,\partial_x \vp\rangle\, ds
				\\&-\,\sum_{k\in \ZZ} \int_0^t 
				\langle
				\sigma_{k,\epsilon} q(\tilde{u}_\epsilon), \partial_x \vp
				\rangle
				\, d\tilde{\beta}^{(k)}_{\epsilon,s}.
			\end{split}
		\end{align}
	\end{enumerate}
\end{cons}

	For technical reasons, we also demanded  existence of the random variable $\tilde{\xi}_\epsilon$, which can be obtained from \cite{dareiotis2021nonnegative} by including the random variable $\|u_0\|_{\MM(\TT)}$ in the application of the stochastic compactness method. This will be important later, since we define the decomposition \eqref{Eq98}  on the new probability space, which cannot be recovered from the regularized initial value $u_{0,\epsilon}$.
	
	To simplify the notation, we assume that all martingale solutions are defined on the original filtered probability space $(\Omega, \mathfrak{A}, \mathfrak{F}, \PP)$ with respect to the given family of Brownian motions $(\beta^{(k)})_{k\in \ZZ}$, attain the original initial value $u_{0,\epsilon}$, and denote them by $u_\epsilon$. Moreover, by equidistribution, we can assume that the auxiliary random variables $\tilde{\xi}_\epsilon$ are given by $ \|u_0\|_{\MM(\TT)}$. This simplification is possible, since the forthcoming estimates only depend on the distribution of the solutions.
	
	\subsection{Application of \^Ito's formula}
	In this section, we establish estimates on the dissipation terms of the $\alpha$-entropy 
	\begin{equation}\label{Eq207}
	\int_{\TT} G_\alpha (u_\epsilon)\,dx,
	\end{equation}
	which are uniform in $\epsilon$, where
	\begin{equation}\label{Eq28}
		G_\alpha(u)\,=\,\tfrac{1}{\alpha(\alpha+1)}u^{\alpha+1}
	\end{equation}
	and 	\begin{equation}\label{Condition_alpha_New}
		-1\, < \, \alpha \, < \, 2-n.
	\end{equation} 
	These estimates are expressed conditionally on $\mathfrak{F}_0$, where we remark that the conditional expectation  is also well-defined for non-negative random variables which do not lie in $L^1(\Omega)$, for details see \cite[Chapter 5]{kallenberg1997foundations}.
		\begin{prop}\label{Thm_Ito_s_formula_New}
		Assume \eqref{Condition_alpha_New}.	
		Then for every  $\epsilon\in (0,1)$ we have 
		\begin{align}\begin{split}\label{Eq13_New}&
			\EE\biggl[
		\int_0^{T}\int_{\TT} u_\epsilon^{\alpha+n-1}
		(\partial_x^2{u_\epsilon} )^2 \,dx\,ds\,+\, 
		\int_0^{T}\int_{\TT}
		u_\epsilon^{\alpha+n-3}(\partial_x u_\epsilon)^4  \,dx\,ds\,\bigg|\, \mathfrak{F}_0
		\biggr]\\&\quad 
		 \lesssim_{\alpha, n,\Lambda,T} \, \|u_0\|_{\MM(\TT)}^{\alpha+1} + \|u_0\|_{\MM(\TT)}^{\alpha+n-1}+\epsilon^{\alpha+1}.
	\end{split}
		\end{align}
	\end{prop}
\begin{proof}
	Throughout this proof we fix an $\epsilon\in (0,1)$ and introduce the functions
	\begin{align}&
		f\,=\, - \mathbbm{1}_{\{{u}_\epsilon> 0\}}m({u_\epsilon})\partial_x^3 {u_\epsilon},\\&
		g\,=\, \tfrac{1}{2}\sum_{k\in \ZZ} 
		\sigma_{k,\epsilon} q'({u}_\epsilon) \partial_x (\sigma_{k,\epsilon} q({u}_\epsilon)),
	\end{align}
	and	the $L_2(l^2(\ZZ), L^2(\TT))$-valued process $\Psi$, defined by
	\[
	\Psi e_k\,=\, 	\partial_x (\sigma_{k, \epsilon} q(u_\epsilon)),
	\]
	where $e_k$ denotes the $k$-th unit vector in $l^2(\ZZ)$. We convince ourselves that 
	\begin{equation}\label{Eq1_New}
		f,g \in L^2(\Omega , L^2(0,T; L^2(\TT)))
	\end{equation} and  
	\begin{equation}\label{Eq2_New}\Psi \in L^2(\Omega, L^2(0,T;L_2(l^2(\ZZ), L^2(\TT)) )).
	\end{equation}
	To this end, we observe that
	\begin{align}
		\EE \bigl[\|f\|_{L^2(0,T; L^2(\TT))}^2\bigr]\,&=\, \EE\biggl[ \int_0^{T} \int_{\TT}|
		\mathbbm{1}_{\{{u}_\epsilon> 0\}}m({u_\epsilon})\partial_x^3 {u_\epsilon}|^2
		\, dx\,dt\biggr]
		\\&
		\le \,\EE \biggl[
		\int_0^{T}  \|q(u_\epsilon)\|_{L^\infty(\TT)}^2\|
		\mathbbm{1}_{\{{u}_\epsilon> 0\}}q({u_\epsilon})\partial_x^3 {u_\epsilon}\|_{L^2(\TT)}^2\,dt\biggr]
		\\&
		\le \, \EE \Bigl[\sup_{0\le t\le T} \|u_\epsilon(t)\|_{H^1(\TT)}^{2n}\Bigr]^\frac{1}{2}\EE \bigl[\|
		\mathbbm{1}_{\{{u}_\epsilon> 0\}}q({u_\epsilon})\partial_x^3 {u_\epsilon}\|_{L^2([0,T]\times \TT)}^4\bigr]^\frac{1}{2} \,<\, \infty
	\end{align}
	by Consequence \ref{Consequences_Ass_1} \ref{Item_Consequences_5}.
	As a consequence of Assumption \ref{Condition_sigma}, we have 
	\begin{align}\begin{split}\label{Consequence_sigma_New}
			&
			\sum_{k\in \ZZ} \|\sigma_k\|_{C^1(\TT)}^2\, =\, 
			\sum_{k\in \ZZ} (\lambda_k(1+|k|))^2\,<\, \infty,\\&
			\sum_{k\in \ZZ} \|\sigma_k\|_{C^2(\TT)}\|\sigma_k\|_{C(\TT)}\, =\, 
			\sum_{k\in \ZZ} \lambda_k^2(1+|k|+|k|^2) \,<\, \infty
		\end{split}
	\end{align}
	and hence 
	\begin{align}
		\EE\bigl[\|g\|_{L^2(0,T; L^2(\TT))}^2\bigr]\,&\le\, \EE\biggl[\int_0^{T} \int_{\TT}\biggl|\sum_{k\in \ZZ} 
		\sigma_{k,\epsilon} q'({u}_\epsilon) \partial_x (\sigma_{k,\epsilon} q({u}_\epsilon)) \biggr|^2
		\, dx\,dt\biggr]
		\\&\lesssim \, 
		\EE\biggl[
		\int_0^{T} \biggl(\sum_{k\in \ZZ} 
		\|\sigma_{k,\epsilon}\|_{L^\infty(\TT)}^2 \biggr)^2 \int_{\TT} ((q'({u}_\epsilon))^2 \partial_x u_\epsilon)^2 \,dx\,dt \biggr]
		\\&\quad\,
		+\,  \EE\biggl[\int_0^{T} 	\biggl(\sum_{k\in \ZZ} 
		\|\sigma_{k,\epsilon}\|_{L^\infty(\TT)}
		\|\partial_x \sigma_{k,\epsilon}\|_{L^\infty(\TT)}  \biggr)^2 \int_{\TT}
		(q'(u_\epsilon)q(u_\epsilon))^2
		\, dx\,dt\biggr]
		\\&\lesssim_{ n, \Lambda}\,
		\EE\biggl[\int_0^{T} \int_{\TT}
		u_\epsilon^{2n-4} (\partial_xu_\epsilon)^2
		\,dx\,dt\,+\, 	\int_0^{T} \int_{\TT}
		u_\epsilon^{2n-2}
		\,dx\,dt\biggr]
		\\&\le \,
		\EE\biggl[
		\int_0^{T}
		\|u_\epsilon\|_{L^\infty(\TT)}^{2n-4} \|\partial_xu_\epsilon\|_{L^2(\TT)}^2
		\,dt\,+\, 	\int_0^{T}
		\|u_\epsilon\|_{L^\infty(\TT)}^{2n-2}\, dt\biggr] 
		\\& \lesssim_T\,
		\EE\Bigl[
		\sup_{0\le t\le T} 
		\|u_\epsilon(t)\|_{H^1(\TT)}^{2n-2}\Bigr]\,< \, \infty,
	\end{align}
	where we have used again Consequence \ref{Consequences_Ass_1} \ref{Item_Consequences_5}.
	By the same arguments, we also have 
	\begin{align}\begin{split}
			\label{Eq8_New}
			\EE\bigl[\|\Psi\|_{L^2(0,T; L_2(l^2(\ZZ) , L^2(\TT)))}^2\bigr]\,&=\,\EE\biggl[
			\int_0^{T}  
			\sum_{k\in \ZZ} 
			\|\partial_x (\sigma_{k,\epsilon} q(u_\epsilon))\|_{L^2(\TT)}^2 
			\, dt \biggr]\\&\lesssim\, \EE\biggl[	\int_0^{T}  
			\sum_{k\in \ZZ} 
			\int_{\TT} (\partial_x \sigma_{k,\epsilon})^2 m(u_\epsilon)\, dx 
			\, dt\biggr]\\&\quad\,+\, 	\EE\biggl[\int_0^{T}  
			\sum_{k\in \ZZ} 
			\int_{\TT} \sigma_{k,\epsilon}^2 (q'(u_\epsilon)\partial_x u_\epsilon)^2\, dx 
			\, dt\biggr]\\&\lesssim_n \, \EE\biggl[
			\int_0^{T}  
			\sum_{k\in \ZZ} \|\partial_x \sigma_{k,\epsilon}\|_{L^\infty(\TT)}^2 \|u_\epsilon\|_{L^\infty(\TT)}^{n}
			\, dt\biggr] \\&\quad \, +\,\EE\biggl[
			\int_0^{T}  
			\sum_{k\in \ZZ} \|\sigma_{k,\epsilon}\|_{L^\infty(\TT)}^2 \|u_\epsilon\|_{L^\infty(\TT)}^{n-2}\|\partial_x u_\epsilon\|_{L^2(\TT)}^2
			\, dt \biggr] \\&\lesssim_{\Lambda, T}\, \EE \Bigl[\sup_{0\le t\le T} \|u_\epsilon(t) \|_{H^1(\TT)}^n\Bigr] \,< \, \infty,
		\end{split}
	\end{align}
	so that \eqref{Eq1_New} and \eqref{Eq2_New}  indeed hold.	Denoting the cylindrical Wiener process $e_k \mapsto \beta^{(k)}$ in $l^2(\ZZ)$  by $\beta$, we see that 
	\begin{equation}\label{Eq4_New}
		u_\epsilon(t)\,=\, u_{0,\epsilon}\,+\, \int_0^t \partial_x ( f+g )\, ds\,+\, \int_0^t \Psi \, d\beta_s
	\end{equation}
suffices all the assumptions of It\^o's formula \cite[Proposition A.1]{DHV_16} on the involved processes by Consequence \ref{Consequences_Ass_1} \ref{Item_Consequences_5} and the integrability statements \eqref{Eq1_New} and \eqref{Eq2_New}. We would like to apply  \cite[Proposition A.1]{DHV_16} to calculate the It\^o expansion of the functional \eqref{Eq207}. Since, however, the function \eqref{Eq28} is not twice continuously differentiable, we instead use the shifted version of  \eqref{Eq28} as introduced in \cite[Proposition 2.2]{Passo98ona}. Specifically, we let $G_{\alpha,\delta}\in C^\infty(\RR)$ such that
		\begin{equation}\label{Eq17_New}
		G_{\alpha, \delta}(u)\,=\,\begin{cases}
			G_\alpha(u+\delta),& u\ge 0,\\
			0, &u\le -1,
		\end{cases}
	\end{equation}
	for $\delta\in (0,1)$. In particular, $G_{\alpha,\delta}$ has bounded second derivative and	hence evaluating It\^o's formula \cite[Proposition A.1]{DHV_16} at time $T$ yields 
	\begin{align}\begin{split}\label{Eq7_New}
			&
			\int_{\TT}
			G_{\alpha,\delta} (u_\epsilon (T ))\, dx \,-\, \int_{\TT}
			G_{\alpha,\delta} (u_{0,\epsilon})\, dx\\&\quad=\,  \int_0^{T}\int_{\{u_{\epsilon}>0\}}
			G_{\alpha,\delta}''(u_\epsilon)\partial_x u_\epsilon
			m({u}_\epsilon) \partial_x^3{u_\epsilon}  \,dx\,dt\\&\qquad\,-\,\tfrac{1}{2}\sum_{k\in \ZZ}
			\int_0^{T} \int_{\TT}
			G_{\alpha, \delta}''(u_\epsilon)
			\partial_x u_\epsilon 
			\sigma_{k,\epsilon} q'({u}_\epsilon) \partial_x (\sigma_{k,\epsilon} q({u}_\epsilon))\, dx\, dt
			\\& \qquad\,+\, 
			\int_0^{T}\langle
			G_{\alpha,\delta}'(u_\epsilon)\Psi\, d\beta_t
			,\mathbbm{1}_{\TT}
			\rangle 
			\,
			+\, \tfrac{1}{2}\sum_{k\in \ZZ} \int_0^{T} \int_{\TT}
			G_{\alpha,\delta}''(u_\epsilon) (\partial_x (\sigma_{k, \epsilon} q(u_\epsilon)))^2
			\,dx\,dt.
		\end{split}
	\end{align}
	To simplify the second and fourth term appearing on the right-hand side, we observe that
\begin{align}\begin{split}\label{Eq259}
		\mathcal{I}_{2,4}(k,t)\,&\coloneqq\,
	-\,\int_{\TT}
	G_{\alpha, \delta}''(u_\epsilon)
	\partial_x u_\epsilon 
	\sigma_{k,\epsilon} q'({u}_\epsilon) \partial_x (\sigma_{k,\epsilon} q({u}_\epsilon))\, dx
	\,+\,\int_{\TT}
	G_{\alpha,\delta}''(u_\epsilon) (\partial_x (\sigma_{k, \epsilon} q(u_\epsilon)))^2
	\,dx\\&=\,
	\int_{\TT}
	G_{\alpha, \delta}''(u_\epsilon)
	\partial_x\sigma_{k,\epsilon} q({u}_\epsilon) \partial_x (\sigma_{k,\epsilon} q({u}_\epsilon))\, dx
	\\&=\,
	\int_{\TT}
	G_{\alpha, \delta}''(u_\epsilon)m({u}_\epsilon)
	(\partial_x\sigma_{k,\epsilon})^2 \, dx\,+\,
	\int_{\TT}
	G_{\alpha, \delta}''(u_\epsilon)  q({u}_\epsilon) q'({u}_\epsilon)\partial_x u_\epsilon\sigma_{k,\epsilon}
	\partial_x\sigma_{k,\epsilon} \, dx.
	\end{split}
\end{align}
		We define 
\[
\zeta_{\alpha, \delta} (u)\,=\, \int_0^u G_{\alpha, \delta}''(x)  q(x) q'(x)\, dx\,=\, \frac{n}{2}
\int_0^u (x+\delta)^{\alpha-1}  x^{n-1}\, dx
\]
for $u\ge 0$,
so that
\begin{equation}\label{Eq256}
\int_{\TT}
G_{\alpha, \delta}''(u_\epsilon)  q({u}_\epsilon) q'({u}_\epsilon)\partial_x u_\epsilon\sigma_{k,\epsilon}
\partial_x\sigma_{k,\epsilon} \, dx\,=\, -\,
\int_{\TT}
\zeta_{\alpha, \delta}(u_\epsilon) \partial_x(\sigma_{k,\epsilon}
\partial_x\sigma_{k,\epsilon}) \, dx.
\end{equation}
By the definition of $\zeta_{\alpha, \delta}$  we have 
\[
\zeta_{\alpha, \delta} (u)\,\le \, 
\frac{n}{2}
\int_0^u   (x+\delta)^{\alpha+n-2}\, dx,
\] and consequently
\begin{align}\label{Eq257}
	|\zeta_{\alpha, \delta}(u)|\,\lesssim_{\alpha,n} \, (u+\delta)^{\alpha+n-1}.\end{align}
Moreover, it holds
\begin{align}\label{Eq258}
	|G_{\alpha, \delta}''(u)m(u)|\,\le\,
	(u+\delta)^{\alpha+n-1}
\end{align}
for $u\ge 0$ by \eqref{Eq17_New} and \begin{align}\label{Eq25_New}
	\|u_\epsilon(t)\|_{L^1(\TT)}\,=\, \|u_{0,\epsilon}\|_{L^1(\TT)}\, \le \, \|u_0\|_{\MM(\TT)}\|\eta_\epsilon\|_{L^1(\TT)}+ \epsilon\,=\, 
	\|u_0\|_{\MM(\TT)}+\epsilon
\end{align}
by  the divergence form of \eqref{Eq55_New}, non-negativity of $u_\epsilon$ and \eqref{Eq102}.
Inserting \eqref{Eq256}, \eqref{Eq257} and \eqref{Eq258} in \eqref{Eq259}
and using  \eqref{Condition_alpha_New}, \eqref{Consequence_sigma_New}
and 
\eqref{Eq25_New}, we conclude that 
\begin{align}\begin{split}\label{Eq11_New}&
		\biggl|\tfrac{1}{2}\sum_{k\in \ZZ}	\int_0^{T}\mathcal{I}_{2,4}(k,t)
		\,dt\biggr|\\&\quad 
		\lesssim_{\alpha, n}\, \sum_{k\in \ZZ}
		\int_0^T\int_{\TT} ( u_\epsilon+\delta)^{\alpha+n-1}( (\partial_x \sigma_{k,\epsilon} )^2 +|\partial_x (\sigma_{k,\epsilon} \partial_x\sigma_{k,\epsilon} )|  )\, dx\, dt
		\\&\quad
		\lesssim_{\Lambda, T}\,\Bigl(\sup_{0\le t\le T} \|u_\epsilon\|_{L^1(\TT)}\,+\,  \delta\Bigr)^{\alpha+n-1}\,\le \, \bigl( \|u_0\|_{\MM(\TT)}+\epsilon+\delta\bigr)^{\alpha+n-1}.
	\end{split}
\end{align}
Next, we estimate the first term on the right-hand side of \eqref{Eq7_New}. The calculations are the same as in the proof of \cite[Proposition 2.2]{Passo98ona} and only contained for convenience of the reader. Namely, using integration by parts we can rewrite 
\begin{align}\begin{split}\label{Eq221}
		 &\int_{\{u_{\epsilon}>0\}}
	G_{\alpha,\delta}''(u_\epsilon)\partial_x u_\epsilon
	m({u}_\epsilon) \partial_x^3{u_\epsilon}  \,dx
	\\&\quad=\, \int_{\{u_{\epsilon}>0\}}
	(u_\epsilon+\delta)^{\alpha+n-1} \Bigl(\tfrac{m(u_\epsilon)}{m(u_\epsilon+\delta)}\Bigr)\partial_x u_\epsilon \partial_x^3{u_\epsilon}  \,dx
	\\&\quad = \,-\, \int_{\TT} (u_\epsilon+\delta)^{\alpha-1} m(u_\epsilon) (\partial_x^2 u_\epsilon)^2\, dx
	\\&
	\qquad \, -\, (\alpha+n-1) \int_{\TT} \Bigl(\tfrac{m(u_\epsilon)}{m(u_\epsilon+\delta)}\Bigr)(u_\epsilon+\delta)^{\alpha+n-2} (\partial_x u_\epsilon)^2 \partial_x^2 u_\epsilon\, dx
	\\&\qquad \, -\, \int_{\TT} \Bigl(
	\tfrac{m'(u_\epsilon)}{m(u_\epsilon+\delta)}  - \tfrac{m(u_\epsilon)m'(u_\epsilon+\delta)}{(m(u_\epsilon+\delta))^2}
	\Bigr)(u_\epsilon+\delta)^{\alpha+n-1}(\partial_x u_\epsilon)^2 \partial_x^2 u_\epsilon\, dx
	\\&\quad = \, \mathcal{J}_1\,+\, \mathcal{J}_2\,+\, \mathcal{R}_1.
	\end{split}
\end{align}
$\PP\otimes dt$-almost everywhere. 
Integrating by parts and employing Young's inequality, we retrieve 
\begin{align}\begin{split}
		\label{Eq220}
		\mathcal{J}_2\,&=\, -\, (\alpha+n-1) \int_{\TT}(u_\epsilon+\delta)^{\alpha+n-2} (\partial_x u_\epsilon)^2 \partial_x^2 u_\epsilon\, dx \\&\quad \,-\, (\alpha+n-1) \int_{\TT} (u_\epsilon+\delta)^\frac{\alpha+n-3}{2} (\partial_x u_\epsilon)^2 \Bigl(\tfrac{m(u_\epsilon)}{m(u_\epsilon+\delta)}-1\Bigr)  (u_\epsilon+\delta)^\frac{\alpha+n-1}{2}\partial_x^2 u_\epsilon\, dx
		\\&\le \, \Bigl(\tfrac{(\alpha+n-1)(\alpha+n-2)}{3}+\kappa_1\Bigr)\int_{\TT}
		(u_\epsilon+\delta)^{\alpha+n-3}(\partial_x u_\epsilon)^4
		\, dx
		\\&\quad \,+\,\tfrac{(\alpha+n-1)^2}{\kappa_1}\int_{\TT} \Bigl(\tfrac{m(u_\epsilon)}{m(u_\epsilon+\delta)}-1\Bigr)^2 (u_\epsilon+\delta)^{\alpha+n-1}(\partial_x^2 u_\epsilon)^2\, dx
	\end{split}
\end{align}
for every $\kappa_1 >0$.
Moreover, by inserting $m(u)=u^n$, we obtain the estimate
\begin{equation}
	\Bigl|
	\tfrac{m'(u)}{m(u+\delta)}  - \tfrac{m(u)m'(u+\delta)}{(m(u+\delta))^2}
	\Bigr|\,=\, n\bigl(\tfrac{u}{u+\delta}\bigr)^{n-1} \tfrac{ \delta}{(u+\delta)^2}\,\le  \, \tfrac{ n \delta}{(u+\delta)^2}
\end{equation}
for $u\ge 0$. Hence for each $\kappa_2>0$, we have 
\begin{align}\begin{split}
		\label{Eq219}
	| \mathcal{R}_1|\,&\le\, \int_{\TT}(u_\epsilon+\delta)^\frac{\alpha+n-3}{2}(\partial_x u_\epsilon)^2 \bigl(\tfrac{ n \delta}{u_\epsilon+\delta}\bigr)(u_\epsilon+\delta)^\frac{\alpha+n-1}{2}  |\partial_x^2 u_\epsilon|\, dx
	\\&\le \,
	\kappa_2\int_{\TT} (u_\epsilon+\delta)^{\alpha+n-3}(\partial_xu_\epsilon)^4\, dx \,+\, 
	\tfrac{1}{\kappa_2}\int_{\TT} \bigl(\tfrac{ n \delta}{u_\epsilon+\delta}\bigr)^2(u_\epsilon+\delta)^{\alpha+n-1} (\partial_{x}^2u_\epsilon)^2\, dx	,
	\end{split}
\end{align}
again by Young's inequality.  Choosing 
\[
\kappa_1\,=\, \kappa_2 \,=\, \tfrac{-(\alpha+n-1)(\alpha+n-2)}{12},
\]
which is positive by \eqref{Condition_alpha_New},
and inserting \eqref{Eq219} and \eqref{Eq220} in \eqref{Eq221} yields
\begin{align}
	 &\int_{\{u_{\epsilon}>0\}}
	G_{\alpha,\delta}''(u_\epsilon)\partial_x u_\epsilon
	m({u}_\epsilon) \partial_x^3{u_\epsilon}  \,dx\,dt
	\\&\quad\le \, -\, \int_{\TT} (u_\epsilon+\delta)^{\alpha-1} m(u_\epsilon) (\partial_x^2 u_\epsilon)^2\, dx
	\\&\qquad \,+\,
	\tfrac{(\alpha+n-1)(\alpha+n-2)}{6}\int_{\TT}
	(u_\epsilon+\delta)^{\alpha+n-3}(\partial_x u_\epsilon)^4
	\, dx
	\\&\qquad \,-\,\tfrac{12(\alpha+n-1)}{\alpha+n-2}\int_{\TT} \Bigl(\tfrac{m(u_\epsilon)}{m(u_\epsilon+\delta)}-1\Bigr)^2 (u_\epsilon+\delta)^{\alpha+n-1}(\partial_x^2 u_\epsilon)^2\, dx
	\\&\qquad \,-\,\tfrac{12n^2}{(\alpha+n-1)(\alpha+n-2)}\int_{\TT}
	\bigl(\tfrac{ \delta}{u_\epsilon+\delta}\bigr)^2(u_\epsilon+\delta)^{\alpha+n-1} (\partial_{x}^2u_\epsilon)^2\, dx.
\end{align}
By using the preceding estimate and \eqref{Eq11_New} in \eqref{Eq7_New}, we obtain moreover
\begin{align}\begin{split}\label{Eq218_New}
			&
			\int_{\TT}
			G_{\alpha,\delta} (u_\epsilon (T ))\, dx \,-\, \int_{\TT}
			G_{\alpha,\delta} (u_{0,\epsilon})\, dx \,+\, 
			\int_0^{T}\int_{\TT} (u_\epsilon+\delta)^{\alpha-1} m(u_\epsilon) (\partial_x^2 u_\epsilon)^2\, dx\,dt \\&-\,\tfrac{	(\alpha+n-1)(\alpha+n-2)}{6}
			\int_0^{T}\int_{\TT} (u_\epsilon+\delta)^{\alpha+n-3}(\partial_x u_\epsilon)^4  \,dx\,dt\,-\, 
			\int_0^{T}\langle
			G_{\alpha,\delta}'(u_\epsilon)\Psi\, d\beta_t
			,\mathbbm{1}_{\TT}
			\rangle 
			\\&\qquad\lesssim_{\alpha, n,\Lambda,T} \,\bigl( \|u_0\|_{\MM(\TT)}+\epsilon+\delta\bigr)^{\alpha+n-1} \\&\qquad\quad\,+\,\biggl| \int_0^T\int_{\TT} \Bigl(\tfrac{m(u_\epsilon)}{m(u_\epsilon+\delta)}-1\Bigr)^2 (u_\epsilon+\delta)^{\alpha+n-1}(\partial_x^2 u_\epsilon)^2\, dx\,dt\biggr|
			\\&\qquad\quad\,
			+\,	 \biggl|
			\int_0^T \int_{\TT}
			\bigl(\tfrac{ \delta}{u_\epsilon+\delta}\bigr)^2(u_\epsilon+\delta)^{\alpha+n-1} (\partial_{x}^2u_\epsilon)^2\, dx\,dt\biggr|.
		\end{split}
\end{align}
We have, by definition \eqref{Eq17_New},
\begin{equation}\label{Eq216_New}-\, \int_{\TT}
	G_{\alpha,\delta} (u_{0,\epsilon})\, dx\,\ge\, 0.\end{equation} 
Moreover, it holds $
	|G_{\alpha,\delta} (u)|  \,\lesssim_\alpha\,
	(u+\delta)^{\alpha+1}$
for $u\ge 0$ by \eqref{Condition_alpha_New} and thus 
\begin{align}\label{Eq217_New}
	\biggl|\int_{\TT}
	G_{\alpha,\delta} (u_{\epsilon}(T))\, dx\biggr|\,\lesssim_\alpha\, 
	\bigl( \|u_\epsilon(T)\|_{L^1(\TT)}\,+\,\delta\bigr)^{\alpha+1}
	 \,\le \,
	 \bigl( \|u_0\|_{\MM(\TT)}+\epsilon+\delta\bigr)^{\alpha+1},
\end{align}
by further invoking \eqref{Eq25_New}. Inserting \eqref{Eq216_New} and \eqref{Eq217_New} in \eqref{Eq218_New} yields
\begin{align}\begin{split}\label{Eq12_New}
		&\int_0^{T}\int_{\TT} (u_\epsilon+\delta)^{\alpha-1} m(u_\epsilon) (\partial_x^2 u_\epsilon)^2\, dx\,dt\\&-\,\tfrac{	(\alpha+n-1)(\alpha+n-2)}{6}
		\int_0^{T}\int_{\TT} (u_\epsilon+\delta)^{\alpha+n-3}(\partial_x u_\epsilon)^4  \,dx\,dt\,-\, 
		\int_0^{T}\langle
		G_{\alpha,\delta}'(u_\epsilon)\Psi_l \,d\beta_t
		,\mathbbm{1}_{\TT}
		\rangle 
		\\&\qquad\lesssim_{\alpha, n,\Lambda,T} \,\bigl( \|u_0\|_{\MM(\TT)}+\epsilon+\delta\bigr)^{\alpha+1}\,+\,\bigl( \|u_0\|_{\MM(\TT)}+\epsilon+\delta\bigr)^{\alpha+n-1} \\&\qquad\quad\,+\,\biggl| \int_0^T\int_{\TT} \Bigl(\tfrac{m(u_\epsilon)}{m(u_\epsilon+\delta)}-1\Bigr)^2 (u_\epsilon+\delta)^{\alpha+n-1}(\partial_x^2 u_\epsilon)^2\, dx\,dt\biggr|\\&\qquad\quad\,+\, 
		 \biggl|
		\int_0^T \int_{\TT}
		\bigl(\tfrac{ \delta}{u_\epsilon+\delta}\bigr)^2(u_\epsilon+\delta)^{\alpha+n-1} (\partial_{x}^2u_\epsilon)^2\, dx\,dt\biggr|.
	\end{split}
\end{align}
We point out that the prefactor
\[\tfrac{	(\alpha+n-1)(\alpha+n-2)}{6}
\] 
of the second term on the left-hand side of \eqref{Eq12_New}
is negative by \eqref{Condition_alpha_New}. By \eqref{Condition_alpha_New} and \eqref{Eq17_New} the function $G_{\alpha, \delta}'$ is bounded, such that
\begin{align}
	\EE\bigl[\|
	\langle
	G_{\alpha,\delta}'(u_\epsilon)\Psi \cdot 
	,\mathbbm{1}_{\TT}
	\rangle  
	\|_{L^2(0,T; L_2(l^2(\ZZ), \RR))}^2\bigr]\,&=\, \EE\biggl[
	\int_0^{T}  
	\sum_{k\in \ZZ} \biggl(\int_{\TT} G_{\alpha,\delta}'(u_\epsilon)
	\partial_x (\sigma_{k,\epsilon} q(u_\epsilon))\, dx\biggr)^2
	\, dt\biggr]\\&\lesssim_{\alpha,\delta}\,\EE\biggl[
	\int_0^{T}  
	\sum_{k\in \ZZ}\left\|
	\partial_x (\sigma_{k,\epsilon} q(u_\epsilon))\, \right\|_{L^2(\TT)}^2
	\, dt\biggr]
	,
\end{align}
which is finite by \eqref{Eq8_New}
and hence the stochastic integral in \eqref{Eq12_New} is a square integrable martingale. In order to estimate the conditional expectation on the left-hand side of \eqref{Eq13_New}, we let $A\in \mathfrak{F}_0$. Multiplying both sides in \eqref{Eq12_New} with $\mathbbm{1}_{A}$ and taking the expectation, we conclude that
\begin{align}\begin{split}\label{Eq222}
		&
	\EE\biggl[\mathbbm{1}_{A}\int_0^{T}\int_{\TT} (u_\epsilon+\delta)^{\alpha-1} m(u_\epsilon) (\partial_x^2 u_\epsilon)^2\, dx\,dt 	\biggr]
	\,
	+\, 	\EE\biggl[\mathbbm{1}_{A}\int_0^{T}\int_{\TT} (u_\epsilon+\delta)^{\alpha+n-3}(\partial_x u_\epsilon)^4  \,dx\,dt
	\biggr]
	\\&\qquad\lesssim_{\alpha, n,\Lambda,T} \, \EE\bigl[\mathbbm{1}_{A} \bigl(\bigl( \|u_0\|_{\MM(\TT)}+\epsilon+\delta\bigr)^{\alpha+1}\,+\,\bigl( \|u_0\|_{\MM(\TT)}+\epsilon+\delta\bigr)^{\alpha+n-1}\bigr)\bigr]\\&\qquad\quad \,+\,
	\EE\biggl[\biggl| \int_0^T\int_{\TT} \Bigl(\tfrac{m(u_\epsilon)}{m(u_\epsilon+\delta)}-1\Bigr)^2 (u_\epsilon+\delta)^{\alpha+n-1}(\partial_x^2 u_\epsilon)^2\, dx\,dt\biggr|\biggr]\\&\qquad\quad \, +\,
	\EE\biggl[
	\biggl|
	\int_0^T \int_{\TT}
	\bigl(\tfrac{ \delta}{u_\epsilon+\delta}\bigr)^2(u_\epsilon+\delta)^{\alpha+n-1} (\partial_{x}^2u_\epsilon)^2\, dx\,dt\biggr|
	\biggr] \,=\, \mathcal{E}_1 \,+\, \mathcal{R}_2 \,+\, \mathcal{R}_3.
	\end{split}
\end{align}
Due to Fatou's lemma and Consequence \ref{Consequences_Ass_1} \ref{Item_Consequences_4}, we can deduce that
		\begin{align}\begin{split}\label{Eq234}
						&
						\EE\biggl[\mathbbm{1}_{A}
						\int_0^{T}\int_{\TT} u_\epsilon^{\alpha+n-1}
						(\partial_x^2{u_\epsilon} )^2 \,dx\,ds	\biggr]\,
						+\, 	\EE\biggl[\mathbbm{1}_{A}
						\int_0^{T}\int_{\TT}
						u_\epsilon^{\alpha+n-3}(\partial_x u_\epsilon)^4  \,dx\,ds
						\biggr]
						\\&\quad \lesssim_{\alpha, n,\Lambda,T} \,\EE\bigl[\mathbbm{1}_{A} \bigl(\|u_0\|_{\MM(\TT)}^{\alpha+1} + \|u_0\|_{\MM(\TT)}^{\alpha+n-1}+\epsilon^{\alpha+1}\bigr)\bigr]
					\end{split}
			\end{align} by letting $\delta\searrow 0$ in \eqref{Eq222}, if we can argue that $ \mathcal{R}_2 + \mathcal{R}_3\to 0$ as $\delta\searrow 0$. To this end, we observe first that
\begin{align}& (u_\epsilon+\delta)^{\alpha+n-1}(\partial_x^2u)^2\,\le \, 
	(1+|u_\epsilon|)(\partial_x^2u)^2,
	\\&
	\Bigl(\tfrac{m(u_\epsilon)}{m(u_\epsilon+\delta)}-1\Bigr)^2\,\le \, 1,
	\\&
	\bigl(\tfrac{ \delta}{u_\epsilon+\delta}\bigr)^2\, \le \, 1,
\end{align}
by \eqref{Condition_alpha_New}. Moreover, 
\begin{align}&	\Bigl(\tfrac{m(u_\epsilon)}{m(u_\epsilon+\delta)}-1\Bigr)^2\,\to \, 0,
	\\&
	\bigl(\tfrac{ \delta}{u_\epsilon+\delta}\bigr)^2\,\to \, 0,
\end{align}
$ \PP\otimes dt \otimes dx$-almost everywhere as $\delta\searrow 0$, by Conseuqence \ref{Consequences_Ass_1} \ref{Item_Consequences_4}. Due to Consequence \ref{Consequences_Ass_1} \ref{Item_Consequences_5},
we have that
\begin{align}
	\EE\biggl[
	\int_0^T\int_{\TT} 	(1+|u_\epsilon|)(\partial_x^2u)^2
	\,dx\,dt
	\biggr]\,&\le \, \EE\biggl[
	\int_0^T	(1+\|u_\epsilon\|_{H^1(\TT)})\|\partial_x^2u\|_{L^2(\TT)}^2\, dt\biggr]\\&\le \, \EE \Bigl[\Bigl(1+ \sup_{0\le t\le T} \|u_\epsilon(t) \|_{H^1(\TT)}\Bigr) \|u_\epsilon\|_{L^2(0,T;H^2(\TT))}^2 \Bigr]
	\\&
	\le \, \Bigl(1\,+\, \EE\Bigl[
	\sup_{0\le t\le T} \|u_\epsilon(t) \|_{H^1(\TT)}^2
	\Bigr]^\frac{1}{2}\Bigr) \EE\bigl[
	\|u_\epsilon\|_{L^2(0,T;H^2(\TT))}^4
	\bigr]^\frac{1}{2}
	,
\end{align}
such that $ \mathcal{R}_2 + \mathcal{R}_3\to 0$ by the dominated convergence theorem. Consequently, \eqref{Eq234} holds true, and \eqref{Eq13_New} follows, since $A\in \mathfrak{F}_0$ was arbitrary.
\end{proof}
	\subsection{Spatial Regularity}\label{SS_spatial}
	In this section, we proceed as explained in Section \ref{SubSec_Proof} and use the Gagliardo–Nirenberg interpolation inequality in conjunction with the $\alpha$-entropy estimate \eqref{Eq13_New} and conservation of mass \eqref{Eq25_New} to obtain estimates on $u_\epsilon$ in suitable Lebesgue and Sobolev norms.
	\begin{lemma}\label{lemma_int_p}
		Let $p\in (n+4,7)$, then
		\begin{equation}\label{Eq38_New}
			\EE\bigl[\|u_\epsilon\|_{L^p([0,T]\times\TT)}^p\,\big|\, \mathfrak{F}_0 \bigr]\,
			\lesssim_{ n,p,\Lambda, T}\,
			(\|u_0\|_{\MM(\TT)}+\epsilon)^4  \bigl(\|u_0\|_{\MM(\TT)}^{p-n-4}+ \|u_0\|_{\MM(\TT)}^{p-4}+\epsilon^{p-n-4}\bigr).
		\end{equation}
	\end{lemma}
	\begin{proof}We choose $\alpha$ in accordance  with 
		\begin{equation}\label{Eq223}p\,=\,\alpha+n+5
		\end{equation} so that in particular \eqref{Condition_alpha_New} is satisfied and thus Proposition \ref{Thm_Ito_s_formula_New} applies. We define the random function $w_\epsilon\,=\, u_\epsilon^\frac{\alpha+n+1}{4} $
		and claim that $w_\epsilon\in W^{1,4}(\TT)$, $\PP\otimes dt$-almost everywhere, and that the chain rule holds for it. To verify this, we observe that $w_{\epsilon, \kappa}\,=\, (u_\epsilon+\kappa)^\frac{\alpha+n+1}{4} $ has the weak derivative $\frac{\alpha+n+1}{4}(u_\epsilon+\kappa)^\frac{\alpha+n-3}{4} \partial_x u_\epsilon $ by the chain rule \cite[Corollary 8.11]{brezis2010functional} for each $\kappa >0$. Hence,
		\begin{equation}\label{Eq103}
			\bigl(\tfrac{4}{\alpha+n+1}\bigr)^4\|\partial_x w_{\epsilon,\kappa}\|_{L^4(\TT)}^4\,=\, \int_{\TT} (u_\epsilon+\kappa)^{\alpha+n-3}
			(\partial_x u_\epsilon)^4
			\, dx \,\le \, \int_{\TT} u_\epsilon^{\alpha+n-3}
			(\partial_x u_\epsilon)^4
			\, dx,
		\end{equation}
		which is $\PP\otimes dt$-almost everywhere finite by \eqref{Eq13_New} so that taking $\kappa\searrow 0$, $w_{\epsilon,\kappa}$ admits a subsequence converging weakly in $W^{1,4}(\TT)$. This limit coincides with $w_\epsilon$, since $w_{\epsilon, \kappa}\to w_\epsilon$ almost everywhere. Moreover, using the weak convergence $\partial_x w_{\epsilon,\kappa}\rightharpoonup \partial_x w_\epsilon$ in $L^4(\TT)$ and the dominated convergence theorem, we conclude that 
		\begin{equation}\label{Eq34}
			\langle \partial_x w_\epsilon, \vp\rangle \,\leftarrow \, 
			\langle \partial_x w_{\epsilon,\kappa}, \vp\rangle\,=\, 
			\tfrac{\alpha+n+1}{4}\bigl\langle (u_\epsilon+\kappa)^\frac{\alpha+n-3}{4} \partial_x u_\epsilon, \vp\bigr\rangle\,\to \, 
			\tfrac{\alpha+n+1}{4}\bigl\langle u_\epsilon^\frac{\alpha+n-3}{4} \partial_x u_\epsilon, \vp\bigr\rangle,
		\end{equation}
	for a subsequence $\kappa \searrow 0$ and every $\vp\in C^\infty(\TT)$, and therefore the chain rule applies to $w_\epsilon$, too. 
		
	By the Gagliardo-Nirenberg interpolation inequality \cite[Eq. (42), p.233]{brezis2010functional} it holds
		\begin{equation}\label{Eq30}
			\| w_\epsilon\|_{L^r(\TT)}\,\lesssim_{\alpha,n} \, \| w_\epsilon\|_{W^{1,4}(\TT)}^\nu  \| w_\epsilon\|_{L^\frac{4}{\alpha+n+1}(\TT)}^{1-\nu} 
		\end{equation}
		for
		\[
		r\,=\, \tfrac{4(\alpha+n+4)+4}{\alpha+n+1},\;\;
		\nu \,=\,\tfrac{\alpha+n+1 -\frac{4}{r}}{\alpha+n+4}.
		\]
		Moreover, by the Poincar\'e-Wirtinger inequality
		\[
		\biggl\| w_\epsilon-\int_{\TT}w_\epsilon\, dx\biggr\|_{L^4(\TT)}\,\lesssim \, \| \partial_xw_\epsilon\|_{L^4(\TT)},
		\]
		we conclude that
		\[
		\| w_\epsilon\|_{W^{1,4}(\TT)}\,\lesssim \, \|w_\epsilon\|_{L^1(\TT)}\,+\,\| \partial_xw_\epsilon\|_{L^4(\TT)}.
		\]
		Inserting this in \eqref{Eq30} and using that $L^\frac{4}{\alpha+n+1}(\TT)\hookrightarrow L^1(\TT)$ due to \eqref{Condition_alpha_New} yields 
		\begin{equation}\label{Eq31}
			\| w_\epsilon\|_{L^r(\TT)}\,\lesssim_{\alpha,n} \, \| \partial_x w_\epsilon\|_{L^4(\TT)}^\nu  \| w_\epsilon\|_{L^\frac{4}{\alpha+n+1}(\TT)}^{1-\nu} \,+\, 
			\| w_\epsilon\|_{L^\frac{4}{\alpha+n+1}(\TT)}.
		\end{equation}
		Since 
		\begin{equation}\label{Eq35}
			r	\nu \,=\, 
			\tfrac{r(\alpha +n+1) -4}{\alpha+n+4}\,=\, 4,
		\end{equation}
		we obtain by integrating  the $r$-th power of \eqref{Eq31}  in time 
		\begin{align}\begin{split}
				\label{Eq32}
				\| w_\epsilon\|_{L^r([0,T]\times\TT)}^r\,&\lesssim_{\alpha, n} \,
				\int_0^T  \| \partial_x w_\epsilon\|_{L^4(\TT)}^{r\nu}  \| w_\epsilon\|_{L^\frac{4}{\alpha+n+1}(\TT)}^{r(1-\nu)}\,dt\,+\,
				\int_0^T \| w_\epsilon\|_{L^\frac{4}{\alpha+n+1}(\TT)}^r\,dt \\&\lesssim_T\, 
				\|\partial_x w_\epsilon \|_{L^4([0,T]\times \TT)}^4 
				\|w_\epsilon\|_{L^\infty(0,T;L^{\frac{4}{\alpha+n+1}}(\TT))}^{r(1-\nu)}\,+\, \|w_\epsilon\|_{L^\infty(0,T;L^{\frac{4}{\alpha+n+1}}(\TT))}^r.
			\end{split}
		\end{align}
		By \eqref{Eq223} we have $\frac{4p}{\alpha+n+1}\,=\, r$ and consequently
		\begin{align}
			\|u_\epsilon\|_{L^p([0,T]\times\TT)}^p\,=\,
			\bigl\|u_\epsilon^\frac{\alpha+n+1}{4}\bigr\|_{L^{\frac{4p}{\alpha+n+1}}([0,T]\times\TT)}^\frac{4p}{\alpha+n+1}\,=\,
			\|w_\epsilon\|_{L^{r}([0,T]\times\TT)}^{r},
		\end{align}
		and moreover
		\[\|u_\epsilon\|_{L^1(\TT)}\,=\, \bigl\|u_\epsilon^\frac{\alpha+n+1}{4}\bigr\|_{L^\frac{4}{\alpha+n+1}(\TT)}^\frac{4}{\alpha+n+1}\,=\, \|w_\epsilon\|_{L^\frac{4}{\alpha+n+1}(\TT)}^\frac{4}{\alpha+n+1}.\]
		Using these two identities in \eqref{Eq32},  taking the conditional expectation with respect to $\mathfrak{F}_0$  and applying  estimate \eqref{Eq25_New}, we conclude that
			\begin{align}&
			\EE\bigl[ \|u_\epsilon\|_{L^p([0,T]\times\TT)}^p\,\big|\,\mathfrak{F}_0 \bigr]
			\\&
			\quad\lesssim_{\alpha, n, T}\,
			\EE\Bigl[ \|\partial_x w_\epsilon \|_{L^4([0,T]\times \TT)}^4 
			\|u_\epsilon\|_{L^\infty(0,T;L^1(\TT))}^{\frac{r(1-\nu)(\alpha+n+1)}{4}}\,\Big|\,\mathfrak{F}_0\Bigr]\,+\, 
			\EE\Bigl[
			\|u_\epsilon\|_{L^\infty(0,T;L^1(\TT))}^{\frac{r(\alpha+n+1)}{4}}\,\Big|\,\mathfrak{F}_0\Bigr]
			\\&
			\quad\le\,
			\EE\Bigl[ \|\partial_x w_\epsilon \|_{L^4([0,T]\times \TT)}^4 
			\bigl(\|u_0\|_{\MM(\TT)}+\epsilon\bigr)^{\frac{r(1-\nu)(\alpha+n+1)}{4}}\,\Big|\,\mathfrak{F}_0\Bigr]\,+\,
			\bigl(\|u_0\|_{\MM(\TT)}+\epsilon\bigr)^{\frac{r(\alpha+n+1)}{4}}.
		\end{align}
				Using that by \eqref{Eq223} and \eqref{Eq35}
		\begin{align}&
			\tfrac{r(\alpha+n+1)}{4}\,=\, p,\\&
			\tfrac{r(1-\nu)(\alpha+n+1)}{4}\,=\, p\,-\, \tfrac{r\nu(\alpha+n+1)}{4}\,=\, 4,
		\end{align}
		and  estimates \eqref{Eq13_New}, \eqref{Eq103}, we  obtain 
			\begin{align}&
			\EE\bigl[ \|u_\epsilon\|_{L^p([0,T]\times\TT)}^p\,\big|\, \mathfrak{F}_0\bigr]
			\\&\quad \lesssim_{\alpha, n,T}\,
			\bigl(\|u_0\|_{\MM(\TT)}+\epsilon\bigr)^4\EE\bigl[\|\partial_x w_\epsilon \|_{L^4([0,T]\times \TT)}^4 \,\big|\, \mathfrak{F}_0
			\bigr]\,+\,
			\bigl(\|u_0\|_{\MM(\TT)}+\epsilon\bigr)^p
			\\&\quad\lesssim_{\alpha, n,p,\Lambda, T}\, 	\bigl(\|u_0\|_{\MM(\TT)}+\epsilon\bigr)^4  \bigl(\|u_0\|_{\MM(\TT)}^{\alpha+1} + \|u_0\|_{\MM(\TT)}^{\alpha+n-1}+\epsilon^{\alpha+1}+\|u_0\|_{\MM(\TT)}^{p-4} +\epsilon^{p-4}\bigr).
		\end{align}
		We use Assumption \ref{Condition_n}, \eqref{Eq223} and that $\epsilon\in (0,1)$ to simplify the right-hand side to
	\[	\EE\bigl[ \|u_\epsilon\|_{L^p([0,T]\times\TT)}^p\,\big|\, \mathfrak{F}_0\bigr]
	\,\lesssim_{\alpha, n,p,\Lambda, T}\,
\bigl(\|u_0\|_{\MM(\TT)}+\epsilon\bigr)^4   \bigl(\|u_0\|_{\MM(\TT)}^{p-n-4} +\|u_0\|_{\MM(\TT)}^{p-4}+\epsilon^{p-n-4}\bigr).
	\]
	Finally, \eqref{Eq38_New} follows by observing that  $\alpha$ depends only on $n$ and $p$.
	\end{proof}
	\begin{lemma}\label{Lemma_int_r}
		Let $r\in (\frac{n+4}{2},\frac{7}{2})$, then 
		\begin{equation}\label{Eq33}
			\EE\bigl[ \|\partial_x u_\epsilon\|_{L^r([0,T]\times \TT)}^r\,\big|\, \mathfrak{F}_0
			\bigr]\,\lesssim_{n, r, \Lambda, T}\, 
			(\|u_0\|_{\MM(\TT)}+\epsilon)^{4-r} \bigl(\|u_0\|_{\MM(\TT)}^{2r-n-4}+ \|u_0\|_{\MM(\TT)}^{2r-4}+\epsilon^{2r-n-4}\bigr).
		\end{equation}
	\end{lemma}
\begin{proof} 
	We define $p=2r$ and $\alpha$ according to \eqref{Eq223}, such that in particular the assumptions of Lemma \ref{lemma_int_p} and Proposition \ref{Thm_Ito_s_formula_New} are satisfied.
	We consider again the function $w_\epsilon= u_\epsilon^\frac{\alpha+n+1}{4} $, which satisfies the chain rule by \eqref{Eq34}. Hence, using H\"older's inequality and that
	\begin{equation}\label{Eq36_New}
		\tfrac{1}{r}\,= \,\tfrac{1}{4}\,+\, \tfrac{3-(\alpha+n)}{4p},
	\end{equation}
	we can estimate
		\begin{align}
		\|\partial_x u_\epsilon\|_{L^r([0,T]\times \TT)}\,& \eqsim_{\alpha,n}\, 
		\bigl\|u_\epsilon^\frac{3-(\alpha+n)}{4}\partial_x w_\epsilon\bigr\|_{L^r([0,T]\times \TT)}\, \le \,
		\left\| \partial_x w_\epsilon\right\|_{L^4([0,T]\times \TT)}
		\bigl\|u_\epsilon^\frac{3-(\alpha+n)}{4}\bigr\|_{L^\frac{4p}{3-(\alpha+n)}([0,T]\times \TT)}\\&=\,
		\left\| \partial_x w_\epsilon\right\|_{L^4([0,T]\times \TT)}
		\left\|u_\epsilon\right\|_{L^p([0,T]\times \TT)}^\frac{3-(\alpha+n)}{4},
	\end{align}
Taking the $r$-th power on both sides, taking the  conditional expectation with respect to $\mathfrak{F}_0$, and employing the conditional H\"older inequality  yields 
\begin{align}
	\EE\bigl[ 
	\|\partial_x u_\epsilon\|_{L^r([0,T]\times \TT)}^r\,\big|\, \mathfrak{F}_0\bigr]\,& \lesssim_{\alpha,n} \,
	\EE\Bigl[ 
	\left\| \partial_x w_\epsilon\right\|_{L^4([0,T]\times \TT)}^r
	\left\|u_\epsilon\right\|_{L^p([0,T]\times \TT)}^\frac{r(3-(\alpha+n))}{4}\,\Big|\,\mathfrak{F}_0\Bigr]\\&\le \, 
	\EE\bigl[   \left\|
	\partial_x w_\epsilon\right\|_{L^4([0,T]\times \TT)}^4\,\big|\, \mathfrak{F}_0
	\bigr]^\frac{r}{4} 
	\EE\bigl[  \left\|u_\epsilon\right\|_{L^p([0,T]\times \TT)}^p\,\big|\, \mathfrak{F}_0
	\bigr]^\frac{r(3-(\alpha+n))}{4p} .
\end{align}
An explicit calculation yields $\frac{r(3-(\alpha+n))}{4p}=1-\frac{r}{4}$, such that inserting \eqref{Eq13_New}, \eqref{Eq38_New} and the definitions of $\alpha,p$ results in 
\begin{align}&
		\EE\bigl[ 
	\|\partial_x u_\epsilon\|_{L^r([0,T]\times \TT)}^r|\mathfrak{F}_0\bigr]\,\lesssim_{\alpha,n,p,\Lambda, T}\,
	\big(\|u_0\|_{\MM(\TT)}+\epsilon\big)^{4-r} \bigl(\|u_0\|_{\MM(\TT)}^{2r-n-4}+ \|u_0\|_{\MM(\TT)}^{2r-4}+\epsilon^{2r-n-4}\bigr).
\end{align}
Finally, using that $\alpha,p$ only depend on $n$ and $r$, we infer that \eqref{Eq33} holds.
\end{proof}
	\subsection{Temporal Regularity}
	In what follows, we use the estimates derived in Subsection \ref{SS_spatial} to deduce uniform estimates on the time increments of $u_\epsilon$  with values in a suitable negative Sobolev space on $\TT$. Since the estimates from Subsection \ref{SS_spatial} only give estimates on $\partial_x u_\epsilon$ and $u_\epsilon$, we need to rewrite the thin-film operator in the weaker form \cite[Eq. (3.2)]{Passo98ona}. Specifically, by integrating by parts  we obtain that
	\begin{align}\begin{split}
		\label{Eq39}&
		\int_{\{u_\epsilon>0\}}
		m(u_\epsilon) \partial_x^3 u_\epsilon \eta\, dx \\&\quad=\, \tfrac{n(n-1)}{2}\, \langle u_\epsilon^{n-2}(\partial_x u_\epsilon)^3, \eta
		\rangle\,+\,\tfrac{3n}{2}\,\langle
		u_\epsilon^{n-1}(\partial_x u_\epsilon)^2 , \partial_x\eta
		\rangle\,+\,
		\langle
		u_\epsilon^n \partial_x u_\epsilon,\partial_x^2 \eta
		\rangle
	\end{split}
\end{align}
	$\PP\otimes dt$-almost everywhere for every $\eta\in C^\infty(\TT)$. The integration by parts is justified by the regularity of $u_\epsilon$ obtained from Consequence \ref{Consequences_Ass_1} \ref{Item_Consequences_5}. In the subsequent lemma, we deduce estimates on the terms on the right-hand side of \eqref{Eq39}.
		\begin{lemma}
		\label{Lemma_reg_tf_int_New}Let $l\in \{0,1,2\}$ and $\nu_l\in (\frac{n+4}{n+4-l},\frac{7}{n+4-l})$, then 
		\begin{align}\begin{split}
				\label{Eq62_New}&
				\EE\biggl[
				\biggl\| \int_0^\cdot  u_\epsilon^{n-2+l} (\partial_x u_\epsilon)^{3-l}\, dt \biggr\|_{W^{1,\nu_l}( 0,T;L^{\nu_l}(\TT)) }^{\nu_l}\,\bigg|\, \mathfrak{F}_0
				\biggr]\\&\quad \lesssim_{ l,n,\nu_l, \Lambda,T}\, (\|u_0\|_{\MM(\TT)}+\epsilon)^{4-\nu_l(3-l)}
			 \bigl(\|u_0\|_{\MM(\TT)}^{\nu_l(n+4-l)-n-4}+ \|u_0\|_{\MM(\TT)}^{\nu_l(n+4-l)-4}+\epsilon^{\nu_l(n+4-l)-n-4}\bigr)
				.
			\end{split}
		\end{align}
	\end{lemma}
	\begin{proof}
		We choose $p=2r=\nu_l(n+4-l)$, such that in particular $p\in (n+4,7)$ and $r\in (\frac{n+4}{2},\frac{7}{2})$, meaning that the assumptions of Lemma \ref{lemma_int_p} and Lemma \ref{Lemma_int_r} are satisfied. Moreover, we have 
		\begin{equation}\label{Eq228}
\tfrac{	n-2+l
}{p}	\,+\, \tfrac{3-l}{r}\,=\, \tfrac{n+4-l}{2r}\,=\, \tfrac{1}{\nu_l}.	\end{equation}
Hence, using that  $\int_0^\cdot  u_\epsilon^{n-2+l} (\partial_x u_\epsilon)^{3-l}\, dt$ is starting at $0$ and admits its integrand as weak derivative, as well as H\"older's inequality, we can estimate
\begin{align}\begin{split}\label{Eq66_New}
		&
		\biggl\| 
		\int_0^\cdot  u_\epsilon^{n-2+l} (\partial_x u_\epsilon)^{3-l}\, dt
		\biggr\|_{W^{1,\nu_l}( 0,T;L^{\nu_l}(\TT)) }\,\lesssim_{T}\, \| u_\epsilon^{n-2+l} (\partial_x u_\epsilon)^{3-l}  \|_{L^{\nu_l}([0,T]\times\TT)} \\&\quad\le \,
		\|u_\epsilon^{n-2+l}\|_{L^\frac{p}{n-2+l}([0,T]\times\TT)} \|(\partial_x u_\epsilon)^{3-l}\|_{L^\frac{r}{3-l}([0,T]\times\TT)}
		\,=\, 
		\|u_\epsilon\|_{L^p([0,T]\times\TT)}^{n-2+l}\|\partial_x u_\epsilon\|_{L^r([0,T]\times\TT)}^{3-l}.
	\end{split}
\end{align}
Taking the $\nu_l$-th power on both sides and the conditional expectation with respect to $\mathfrak{F}_0$, using \eqref{Eq38_New} and  \eqref{Eq33}, and employing the conditional H\"older's inequality we conclude 
\begin{align}\begin{split}\label{Eq67}
		&
		\EE\biggl[\biggl\| 
		\int_0^\cdot  u_\epsilon^{n-2+l} (\partial_x u_\epsilon)^{3-l}\, dt
		\biggr\|_{W^{1,\nu_l}( 0,T;L^{\nu_l}(\TT)) }^{\nu_l}\, \bigg|\, \mathfrak{F}_0
		\biggr]\,\lesssim_{\nu_l,T} \, 
		\EE\bigl[
		\|u_\epsilon\|_{L^p([0,T]\times\TT)}^{{\nu_l}(n-2+l)}\|\partial_x u_\epsilon\|_{L^r([0,T]\times\TT)}^{{\nu_l}(3-l)}\,\big|\,\mathfrak{F}_0
		\bigr]\\&\quad \le \, 
		\EE\bigl[
		\|u_\epsilon\|_{L^p([0,T]\times\TT)}^{p}\,\big|\, \mathfrak{F}_0
		\bigr]^{\frac{{\nu_l}(n-2+l)}{p}} 
		\EE\bigl[\|\partial_x u_\epsilon\|_{L^r([0,T]\times\TT)}^{r}\,\big|\, \mathfrak{F}_0
		\bigr]^{\frac{{\nu_l}(3-l)}{r}} 
		\\&\quad \lesssim_{n,p,r,\Lambda, T}\,
		(\|u_0\|_{\MM(\TT)}+\epsilon)^{\frac{4{\nu_l}(n-2+l)}{p} 
		}\bigl(\|u_0\|_{\MM(\TT)}^{p-n-4}+ \|u_0\|_{\MM(\TT)}^{p-4}+\epsilon^{p-n-4}\bigr)^\frac{{\nu_l}(n-2+l)}{p} 
	\\&\qquad \times 
			(\|u_0\|_{\MM(\TT)}+\epsilon)^{\frac{{(4-r)\nu_l}(3-l)}{r}} \bigl(\|u_0\|_{\MM(\TT)}^{2r-n-4}+ \|u_0\|_{\MM(\TT)}^{2r-4}+\epsilon^{2r-n-4}\bigr)^\frac{\nu_l(3-l)}{r}.
		\end{split}
	\end{align}
	The claim follows by using \eqref{Eq228} and inserting the definitions of $p$ and $r$.
	\end{proof}
		We derive similar estimates on terms appearing in the Stratonovich correction term  in \eqref{Eq55_New}.
	\begin{lemma}\label{Lemma_reg_strat_int_New}
			Let $l\in \{3,4\}$ and $\nu_l\in(\frac{n+4}{n+3-l},\frac{7}{n+3-l})$, then 
			\begin{align}\begin{split}\label{Eq229}&
						\EE\biggl[
					\biggl\| \sum_{k\in \ZZ} \int_0^\cdot  
					\sigma_{k,\epsilon}^2 (q'(u_\epsilon))^2 \partial_x u_\epsilon
					\, dt \biggr\|_{W^{1,\nu_3}(0,T;L^{\nu_3}(\TT)) }^{\nu_3}\,\bigg|\,\mathfrak{F}_0
					\biggr]
					\\&\quad \lesssim_{n, \nu_3,\Lambda, T}
					\, (\|u_0\|_{\MM(\TT)}+\epsilon)^{4-\nu_3}
					\bigl(
					\|u_0\|_{\MM(\TT)}^{n\nu_3-n-4}+\|u_0\|_{\MM(\TT)}^{n\nu_3-4}+\epsilon^{n\nu_3-n-4}
					\bigr)
				\end{split}
			\end{align} and 
		\begin{align}\begin{split}\label{Eq231}
				& 
			\EE\biggl[
			\biggl\| \sum_{k\in \ZZ} \int_0^\cdot 
			\sigma_{k,\epsilon}\partial_x \sigma_{k,\epsilon} q(u_\epsilon) q'(u_\epsilon) 
			\, dt \biggr\|_{W^{1,\nu_4}( 0,T;L^{\nu_4}(\TT)) }^{\nu_4} \,\bigg|\,\mathfrak{F}_0
			\biggr]\\&\quad \lesssim_{ n,\nu_4, \Lambda,T}\, (\|u_0\|_{\MM(\TT)}+\epsilon)^{4}
			\bigl(
			\|u_0\|_{\MM(\TT)}^{(n-1)\nu_4-n-4}+\|u_0\|_{\MM(\TT)}^{(n-1)\nu_4-4}+\epsilon^{(n-1)\nu_4-n-4}
			\bigr).
			\end{split}
		\end{align} 
	\end{lemma}	
\begin{proof}We first consider \eqref{Eq229} and define $p,r$  by $p=2r=n\nu_3$, such that the assumptions of Lemma \ref{lemma_int_p} and Lemma \ref{Lemma_int_r} are satisfied. Then
	\begin{equation}\label{Eq230}
		\tfrac{n-2}{p}\,+\, \tfrac{1}{r}\,=\, \tfrac{1}{\nu_3}.
	\end{equation}
		We use that $\sum_{k\in \ZZ} \int_0^\cdot  
	\sigma_{k,\epsilon}^2 (q'(u_\epsilon))^2 \partial_x u_\epsilon
	\, dt$ starts at $0$ to estimate  
	\begin{align}&
		\biggl\| \sum_{k\in \ZZ} \int_0^\cdot  
		\sigma_{k,\epsilon}^2 (q'(u_\epsilon))^2 \partial_x u_\epsilon
		\, dt \biggr\|_{W^{1,\nu_3}(0,T;L^{\nu_3}(\TT)) }
		\,\lesssim_T \,
		\biggl\| \sum_{k\in \ZZ} 
		\sigma_{k,\epsilon}^2 (q'(u_\epsilon))^2 \partial_x u_\epsilon
		\biggr\|_{L^{\nu_3}([0,T]\times \TT) }
		\\&\quad \lesssim_n\,
		\sum_{k\in \ZZ} 
		\|\sigma_{k,\epsilon}\|_{C(\TT)}^2 \|u_\epsilon^{n-2}\partial_x u_\epsilon
		\|_{L^{\nu_3}([0,T]\times \TT) }\, \lesssim_\Lambda\, \|u_\epsilon^{n-2}\partial_x u_\epsilon
		\|_{L^{\nu_3}([0,T]\times \TT) }.
	\end{align}
	Proceeding as in \eqref{Eq67}, we obtain that
	\begin{align}&
		\EE\biggl[
		\biggl\| \sum_{k\in \ZZ} \int_0^\cdot  
		\sigma_{k,\epsilon}^2 (q'(u_\epsilon))^2 \partial_x u_\epsilon
		\, dt \biggr\|_{W^{1,\nu_3}(0,T;L^{\nu_3}(\TT)) }^{\nu_3}\,\bigg|\, \mathfrak{F}_0
		\biggr]\\&\quad \lesssim_{n,\Lambda, T}\, \EE\bigl[
	\|u_\epsilon\|_{L^p([0,T]\times \TT)}^p\,\big|\, \mathfrak{F}_0
		\bigr]^\frac{\nu_3(n-2)}{p}
		\EE\bigl[\|\partial_x u_\epsilon\|_{L^r([0,T]\times \TT)}^r\,\big|\,\mathfrak{F}_0
		\bigr]^\frac{\nu_3}{r}
		\\&\quad \lesssim_{n,p,r,\Lambda, T} 	(\|u_0\|_{\MM(\TT)}+\epsilon)^\frac{4\nu_3(n-2)}{p}  \bigl(\|u_0\|_{\MM(\TT)}^{p-n-4}+ \|u_0\|_{\MM(\TT)}^{p-4}+\epsilon^{p-n-4}\bigr)^\frac{\nu_3(n-2)}{p}
		\\&\quad \quad \times 
		(\|u_0\|_{\MM(\TT)}+\epsilon)^\frac{(4-r)\nu_3}{r} \bigl(\|u_0\|_{\MM(\TT)}^{2r-n-4}+ \|u_0\|_{\MM(\TT)}^{2r-4}+\epsilon^{2r-n-4}\bigr)^\frac{\nu_3}{r}.
	\end{align}
	The claimed estimate \eqref{Eq229} follows by using \eqref{Eq230} and inserting the definitions of $p$ and $r$. The second estimate \eqref{Eq231} can be derived analogously with the choice $p=(n-1)\nu_3$.
	\end{proof}
		Lastly, we obtain temporal regularity of the stochastic integral in \eqref{Eq55_New}.
		
		\begin{lemma}\label{Lemma_reg_stoch_int_New}
			Let $\nu_5\in (\frac{2(n+4)}{n}, \frac{14}{n})$, $\gamma_5 \in (0,\frac{1}{2})$. Then 
			\begin{align}\begin{split}\label{Eq233}
					&
			\EE\biggl[\biggl\|
			\sum_{k\in \ZZ} \int_0^\cdot \sigma_{k,\epsilon} q(u_\epsilon)\, d\beta^{(k)}_t
			\biggr\|_{W^{\gamma_5,\nu_5}(0,T; L^2(\TT))}^{\nu_5}\,\bigg|\,\mathfrak{F}_0
			\biggr]
			\\&\quad \lesssim_{\gamma_5, n, \nu_5, \Lambda, T}\,
			(\|u_0\|_{\MM(\TT)}+\epsilon)^4 \Bigl(
			\|u_0\|_{\MM(\TT)}^{\frac{n\nu_5}{2}-n-4} +\|u_0\|_{\MM(\TT)}^{\frac{n\nu_5}{2}-4}+\epsilon^{\frac{n\nu_5}{2}-n-4}
			\Bigr).
				\end{split}
			\end{align}
		\end{lemma}
		\begin{proof}
			We define the linear operator $\Phi_\epsilon\colon l^2(\ZZ)\to L^2(\TT)$ by setting $\Phi_\epsilon e_k\,=\,\sigma_{k,\epsilon}q(u_\epsilon)$ so that we can write in what follows
			\[\sum_{k\in \ZZ}
			\int_0^\cdot \sigma_{k,\epsilon} q(u_\epsilon)\, d\beta^{(k)}_t\,=\,
			\int_0^\cdot \Phi_\epsilon \, d\beta_t,
			\]
			where $\beta$ is the cylindrical Wiener process in $l^2(\ZZ)$ given by $e_k \mapsto \beta^{(k)}$. We let $A\in \mathfrak{F}_0$. Then, using \cite[Lemma 2.1]{flandoli_gatarek_1995}, we calculate
			\begin{align}& \EE\biggl[
				\mathbbm{1}_{A} \biggl\|
				\int_0^\cdot \Phi_\epsilon \, d\beta_t
				\biggr\|_{W^{\gamma_5,{\nu_5}}(0,T; L^2(\TT))}^{\nu_5}
				\biggr]\,=\,
				\EE\biggl[ \biggl\|
				\int_0^\cdot 
				\mathbbm{1}_{A}\Phi_\epsilon \, d\beta_t
				\biggr\|_{W^{\gamma_5,{\nu_5}}(0,T; L^2(\TT))}^{\nu_5}
				\biggr]\\&\quad\lesssim_{\gamma_5,{\nu_5}}\, 
				\EE\biggl[ 
				\int_0^T \|
				\mathbbm{1}_{A}\Phi_\epsilon \|_{L_2(l^2(\ZZ), L^2(\TT))}^{\nu_5}\, dt
				\biggr]\,=\,
				\EE\biggl[ 
				\mathbbm{1}_{A}
				\int_0^T \|\Phi_\epsilon \|_{L_2(l^2(\ZZ), L^2(\TT))}^{\nu_5}\, dt
				\biggr].
			\end{align}
			To further estimate the latter, we use \eqref{Consequence_sigma_New} and that $\nu_5\ge 2$ to infer
			\begin{align}&
				\int_0^T  \left\|\Phi_\epsilon \right\|_{L_2(l^2(\ZZ), L^2(\TT))}^{\nu_5}\, dt\,=\,
				\int_0^T \biggl(\sum_{k\in \ZZ} \|\sigma_{k,\epsilon} q(u_\epsilon)\|_{L^2(\TT)}^2 \biggr)^\frac{{\nu_5}}{2}\, dt\\&\quad \le \, 
				\int_0^T \biggl(\sum_{k\in \ZZ} \|\sigma_{k,\epsilon}\|_{C(\TT)}^2 \bigl\|u_\epsilon^\frac{n}{2}\bigr\|_{L^2(\TT)}^2 \biggr)^\frac{{\nu_5}}{2}\, dt\,\lesssim_\Lambda\, 
				\int_0^T \|u_\epsilon\|_{L^n(\TT)}^\frac{{\nu_5} n}{2}\,dt\,\le \, 
				\left\|u_\epsilon\right\|_{L^\frac{{\nu_5} n}{2}([0,T]\times\TT)}^\frac{{\nu_5} n}{2}.
			\end{align}
			Finally, we set $p=\frac{{\nu_5} n}{2}$ in accordance with the assumption of Lemma \ref{lemma_int_p} and consequently we can use \eqref{Eq38_New} to conclude that
			\begin{align}& \EE\biggl[
				\mathbbm{1}_{A} \biggl\|
				\int_0^\cdot \Phi_\epsilon \, d\beta_t
				\biggr\|_{W^{\gamma_5,{\nu_5}}(0,T; L^2(\TT))}^{\nu_5}
				\biggr]\,\lesssim_{\gamma_5, {\nu_5}, \Lambda}\, 
				\EE\bigl[ 
				\mathbbm{1}_{A}\left\|u_\epsilon\right\|_{L^p([0,T]\times\TT)}^p
				\bigr]\\&\quad
				=\, \EE\bigl[ 
				\mathbbm{1}_{A}\EE\bigl[\left\|u_\epsilon\right\|_{L^p([0,T]\times\TT)}^p\,\big|\, \mathfrak{F}_0\bigr]
				\bigr]
				\\&\quad \lesssim_{ n,p,\Lambda, T}\, 
			 \EE\bigl[\mathbbm{1}_{A}	(\|u_0\|_{\MM(\TT)}+\epsilon)^4 \bigl(\|u_0\|_{\MM(\TT)}^{p-n-4}+ \|u_0\|_{\MM(\TT)}^{p-4}+\epsilon^{p-n-4}\bigr)\bigr].
			\end{align}
		It remains to use that $A\in \mathfrak{F}_0$ was arbitrary and to insert the definition of $p$.
		\end{proof}
			Finally, we combine  the previous results from this subsection to deduce a uniform estimate on the temporal increments of $u_\epsilon$ in terms of its Sobolev-Slobodeckij norm.
			\begin{lemma}\label{Lemma_temporal_regularity_New}
					Let $\gamma\in(0,\frac{1}{2})$, $\mu\in (\frac{n+4}{n+2},\frac{7}{n+2})$ and $\nu\in (1,\frac{7}{n+4})$, then
					\begin{align}\begin{split}\label{Eq75_New2}
							&
							\EE\Bigl[
							\left\|
							u_\epsilon
							\right\|^{\nu}_{W^{\gamma, \frac{2\nu}{2-\nu}}(0,T;W^{-3,\mu}(\TT))}
							\,\Big|\, \mathfrak{F}_0 \Bigr]\\&\quad\lesssim_{\gamma, n, \mu,\nu, \Lambda, T}\, \bigl(
							\|u_0\|_{\MM(\TT)}+\epsilon
							\bigr)^{(n-1-\frac{n^2}{p})\nu}\,+\, 
							\bigl(
							\|u_0\|_{\MM(\TT)}+\epsilon
							\bigr)^{(n+1)\nu},
						\end{split}
					\end{align}
					where $p=\max\{ \mu(n+2), \nu(n+4) \}$.
			\end{lemma}
		\begin{proof}
			By Consequence \ref{Consequences_Ass_1} \ref{Item_Consequences_6} and \eqref{Eq39}, the equality 
		\begin{align}\begin{split}
				\label{Eq208_New}
				u_\epsilon\,=\,& u_{0,\epsilon}
				\,-\, \tfrac{n(n-1)}{2}\int_0^\cdot\partial_x(u_\epsilon^{n-2}(\partial_x u_\epsilon)^3)\,dt
				\,+\, \tfrac{3n}{2}\int_0^\cdot\partial_x^2 (u_\epsilon^{n-1} (\partial_x u_\epsilon)^2)\,dt
				\,-\, \int_0^\cdot\partial_x^3(u_\epsilon^{n}\partial_x u_\epsilon)\,dt
				\\&+\,\tfrac{1}{2}\sum_{k\in \ZZ}\int_0^\cdot  
				\partial_x (\sigma_{k,\epsilon}^2 (q'(u_\epsilon))^2 \partial_x u_\epsilon)\, dt
				\,+\,\tfrac{1}{2}\sum_{k\in \ZZ}\int_0^\cdot  
				\partial_x (\sigma_{k,\epsilon}\partial_x \sigma_{k,\epsilon} q(u_\epsilon) q'(u_\epsilon) )\, dt
				\\&+\, \sum_{k\in \ZZ} \int_0^\cdot \partial_x(\sigma_{k,\epsilon} q(u_\epsilon))\, d\beta^{(k)}_t
			\end{split}
		\end{align}	
	holds almost surely, where the integrals on the right hand side converge in suitable negative Sobolev-spaces as a consequence of Lemma \ref{Lemma_reg_tf_int_New}, Lemma \ref{Lemma_reg_strat_int_New} and Lemma \ref{Lemma_reg_stoch_int_New}. We proceed by estimating the $W^{\gamma, \frac{2\nu}{2-\nu}}(0,T;W^{-3,\mu}(\TT))$-norm of each of the terms on the right hand-side of \eqref{Eq208_New} separately. Since $u_{0,\epsilon}$ is constant in time, we can estimate by the Sobolev-embedding theorem and \eqref{Eq25_New}
	\begin{equation}\label{Eq239}
	\|u_{0,\epsilon}\|_{W^{\gamma, \frac{2\nu}{2-\nu}}(0,T; W^{-3,\mu}(\TT))}\,\lesssim_\mu\,
	\|u_{0,\epsilon}\|_{L^{\frac{2\nu}{2-\nu}}(0,T;L^1(\TT))}\,\lesssim_{\nu, T}\,\|u_{0,\epsilon}\|_{L^1(\TT)}\,\le\,\|u_0\|_{\MM(\TT)}\,+\,\epsilon.
	\end{equation}
	For the remaining terms, we choose
	\begin{align}\begin{split}
			\label{Eq232}
			&\nu_l \,=\, \tfrac{p}{n+4-l},\quad l\in \{0,1,2\},
			\\&\nu_l\,=\, \tfrac{p}{n+3-l},\quad l\in \{3,4\},
			\\&\nu_5\,=\, \tfrac{2p}{n}, \quad\gamma_5=\gamma,
		\end{split}
	\end{align} where $p$ is defined in the claim. In particular, we have $\nu_l\ge \nu$  and therefore 
\begin{equation}\label{Eq209_New}
	1-\tfrac{1}{\nu_l}\,\ge \, 1-\tfrac{1}{\nu}\,=\, \tfrac{1}{2}-\tfrac{2-\nu}{2\nu}\,>\, \gamma-\tfrac{2-\nu}{2\nu},\quad l\in\{0,\dots, 4\}.
\end{equation}
Using additionally that $\nu_2\ge \mu$ and employing the Sobolev embedding theorem in time and space, we obtain that
\begin{align}\begin{split}
		\label{Eq235}
	\biggl\| \int_0^\cdot \partial_x^{l+1} (u_\epsilon^{n-2+l} (\partial_x u_\epsilon)^{3-l})\, dt \biggr\|_{W^{\gamma,\frac{2\nu}{2-\nu}}( 0,T;W^{-3,\mu}(\TT)) }\,& \le \,
	\biggl\| \int_0^\cdot u_\epsilon^{n-2+l} (\partial_x u_\epsilon)^{3-l}\, dt \biggr\|_{W^{\gamma,\frac{2\nu}{2-\nu}}( 0,T;W^{l-2,\mu}(\TT)) } \\&\lesssim_{l,\gamma,\mu,\nu,\nu_l,T}\,
	\biggl\| \int_0^\cdot  u_\epsilon^{n-2+l} (\partial_x u_\epsilon)^{3-l}\, dt \biggr\|_{W^{1,\nu_l}( 0,T;L^{\nu_l}(\TT)) }
	\end{split}
\end{align}
for $l\in \{0,1,2\}$. 	For $l\in \{3,4\}$, we proceed similarly and use again \eqref{Eq209_New} and the Sobolev embedding to conclude
\begin{align}\begin{split}\label{Eq236}
		&
	\biggl\| \sum_{k\in \ZZ}\int_0^\cdot  
	\partial_x (\sigma_{k,\epsilon}^2 (q'(u_\epsilon))^2 \partial_x u_\epsilon)\, dt \biggr\|_{W^{\gamma,\frac{2\nu}{2-\nu}}( 0,T;W^{-3,\mu}(\TT)) }\\&\quad\lesssim_{\gamma,\mu,\nu,\nu_3,T}\,
	\biggl\| \sum_{k\in \ZZ}\int_0^\cdot  
	\sigma_{k,\epsilon}^2 (q'(u_\epsilon))^2 \partial_x u_\epsilon\, dt \biggr\|_{W^{1,\nu_3}( 0,T;L^{\nu_3}(\TT)) }
	\end{split}
\end{align}
and 
\begin{align}\begin{split}\label{Eq237}&
	\biggl\| \sum_{k\in \ZZ}\int_0^\cdot  
	\partial_x (\sigma_{k,\epsilon}\partial_x \sigma_{k,\epsilon} q(u_\epsilon) q'(u_\epsilon) )\, dt \biggr\|_{W^{\gamma,\frac{2\nu}{2-\nu}}( 0,T;W^{-3,\mu}(\TT)) }\\&\quad\lesssim_{\gamma,\mu,\nu,\nu_4,T}\,
	\biggl\|  \sum_{k\in \ZZ}\int_0^\cdot  
	\sigma_{k,\epsilon}\partial_x \sigma_{k,\epsilon} q(u_\epsilon) q'(u_\epsilon) \, dt \biggr\|_{W^{1,\nu_4}( 0,T;L^{\nu_4}(\TT)) }.
	\end{split}
\end{align}
Lastly, we observe that by Consequence \ref{Consequences_Ass_1} \ref{Item_Consequences_1} and $\nu<\frac{7}{n+4}$, it also holds $\nu_5\ge \frac{2(n+4)\nu}{n}\ge \frac{2\nu}{2-\nu}$. Hence, by the Sobolev embedding theorem, we infer
\begin{align}\begin{split}\label{Eq238}&
	\biggl\|\sum_{k\in \ZZ} \int_0^\cdot \partial_x(\sigma_{k,\epsilon} q(u_\epsilon))\, d\beta^{(k)}_t \biggr\|_{W^{\gamma,\frac{2\nu}{2-\nu}}( 0,T;W^{-3,\mu}(\TT))}\\&\quad \lesssim_{\gamma,\mu,\nu,\nu_5}\,
	\biggl\|
	\sum_{k\in \ZZ} \int_0^\cdot \sigma_{k,\epsilon} q(u_\epsilon)\, d\beta^{(k)}_t
	\biggr\|_{W^{\gamma_5,\nu_5}(0,T; L^2(\TT))}.
\end{split}
\end{align}
Employing the triangle inequality in $W^{\gamma, \frac{2\nu}{2-\nu}}(0,T;W^{-3,\mu}(\TT))$ and the conditional Minkowski inequality in \eqref{Eq208_New} yields
\begin{align}&
	\EE\Bigl[
	\|
	u_\epsilon
	\|^{\nu}_{W^{\gamma, \frac{2\nu}{2-\nu}}(0,T;W^{-3,\mu}(\TT))}\,\Big|\, \mathfrak{F}_0
	\Bigr]^{\frac{1}{\nu}}\, \lesssim_{n} \,
	\EE\Bigl[\|u_{0,\epsilon}\|_{W^{\gamma, \frac{2\nu}{2-\nu}}(0,T;W^{-3,\mu}(\TT))}^\nu\,\Big|\,\mathfrak{F}_0
	\Bigr]^{\frac{1}{\nu}}
	\\& \qquad+\,
	\sum_{l=0}^2\EE\biggl[
	\biggl\| \int_0^\cdot  \partial_x^{l+1}(u_\epsilon^{n-2+l} (\partial_x u_\epsilon)^{3-l})\, dt \biggr\|_{W^{\gamma, \frac{2\nu}{2-\nu}}(0,T;W^{-3,\mu}(\TT))}^{\nu}\,\bigg|\,\mathfrak{F}_0
	\biggr]^\frac{1}{\nu}
	\\& \qquad
	+\, \EE\biggl[
	\biggl\| \sum_{k\in \ZZ}\int_0^\cdot \partial_{x}( 
	\sigma_{k,\epsilon}^2 (q'(u_\epsilon))^2 \partial_x u_\epsilon)\, dt \biggr\|^{\nu}_{W^{\gamma, \frac{2\nu}{2-\nu}}(0,T;W^{-3,\mu}(\TT))} \,\bigg|\,\mathfrak{F}_0
	\biggr]^\frac{1}{\nu}
	\\& \qquad
	+\, \EE\biggl[
	\biggl\|  \sum_{k\in \ZZ}\int_0^\cdot \partial_x (
	\sigma_{k,\epsilon}\partial_x \sigma_{k,\epsilon} q(u_\epsilon) q'(u_\epsilon)) \, dt \biggr\|^{\nu}_{W^{\gamma, \frac{2\nu}{2-\nu}}(0,T;W^{-3,\mu}(\TT))}\,\bigg|\,\mathfrak{F}_0
	\biggr]^\frac{1}{\nu}
	\\& \qquad
	+\, \EE\biggl[  \biggl\|
	\sum_{k\in \ZZ} \int_0^\cdot \partial_x (\sigma_{k,\epsilon} q(u_\epsilon))\, d\beta^{(k)}_t
	\biggr\|^{\nu}_{W^{\gamma, \frac{2\nu}{2-\nu}}(0,T;W^{-3,\mu}(\TT))}\,\bigg|\,\mathfrak{F}_0
	\biggr]^\frac{1}{\nu}.
\end{align}
The estimates \eqref{Eq239}, \eqref{Eq235}, \eqref{Eq236}, \eqref{Eq237} and \eqref{Eq238} and the conditional Jensen inequality lead to
\begin{align}&
	\EE\Bigl[
	\|
	u_\epsilon
	\|^{\nu}_{W^{\gamma, \frac{2\nu}{2-\nu}}(0,T;W^{-3,\mu}(\TT))}\,\Big|\, \mathfrak{F}_0
	\Bigr]^{\frac{1}{\nu}}\, \lesssim_{\gamma,n,\mu,\nu,T} \,
	\EE\bigl[\bigl(\|u_0\|_{\MM(\TT)}\,+\, \epsilon\bigr)^\nu\,\big|\, \mathfrak{F}_0
	\bigr]^{\frac{1}{\nu}}
	\\& \qquad+\,
	\sum_{l=0}^2\EE\biggl[
	\biggl\| \int_0^\cdot  u_\epsilon^{n-2+l} (\partial_x u_\epsilon)^{3-l}\, dt \biggr\|_{W^{1,\nu_l}( 0,T;L^{\nu_l}(\TT)) }^{\nu_l}\,\bigg|\, \mathfrak{F}_0
	\biggr]^\frac{1}{\nu_l}
	\\& \qquad
	+\, \EE\biggl[
	\biggl\| \sum_{k\in \ZZ}\int_0^\cdot  
	\sigma_{k,\epsilon}^2 (q'(u_\epsilon))^2 \partial_x u_\epsilon\, dt \biggr\|_{W^{1,\nu_3}( 0,T;L^{\nu_3}(\TT)) }^{\nu_3}\,\bigg|\, \mathfrak{F}_0
	\biggr]^\frac{1}{\nu_3}
	\\& \qquad
	+\, \EE\biggl[
	\biggl\|  \sum_{k\in \ZZ}\int_0^\cdot  
	\sigma_{k,\epsilon}\partial_x \sigma_{k,\epsilon} q(u_\epsilon) q'(u_\epsilon) \, dt \biggr\|_{W^{1,\nu_4}( 0,T;L^{\nu_4}(\TT)) }^{\nu_4}\,\bigg|\, \mathfrak{F}_0
	\biggr]^\frac{1}{\nu_4}
	\\& \qquad
	+\, \EE\biggl[  \biggl\|
	\sum_{k\in \ZZ} \int_0^\cdot \sigma_{k,\epsilon} q(u_\epsilon)\, d\beta^{(k)}_t
	\biggr\|_{W^{\gamma_5,\nu_5}(0,T; L^2(\TT))}^{\nu_5}\,\bigg|\, \mathfrak{F}_0
	\biggr]^\frac{1}{\nu_5},
\end{align}
where we also used that the $\nu_l$  only depend on  $n$, $\nu$, $\mu$.
Since $p\in (n+4,7)$,  the parameters from \eqref{Eq232} satisfy the assumptions from Lemma \ref{Lemma_reg_tf_int_New}, Lemma \ref{Lemma_reg_strat_int_New} and Lemma  \ref{Lemma_reg_stoch_int_New}. Hence, using \eqref{Eq62_New}, \eqref{Eq229}, \eqref{Eq231} and \eqref{Eq233} as well as that , we obtain that
\begin{align}\begin{split}
		\label{Eq75_New}&
		\EE\Bigl[
		\left\|
		u_\epsilon
		\right\|^{\nu}_{W^{\gamma, \frac{2\nu}{2-\nu}}(0,T;W^{-3,\mu}(\TT))}
		\,\Big|\, \mathfrak{F}_0 \Bigr]^{\frac{1}{\nu}}\,\lesssim_{\gamma, n, \mu,\nu, \Lambda, T}\, \bigl(
		\|u_0\|_{\MM(\TT)}+\epsilon
		\bigr)
		\\&\qquad +\,\sum_{l=0}^2\Bigl[ (\|u_0\|_{\MM(\TT)}+\epsilon)^{4-\nu_l(3-l)}
		\bigl(\|u_0\|_{\MM(\TT)}^{\nu_l(n+4-l)-n-4}+ \|u_0\|_{\MM(\TT)}^{\nu_l(n+4-l)-4}+\epsilon^{\nu_l(n+4-l)-n-4}\bigr)\Bigr]^{\frac{1}{\nu_l}}\\&\qquad +\,
		\Bigl[	(\|u_0\|_{\MM(\TT)}+\epsilon)^{4-\nu_3}
		\bigl(
		\|u_0\|_{\MM(\TT)}^{n\nu_3-n-4}+\|u_0\|_{\MM(\TT)}^{n\nu_3-4}+\epsilon^{n\nu_3-n-4}
		\bigr)\Bigr]^\frac{1}{\nu_3}	\\&\qquad +\,
		\Bigl[	(\|u_0\|_{\MM(\TT)}+\epsilon)^{4}
		\bigl(
		\|u_0\|_{\MM(\TT)}^{(n-1)\nu_4-n-4}+\|u_0\|_{\MM(\TT)}^{(n-1)\nu_4-4}+\epsilon^{(n-1)\nu_4-n-4}
		\bigr)\Bigr]^\frac{1}{\nu_4} 	\\&\qquad +\,
		\Bigl[
		(\|u_0\|_{\MM(\TT)}+\epsilon)^4 \Bigl(
		\|u_0\|_{\MM(\TT)}^{\frac{n\nu_5}{2}-n-4} +\|u_0\|_{\MM(\TT)}^{\frac{n\nu_5}{2}-4}+\epsilon^{\frac{n\nu_5}{2}-n-4}
		\Bigr)
		\Bigr]^\frac{1}{\nu_5}.
	\end{split}
\end{align}
To simplify the right-hand side, we  estimate $\epsilon$ and $\|u_0\|_{\MM(\TT)}$ in \eqref{Eq75_New} by $(\epsilon+\|u_0\|_{\MM(\TT)})$ to conclude 
\begin{align}\begin{split}\label{Eq244}
		&
		\EE\Bigl[
		\left\|
		u_\epsilon
		\right\|^{\nu}_{W^{\gamma, \frac{2\nu}{2-\nu}}(0,T;W^{-3,\mu}(\TT))}
		\,\Big|\, \mathfrak{F}_0 \Bigr]^\frac{1}{\nu}\,\lesssim_{\gamma, n, \mu,\nu, \Lambda, T}\, \bigl(
		\|u_0\|_{\MM(\TT)}+\epsilon
		\bigr)
		\\&\qquad +\,\sum_{l=0}^2 (\|u_0\|_{\MM(\TT)}+\epsilon)^{n+1-\frac{n}{\nu_l}}\,+\, (\|u_0\|_{\MM(\TT)}+\epsilon)^{n+1} \\&\qquad +\,
		(\|u_0\|_{\MM(\TT)}+\epsilon)^{n-1-\frac{n}{\nu_3}}\,+\, (\|u_0\|_{\MM(\TT)}+\epsilon)^{n-1}
		\\&\qquad +\,	(\|u_0\|_{\MM(\TT)}+\epsilon)^{n-1-\frac{n}{\nu_4}}\,+\, 
		(\|u_0\|_{\MM(\TT)}+\epsilon)^{n-1}\		\\&\qquad +\,	(\|u_0\|_{\MM(\TT)}+\epsilon)^{\frac{n}{2}-\frac{n}{\nu_5}}\,+\, 
		(\|u_0\|_{\MM(\TT)}+\epsilon)^{\frac{n}{2}},
	\end{split}
\end{align}
We notice that the largest power on the right-hand side of \eqref{Eq244} is $n+1$. 
By Consequence \ref{Consequences_Ass_1} \ref{Item_Consequences_1}, the smallest power is either $1$, $n-1-\frac{n}{\nu_3}$ or $\frac{n}{2}-\frac{n}{\nu_5}$. To find the smallest one, we insert \eqref{Eq232} to rewrite the powers to
\begin{equation}\label{Eq245}1,\qquad
	n-1-\tfrac{n}{\nu_3}\,=\, n-1-\tfrac{n^2}{p},\qquad 
	\tfrac{n}{2}-\tfrac{n}{\nu_5}\,=\,\tfrac{n}{2}-\tfrac{n^2}{2p}
\end{equation}
and consider the respective parabolas
\begin{equation}\label{Eq245_New}g_1(x)\,=\,1,\qquad
g_2(x)\,=\, x-1-\tfrac{x^2}{p},\qquad 
	g_3(x)\,=\,\tfrac{x}{2}-\tfrac{x^2}{2p}.
\end{equation}
We notice that all three parabolas attain their maximum value at $\frac{p}{2}$, with 
\[g_1(\tfrac{p}{2})
\,=\, 1\,\ge \, g_3(\tfrac{p}{2})\,=\,\tfrac{p}{8}\,\ge \, g_2(\tfrac{p}{2})\,=\,\tfrac{p}{4}-1,
\]
since $p\le 8$. Because the second derivatives of the parabolas obey the same ordering, we conclude $g_1(x)\ge g_3(x)\ge g_2(x)$ for all $x\in \RR$ and in particular that $n-1-\tfrac{n^2}{p}$ is the smallest power in \eqref{Eq244}. Whence,
\[
\EE\Bigl[
\left\|
u_\epsilon
\right\|^{\nu}_{W^{\gamma, \frac{2\nu}{2-\nu}}(0,T;W^{-3,\mu}(\TT))}
\,\Big|\, \mathfrak{F}_0 \Bigr]^\frac{1}{\nu}\,\lesssim_{\gamma, n, \mu,\nu, \Lambda, T} (\|u_0\|_{\MM(\TT)}+\epsilon)^{n-1-\frac{n^2}{p}} \,+\,(\|u_0\|_{\MM(\TT)}+\epsilon)^{n+1} 
\]
and raising both sides of the preceding inequality to the $\nu$-th power yields \eqref{Eq75_New2}.
\end{proof}
	\subsection{Simplified Estimates} 
	In the previous subsections, we derived uniform estimates on the conditional expectations of the approximate solutions $(u_\epsilon)_{\epsilon\in (0,1)}$. To work with these estimates efficiently in the preceding section, we derive corresponding moment estimates with a simpler right-hand side. To this end, we introduce the sets
	\begin{equation}\label{Eq98}
				A_j\,=\, \bigl\{
				\|u_0\|_{\MM(\TT)}\in [j-1,j)
				\bigr\},\;\; j\in \NN,
			\end{equation}
		providing an $\mathfrak{F}_0$-measurable partition of the probability space $\Omega$ and point out that it suffices to show tightness on each of the sets $A_j$ separately in light of Lemma \ref{Lemma_localized_tightness}. Using that
		\begin{equation}\label{Eq249}
			\|u_0\|_{\MM(\TT)}+\epsilon \,\le \, j\,+\,1
		\end{equation}
		on $A_j$, we obtain by multiplying	\eqref{Eq13_New}  with $\mathbbm{1}_{A_j}$ and taking the expectation 
	\begin{align}\begin{split}&\label{Eq253}
			\forall \alpha\in (-1,2-n):\\&\quad 
				\EE\biggl[\mathbbm{1}_{A_j}
				\int_0^{T}\int_{\TT} u_\epsilon^{\alpha+n-1}
				(\partial_x^2{u_\epsilon} )^2 \,dx\,ds\,+\, \mathbbm{1}_{A_j}
				\int_0^{T}\int_{\TT}
				u_\epsilon^{\alpha+n-3}(\partial_x u_\epsilon)^4  \,dx\,ds
				\biggr]\,\quad 
				\lesssim_{\alpha, n,\Lambda,T} \, (j+1)^{\alpha+n-1}.
			\end{split}
		\end{align}
		In the same way, we conclude from Lemma \ref{lemma_int_p}  that
		\begin{equation}\label{Eq240}
			\forall p\in [1,7):\quad
			\EE\bigl[\mathbbm{1}_{A_j}\|u_\epsilon\|_{L^p([0,T]\times\TT)}^p\bigr]\,
			\lesssim_{ n,p,\Lambda, T}\,
			(j+1)^p,
		\end{equation} 
where the fact that the inequality also holds for $p\in [1,n+4]$ follows by H\"older's inequality. Analogously, we conclude from Lemma \ref{Lemma_int_r} that 
	\begin{equation}\label{Eq251}
		\forall r\in [1,\tfrac{7}{2}):\quad 
		\EE\bigl[\mathbbm{1}_{A_j}\|\partial_x u_\epsilon\|_{L^r([0,T]\times\TT)}^r\bigr]\,
		\lesssim_{ n,r,\Lambda, T}\,
		(j+1)^r.
	\end{equation}
Using additionally the Sobolev embedding theorem in space, we conclude from Lemma \ref{Lemma_temporal_regularity_New} 
\begin{align}\begin{split}\label{Eq250}
		&
	\forall \gamma\in (0,\tfrac{1}{2})\,\forall \mu\in (1, \tfrac{7}{n+2})\, \forall \nu\in [1, \tfrac{7}{n+4}):\\&\quad 
	\EE\Bigl[	\mathbbm{1}_{A_j}
		\left\|
	u_\epsilon
	\right\|^{\nu}_{W^{\gamma, \frac{2\nu}{2-\nu}}(0,T;W^{-3,\mu}(\TT))}
	\Bigr]\,\lesssim_{\gamma, n, \mu,\nu, \Lambda, T}\, (j+1)^{\nu(n+1)}.
	\end{split}
\end{align}
	\section{Limiting procedure}
	In this section,  we construct a martingale solution in the sense of Theorem \ref{thm_main} to the stochastic thin-film equation with initial value $u_0$. To this end, we  show tightness of the approximating family $(u_\epsilon)_{\epsilon\in (0,1)}$ in suitable spaces in Subsection \ref{Sec_tightness} and extract an equidistributed convergent subsequence converging to a solution in Subsection \ref{Sec_convergence}.
	\subsection{Tightness properties}\label{Sec_tightness}
	We define $\Xind=\RR$, $\X_{\text{BM}}=C([0,T])$, $\X_{\text{power}}=L^{2}(0,T; H^2(\TT))$, where we equip the latter space with its weak topology. Moreover, we choose sequences $-1\le \kappa_l\nearrow \frac{-1}{2}$, $1< p_l\nearrow 7$ and $1< r_l\nearrow \frac{7}{2}$. Then, we define the spaces $\X_{\text{cont}}$, $\X_{\text{Lebesgue}}$, $\X_{\text{Sobolev}}$ as the projective limit of the sequences $(C([0,T];H^{\kappa_l}(\TT)))_{l\in \NN}$, $(L^{p_l}([0,T]\times \TT))_{l\in \NN}$ and $(L^{r_l}(0,T; W^{1,r_l}(\TT)))_{l\in \NN}$, where we consider the latter sequence of spaces with their weak typologies, for details see Appendix \ref{Appendix_lcs}. Finally, we define the space 
	\begin{equation}\label{Eq79}
		\X\,=\, \Xii \times \XBMi  \times \Xc \times \XL \times \XS \times \Xpi,
	\end{equation}
	and equip it with the product topology.
	\begin{lemma}\label{Lemma_tightness}Let $\alpha_l \nearrow 2-n$ be a sequence satisfying \eqref{Condition_alpha_New}. Then the family 
		\begin{equation}\label{Eq77}
			\Bigl(	\Bigl( (\mathbbm{1}_{A_j})_{j\in \NN}, (\beta^{(k)})_{k\in \ZZ}, u_\epsilon, u_\epsilon, u_\epsilon, \bigl(u_\epsilon^{\frac{\alpha_l+n+1}{2}}\bigr)_{l\in \NN}\Bigr)\Bigr)_{\epsilon\in (0,1)}
		\end{equation}
		lies tight on $\X$.
	\end{lemma}
	\begin{proof}
		By Lemma \ref{Lemma_prod_tightness} it suffices to show tightness of each of the components of \eqref{Eq77} in their respective space separately. 
		
		\textit{Tightness of the indicator functions and Brownian motions.} The set of real numbers $\Xind$ is a Radon space and thus the law of $\mathbbm{1}_{A_j}$ is inner regular on open sets for every $j\in \NN$. Consequently, the family $(\mathbbm{1}_{A_j})_{\epsilon\in (0,1)}$ lies tight on $\Xind$. Tightness of the sequence $((\mathbbm{1}_{A_j})_{j\in \NN})_{\epsilon\in (0,1)}$ on $\Xii$ follows by Lemma \ref{Lemma_prod_tightness}.
		Similarly, since  $\XBM$ is a Radon space, the law of $\beta^{(k)}$ is inner regular on open sets, and consequently the family  $(\beta^{(k)})_{\epsilon\in (0,1)}$ lies tight on it. Another application of Lemma \ref{Lemma_prod_tightness} yields tightness of $((\beta^{(k)})_{k\in \mathbb{Z}})_{\epsilon\in (0,1)}$ on $\XBMi$.
		
		\textit{Tightness on $\Xc$.} By Lemma \ref{Tightness_Proj_Limit} and Lemma \ref{Lemma_localized_tightness} it suffices to show tightness of $(\mathbbm{1}_{A_j}u_\epsilon)_{\epsilon\in (0,1)}$ on $C([0,T];H^{\kappa_l}(\TT))$ for each $j,l\in \NN$. To this end, we choose $\gamma\in(0,\frac{1}{2})$, $\mu\in (1,\frac{7}{n+2})$ and $\nu\in [1,\frac{7}{n+4})$ such that $\gamma-\frac{2-\nu}{2\nu}>0$ and in particular the embedding 
		\[
		L^\infty(0,T;L^1(\TT))\cap W^{\gamma, \frac{2\nu}{2-\nu}}(0,T;W^{-3,\mu}(\TT)) \,\hookrightarrow\, C([0,T]; H^{\kappa_l}(\TT))
		\]
		is compact by \cite[Corollary 5]{Simon1986} and the Rellich–Kondrachov theorem. Consequently, the set 
			\[
		K_\delta \,=\, \Bigl\{u\,\Big|\,  \|u\|_{ L^\infty(0,T; L^1(\TT))}\le \tfrac{1}{\delta} , \, \|u\|_{W^{\gamma, \frac{2\nu}{2-\nu}}(0,T;W^{-3,\mu}(\TT))}\, \le \tfrac{1}{\delta}
		\Bigr\}
		\]
		lies compact in $C([0,T];H^{\kappa_l}(\TT))$ for $\delta>0$. By \eqref{Eq25_New} and \eqref{Eq249}, we have that
		 \[\EE\bigl[ \mathbbm{1}_{A_j}\|u_\epsilon\|_{L^\infty(0,T; L^1(\TT))} \bigr]\,\le\, j+1\]
		and hence together with \eqref{Eq250} and Chebychev's inequality, we conclude that
		\begin{align}\begin{split}
				\label{Eq78}
				\PP(\{
				\mathbbm{1}_{A_j}u_\epsilon\notin K_\delta 
				\})\,& \le \,
				\PP \bigl(\bigl\{
				\mathbbm{1}_{A_j}\|u_\epsilon \|_{L^\infty(0,T; L^1(\TT))}>\tfrac{1}{\delta}
				\bigr\}\bigr) \,+\, 
				\PP \Bigl(\Bigl\{
				\mathbbm{1}_{A_j}\|u_\epsilon\|_{ W^{\gamma, \frac{2\nu}{2-\nu}}(0,T;W^{-3,\mu}(\TT))}>\tfrac{1}{\delta}
				\Bigr\}\Bigr)
				\\&\le\,\delta \EE\bigl[
				\mathbbm{1}_{A_j}\|u_\epsilon \|_{L^\infty(0,T; L^1(\TT))}
				\bigr]\,+\, \delta^{\nu} \EE\Bigl[	\mathbbm{1}_{A_j}\|u_\epsilon\|_{ W^{\gamma, \frac{2\nu}{2-\nu}}(0,T;W^{-3,\mu}(\TT))}^\nu
				\Bigr]
				\\&\lesssim_{\gamma, n,\mu, \nu, \Lambda, T}\, \delta(j+1) \,+\, \delta^{\nu} (j+1)^{\nu(n+1)}.
			\end{split}
		\end{align}
		Tightness follows, since the right-hand side tends to $0$ as $\delta\searrow 0$.
		
		\textit{Tightness on $\XL$.} By Lemma \ref{Tightness_Proj_Limit} and Lemma \ref{Lemma_localized_tightness} it is enough to show tightness of $(\mathbbm{1}_{A_j}u_\epsilon)_{\epsilon\in (0,1)}$ on $L^{p_l}([0,T]\times \TT)$ for each $j,l\in \NN$. We let again $\gamma\in(0,\frac{1}{2})$, $\mu\in (1,\frac{7}{n+2})$, $\nu\in [1,\frac{7}{n+4})$ with $\gamma-\frac{2-\nu}{2\nu}>0$ and define the set
			\[
		K_\delta \,=\, \Bigl\{ u\,
		\Big|\,\max\Bigl\{\|u\|_{L^{p_{l+1}}([0,T]\times \TT)} , \|\partial_x u\|_{L^{1}([0,T]\times \TT)} , 
		\|u\|_{W^{\gamma, \frac{2\nu}{2-\nu}}(0,T;W^{-3,\mu}(\TT))}\Bigr\}\, \le \tfrac{1}{\delta} 
		\Bigr\}.
		\]
		The set $K_\delta$ is bounded in $L^1(0,T; W^{1,1}(\TT))$ and thus compact in $L^1([0,T]\times \TT)$ by compactness of the embedding 
		\[
		L^1(0,T;W^{1,1}(\TT))\cap W^{\gamma, \frac{2\nu}{2-\nu}}(0,T;W^{-3,\mu}(\TT)) \,\hookrightarrow\, L^1(0,T;L^1(\TT)),
		\]
		see again \cite[Corollary 5]{Simon1986}. Since $K_\delta$ is moreover bounded in $L^{p_{l+1}}([0,T]\times \TT)$, it is compact in $L^{p_l}([0,T]\times \TT)$ by interpolation. Lastly, with the help of \eqref{Eq240}, \eqref{Eq251} and \eqref{Eq250}, we can conclude that $\PP(\{ 
		\mathbbm{1}_{A_j}u_\epsilon\notin K_\delta  \}) \to 0$ uniformly in $\epsilon$ as $\delta\searrow 0$,  analogously to \eqref{Eq78}.
		
		\textit{Tightness on $\XS$.} By Lemma \ref{Tightness_Proj_Limit} and Lemma \ref{Lemma_localized_tightness} it is again sufficient to verify that $(\mathbbm{1}_{A_j}u_\epsilon)_{\epsilon\in (0,1)}$ lies tight on $L^{r_l}(0,T;W^{1,r_l}(\TT))$, equipped with its weak topology for each $j,l\in \NN$. To this end, we define the set
		\[
		K_\delta\,=\, \bigl\{
		u\,\big|\,
		\|u\|_{L^{r_l}([0,T]\times \TT)} \le \tfrac{1}{\delta},
		\, \|\partial_x u\|_{L^{r_l}([0,T]\times \TT)} \le \tfrac{1}{\delta}
		\bigr\},
		\]
		which is bounded in $L^{r_l}(0,T;W^{1,r_l}(\TT))$ and consequently compact with respect to the weak topology by the Banach-Alaoglu theorem. Following the lines of \eqref{Eq78}, we conclude that  $\PP(\{ 
		\mathbbm{1}_{A_j}u_\epsilon\notin K_\delta  \}) \to 0$ uniformly in $\epsilon$ as $\delta\searrow 0$ by  \eqref{Eq240} and \eqref{Eq251}, which implies tightness.
		
		\textit{Tightness of the powers.} A last application of Lemma  \ref{Lemma_prod_tightness} and Lemma \ref{Lemma_localized_tightness} yields that it suffices to show tightness of $\bigl(\mathbbm{1}_{A_j}u_\epsilon^\frac{\alpha_l+n+1}{2}\bigr)_{\epsilon\in (0,1)}$ on $\Xp$ for all $j,l\in \NN$. To this end, we define $v_\epsilon= u_\epsilon^\frac{\alpha_l+n+1}{2}$ and arguing as for $w_\epsilon$ in the beginning of the proof of Lemma \ref{lemma_int_p}, we conclude that $v_\epsilon$ admits 
		\[
		\tfrac{(\alpha_l+n+1)(\alpha_l+n-1)}{4}u_\epsilon^{\frac{\alpha_l+n-3}{2}} (\partial_x u_\epsilon)^2\,+\, \tfrac{\alpha_l+n+1}{2}u_\epsilon^\frac{\alpha_l+n-1}{2}\partial_x^2u_\epsilon
		\]
		as a second weak derivative. In particular, using Parseval's relation we can estimate
		\begin{align}
			\|v_\epsilon\|_{H^2(\TT)}^2 \,&=\, \sum_{k\in \ZZ}  (1+(2\pi k)^2 +(2\pi k)^4) |\hat{v_\epsilon}(k)|^2\\&\lesssim\, |\hat{v_\epsilon}(0)|^2\,+\, \sum_{k\in \ZZ}  (2\pi k)^4|\hat{v_\epsilon}(k)|^2\,\le\, \|v_\epsilon\|_{L^2(\TT)}^2\,+\, \|\partial_x^2 v_\epsilon\|_{L^2(\TT)}^2, 	
		\end{align}
		such that
		\begin{align}\begin{split}
				\label{Eq215}
			\|v_\epsilon\|_{L^2(0,T; H^2(\TT))}^2&\lesssim\, \int_0^T
			\|v_\epsilon\|_{L^2(\TT)}^2\,+\, \|\partial_x^2 v_\epsilon\|_{L^2(\TT)}^2
			\, dt  \\ &\lesssim_{\alpha_l, n}\,
			\int_0^T \bigl\|u_\epsilon^\frac{\alpha_l +n+1}{2}\bigr\|_{L^2(\TT)}^2\,dt\,+\, 
			\int_0^T \int_{\TT}  u_\epsilon^{{\alpha_l+n-1}} (\partial_x^2 u_\epsilon)^2 \,+\,u_\epsilon^{{\alpha_l+n-3}} (\partial_x u_\epsilon)^4\,dx\,dt\\&=\,
			\|u_\epsilon\|_{L^{\alpha_l+n+1}([0,T]\times\TT)}^{\alpha_l+n+1}\,+\, 
			\int_0^T \int_{\TT} u_\epsilon^{{\alpha_l+n-1}} (\partial_x^2 u_\epsilon)^2 \,+\,u_\epsilon^{{\alpha_l+n-3}} (\partial_x u_\epsilon)^4  \,dx\,dt.
			\end{split}
		\end{align}
		Hence, by invoking \eqref{Eq253} and \eqref{Eq240} we obtain
		\begin{align}&
			\EE\bigl[
			\mathbbm{1}_{A_j} \|v_\epsilon\|_{L^2(0,T; H^2(\TT))}^2
			\bigr] \\&\quad
			\lesssim_{\alpha_l, n} \,\EE\bigl[
			\mathbbm{1}_{A_j} \|u_\epsilon\|_{L^{\alpha_l+n+1}([0,T]\times\TT)}^{\alpha_l+n+1}
			\bigr]\,+\, \EE\biggl[
			\mathbbm{1}_{A_j} \int_0^T \int_{\TT} u_\epsilon^{{\alpha_l+n-1}} (\partial_x^2 u_\epsilon)^2 \,+\,u_\epsilon^{{\alpha_l+n-3}} (\partial_x u_\epsilon)^4  \,dx\,dt
			\biggr]\\&\quad\lesssim_{\alpha_l, n,\Lambda,T} \,(j+1)^{\alpha_l+n+1}.
		\end{align}
	Tightness of $(\mathbbm{1}_{A_j}v_\epsilon)_{\epsilon\in (0,1)}$ in $\Xp$ follows by Chebychev's inequality and the Banach-Alaoglu theorem.
	\end{proof}
	Since the sequence \eqref{Eq77} lies tight on the space \eqref{Eq79}, we can extract an equidistributed convergent subsequence.  
	\begin{cor}\label{Cor_SS}
		There exists a complete probability space $(\tilde{\Omega}, \tilde{\mathfrak{A}}, \tilde{\PP})$ and an equidistributed subsequence
		\begin{equation}\label{Eq81}
			\Bigl(\Bigl( (\tilde{\chi}_\epsilon^{(j)})_{j\in \NN}, (\tilde{\beta}_\epsilon^{(k)})_{k\in \mathbb{Z}},\tilde{ u}_\epsilon,\tilde{ u}_\epsilon, \tilde{u}_\epsilon, \bigl(\tilde{u}_\epsilon^{\frac{\alpha_l+n+1}{2}}\bigr)_{l\in \NN}\Bigr)\Bigr)_{\epsilon}
		\end{equation}
		of \eqref{Eq77}, consisting of $\mathfrak{B}(\X)$-measurable, $\X$-valued random variables
		and a $\mathfrak{B}(\X)$-measurable, $\X$-valued random variable
		\begin{equation}\label{Eq82}
			\Bigl(
			(\tilde{\chi}^{(j)})_{j\in \NN},
			(\tilde{\beta}^{(k)})_{k\in \mathbb{Z}}, \tilde{u}, \tilde{u}, \tilde{u}, \bigl(\tilde{u}^{\frac{\alpha_l+n+1}{2}}\bigr)_{l\in \NN}\Bigr),
		\end{equation}
		defined 
		on $(\tilde{\Omega}, \tilde{\mathfrak{A}}, \tilde{\PP})$  such that
		\begin{equation}\label{Eq83}\Bigl(
			(\tilde{\chi}_\epsilon^{(j)})_{j\in \NN}, (\tilde{\beta}_\epsilon^{(k)})_{k\in \ZZ}, \tilde{u}_\epsilon, \tilde{u}_\epsilon,\tilde{u}_\epsilon, \bigl(\tilde{u}_\epsilon^{\frac{\alpha_l+n+1}{2}}\bigr)_{l\in \NN}\Bigr)\,\to \,\Bigl( (\tilde{\chi}^{(j)})_{j\in \NN},
			(\tilde{\beta}^{(k)})_{k\in \ZZ}, \tilde{u}, \tilde{u}, \tilde{u}, \bigl(\tilde{u}^{\frac{\alpha_l+n+1}{2}}\bigr)_{l\in \NN}\Bigr)
		\end{equation}
		$\tilde{\PP}$-almost surely in $\X$,
		as $\epsilon\searrow 0$.
	\end{cor}
	\begin{proof}
		The existence of an equidistributed subsequence follows from Lemma \ref{Lemma_tightness}, when we can verify the technical assumption of \cite[Theorem 2]{jakub98}, i.e. that $\X$ admits a countable family of continuous functions separating its points and $\mathfrak{A}$-measurability of these functions, when composed with \eqref{Eq77}. To construct such a family of functions, separating the points in $\X$, we can project on a component of \eqref{Eq79} and then apply functions separating the points in this component, such that we can consider the spaces $\Xind$, $\XBM$,  $\Xc$, $\XL$, $\XS$ and $\Xp$  individually. For $\Xc$, we let $\rho\colon \RR\to [-1,1]$ be a continuous injection and take  the functions $
		\Xc \to [-1,1], u\mapsto \rho (\langle u(t),f_k\rangle )$ for $t\in [0,T]\cap \mathbbm{Q}$ and $k\in \ZZ$, where we recall that $f_k$ was defined by \eqref{Eq80}. For $\XBM$ the same construction applies and for $\Xind$ the function $\rho$ itself separates the points. Since  the spaces $\XL$, $\XS$ and $\Xp$ embed into $L^2([0,T]\times \TT)$ with its weak topology, the family $u\mapsto \rho( \langle u, f_j(T^{-1}\cdot ) \otimes f_k \rangle)$, $j,k\in \ZZ$ separates the points in them. 
		
		To also check $\mathfrak{A}$-measurability of these functions composed with \eqref{Eq77}, we first observe that  
		$\rho(\mathbbm{1}_{A_j})$ and
		$\rho(\beta^{(k)}(t))$ are random variables in $\RR$ for every $j\in \NN$ and $k\in \ZZ$. For the other cases, we note that $u_\epsilon$ is adapted in $H^1(\TT)$ such that $u_\epsilon(t,x)$ is $\mathfrak{A}$-measurable for each $(t,x)\in [0,T]\times \TT$ and thus also $u_\epsilon$ as a random variable in $C([0,T]\times \TT)$. In particular, the compositions
		\[
		\rho (\langle u_\epsilon(t),f_k\rangle ),\quad \rho( \langle u_\epsilon, f_j(T^{-1}\cdot ) \otimes f_k \rangle)
		\]
		are $\mathfrak{A}$-measurable, too. Lastly, we use that 
		\[
		C([0,T]\times \TT)\to 
		C([0,T]\times \TT),\,u\mapsto u^\frac{\alpha_l+n+1}{2}
		\]
		is continuous, to conclude  that \[
		\rho\bigl(\bigl \langle u_\epsilon^\frac{\alpha_l+n+1}{2}, f_j(T^{-1}\cdot ) \otimes f_k \bigr\rangle\bigr)
		\] is $\mathfrak{A}$-measurable. 
		
		Hence, \cite[Theorem 2]{jakub98} is indeed applicable and there exists an equidistributed, convergent subsequence \eqref{Eq81} of \eqref{Eq77}, which converges almost surely to a random variable
		\[
		\bigl( (\tilde{\chi}^{(j)})_{j\in \NN}, (\tilde{\beta}^{(k)})_{k\in \mathbb{Z}}, \tilde{u}, \tilde{f}, \tilde{g}, (\tilde{v}_l)_{l\in \NN}\bigr)
		\]
		in $\X$. Since $\tilde{u}_\epsilon$ converges almost surely to $\tilde{u}$, $\tilde{f}$ and $\tilde{g}$ in the space of distributions on $(0,T)\times \TT$, it holds $\tilde{u}=\tilde{f}=\tilde{g}$. Moreover, since $\tilde{u}_\epsilon\to \tilde{u}$ in $\XL$ we have that 
		$\tilde{u}_\epsilon^\frac{\alpha_l+n+1}{2}\to \tilde{u}^\frac{\alpha_l+n+1}{2}$ in $L^2([0,T]\times\TT)$ by Vitali's convergence theorem. With the help of  $\tilde{u}_\epsilon^{\frac{\alpha_l+n+1}{2}}\to\tilde{v}_l$ in $\Xp$, we conclude  $\tilde{v}_l=\tilde{u}^\frac{\alpha_l+n+1}{2}$, which finishes the proof.
	\end{proof}
	\subsection{Convergence to a solution and a-priori estimates}\label{Sec_convergence} 
	For the rest of this article, we consider the complete probability space $(\tilde{\Omega}, \tilde{\mathfrak{A}}, \tilde{\PP})$ with the random variables \eqref{Eq81} and \eqref{Eq82} obtained in Corollary \ref{Cor_SS}. We introduce the filtration $\tilde{\mathfrak{F}}$ on $(\tilde{\Omega}, \tilde{\mathfrak{A}}, \tilde{\PP})$ as the augmentation of $\tilde{\mathfrak{G}}$, defined by
	\[
	\tilde{\mathfrak{G}}_t\,=\, \sigma \bigl\{ \tilde{\chi}^{(j)}, 
	\tilde{u}(s), \tilde{\beta}^{(k)}(s)\,|\, 
	j\in \NN,\,0\le s\le t,\, k\in \mathbb{Z}
	\bigr\}.
	\]
	Here, we consider $\tilde{u}(s)$ as a random element in $H^{-1}(\TT)$. 
	\begin{lemma}\label{Lemma_BM}
		The family $(\tilde{\beta}^{(k)})_{k\in \ZZ}$ is a family of independent $\tilde{\mathfrak{F}}$-Brownian motions.	
	\end{lemma}
	\begin{proof}The proof is standard and we refer to \cite[Lemma 5.7]{fischer_gruen_2018}.
	\end{proof}
	\begin{lemma}\label{Lemma_new_partition}
		The family $(\tilde{\chi}^{(j)})_{j\in \NN}$ consists of indicator functions of a measurable partition $(\tilde{A}_j)_{j\in \NN}$ of $\tilde{\Omega}$.
	\end{lemma}
	\begin{proof}
		The family $(\tilde{\chi}_{\epsilon}^{(j)})_{j\in \NN}\sim (\mathbbm{1}_{A_j})_{j\in \NN}$ converges to $(\tilde{\chi}_j)_{j\in \NN}$ almost surely and thus in law by \cite[Proposition 9.3.5]{Dudley_2002}. Hence, the law of $(\tilde{\chi}_j)_{j\in \NN}$ coincides with the law of $(\mathbbm{1}_{A_j})_{j\in \NN}$ and the claim follows.
	\end{proof}
	Additionally to the convergences \eqref{Eq83}, we obtain also strong convergence for the powers of $\tilde{u}_\epsilon$.
	\begin{lemma}\label{Lemma_additional_convergence}
		For each $l\in \NN$ one has 
		\begin{equation}\label{Eq84}
			\tilde{u}_\epsilon^\frac{\alpha_l+n+1}{2} \,\to \, 
			\tilde{u}^\frac{\alpha_l+n+1}{2} 
		\end{equation}
		almost surely in $L^2(0,T; H^1(\TT))$, as $\epsilon\searrow 0$.
	\end{lemma}
	\begin{proof}At the end of the proof of Corollary \ref{Cor_SS}, we showed that $\tilde{u}_\epsilon^\frac{\alpha_l+n+1}{2} \to \tilde{u}^\frac{\alpha_l+n+1}{2}$ almost surely in $L^2([0,T]\times \TT)$.
		Hence, the claim follows by the weak convergence  $\tilde{u}_\epsilon^\frac{\alpha_l+n+1}{2} \to \tilde{u}^\frac{\alpha_l+n+1}{2}$ in $\Xp$ and interpolation.
	\end{proof}
	\begin{lemma}\label{Lemma_Limits}
		For each $\vp\in C^\infty(\TT)$ and $t\in [0,T]$, the following holds almost surely  as $\epsilon\searrow 0$.\begin{enumerate}[label=(\roman*)]
			\item 
			\label{Item1}	\[\langle \tilde{u}_{\epsilon}(t), \vp\rangle\,\to\, \langle \tilde{u}(t), \vp\rangle,\]
			\item
			\label{Item2}	\[
			\int_0^t\langle
			\tilde{u}_\epsilon^{n-2} (\partial_x \tilde{u}_{\epsilon})^3, \partial_x\vp\rangle\, ds \,\to \,
			\int_0^t \langle
			\tilde{u}^{n-2} (\partial_x \tilde{u})^3, \partial_x\vp\rangle\, ds,
			\]
			\item \label{Item3}
			\[\int_0^t\langle \tilde{u}_\epsilon^{n-1}(\partial_x\tilde{u}_{\epsilon})^2 ,\partial_x^2\vp\rangle\, ds\,\to\,\int_0^t\langle
			\tilde{u}^{n-1} (\partial_x \tilde{u})^2, \partial_x\vp\rangle\, ds,
			\]
			\item \label{Item4}
			\[
			\int_0^t\langle \tilde{u}_\epsilon^{n}\partial_x \tilde{u}_{\epsilon }, \partial_x^3\vp \rangle\,ds	\,\to \, 
			\int_0^t\langle \tilde{u}^{n}\partial_x \tilde{u} , \partial_x^3\vp \rangle\,ds,
			\]
			\item \label{Item5}
			\[
			\sum_{k\in \ZZ}\int_0^t \langle
			\sigma_{k,\epsilon} q'(\tilde{u}_\epsilon) \partial_x (\sigma_{k,\epsilon} q(\tilde{u}_\epsilon)) ,\partial_x \vp\rangle\, ds\,\to \, 
			\sum_{k\in \ZZ}\int_0^t \langle
			\sigma_k q'(\tilde{u}) \partial_x (\sigma_k q(\tilde{u})) ,\partial_x \vp\rangle\, ds,
			\]
			\item  \label{Item6}
			\[
			\sum_{k\in \ZZ}  \int_0^t 
			\langle
			\sigma_{k,\epsilon} q(\tilde{u}_\epsilon), \partial_x \vp
			\rangle^2
			\,ds \,\to \, 
			\sum_{k\in \ZZ}  \int_0^t 
			\langle
			\sigma_{k} q(\tilde{u}), \partial_x \vp
			\rangle^2
			\,ds,
			\]
			\item \label{Item7}
			\[\int_0^t 
			\langle
			\sigma_{k,\epsilon} q(\tilde{u}_\epsilon), \partial_x \vp
			\rangle
			\,ds
			\,\to \,\int_0^t 
			\langle
			\sigma_{k} q(\tilde{u}), \partial_x \vp
			\rangle
			\,ds.
			\]
		\end{enumerate}
	\end{lemma}
	\begin{proof}We fix $\tilde{\omega} \in \tilde{\Omega}$ such that \eqref{Eq83} holds in $\X$ and \eqref{Eq84} holds in $L^2(0,T; H^1(\TT))$.
		The convergence \ref{Item1} follows from $\tilde{u}_{\epsilon}\to \tilde{u}$ in $\Xc$. For \ref{Item2}, \ref{Item3}, \ref{Item4} and \ref{Item5}, we consider an arbitrary subsequence such that it suffices to show that for another subsequence the respective convergence holds. By  $ \tilde{u}_\epsilon\to \tilde{u}$ in $\XL$ and $\tilde{u}_\epsilon^\frac{\alpha_l+n+1}{2} \,\to \, 
		\tilde{u}^\frac{\alpha_l+n+1}{2}$ in $L^2(0,T; H^1(\TT))$, we can choose this subsequence such that $\tilde{u}_\epsilon\to\tilde{u}$ and $\partial_x \tilde{u}_\epsilon^\frac{\alpha_l+n+1}{2} \,\to \, 
		\partial_x \tilde{u}^\frac{\alpha_l+n+1}{2} $  almost everywhere on $[0,T]\times \TT$ and therefore $\partial_x \tilde{u}_\epsilon \,\to \,\partial_x \tilde{u} $ on the set $\{\tilde{u}>0\}$. 
		
		We verify \ref{Item2}  by showing separately that 	
		\begin{align}
			&\label{Eq85}
			\int_0^t\int_{\{\tilde{u}=0\}}
			\tilde{u}_\epsilon^{n-2} (\partial_x \tilde{u}_{\epsilon})^3\partial_x\vp
			\, dx\, ds \,\to \, 0,\\& \label{Eq86}
			\int_0^t\int_{\{\tilde{u}>0\}}
			\tilde{u}_\epsilon^{n-2} (\partial_x \tilde{u}_{\epsilon })^3\partial_x\vp
			\, dx\, ds \,\to \, \int_0^t\int_{\{\tilde{u}>0\}}
			\tilde{u}^{n-2} (\partial_x \tilde{u})^3\partial_x\vp
			\, dx\, ds.
		\end{align} 
		Using that $\tilde{u}_\epsilon \to \tilde{u}$ in $\XL$ and $\XS$, \eqref{Eq85} follows by H\"older's inequality and \eqref{Eq86} by Vitali's convergence theorem. The claims \ref{Item3} and \ref{Item4} can be shown analogously. 
		
		For \ref{Item5} we  rewrite 
		\begin{align}&
			\sum_{k\in \ZZ}\int_0^t \langle
			\sigma_{k,\epsilon} q'(\tilde{u}_\epsilon) \partial_x (\sigma_{k,\epsilon} q(\tilde{u}_\epsilon)) ,\partial_x \vp\rangle\, ds\\&\quad =\, 
			\sum_{k\in \ZZ}\int_0^t \langle
			\sigma_{k,\epsilon}^2(q'(\tilde{u}_\epsilon))^2 \partial_x \tilde{u}_\epsilon  ,\partial_x \vp\rangle\,+\,
			\langle
			\sigma_{k,\epsilon}\partial_x \sigma_{k,\epsilon} q(\tilde{u}_\epsilon)q'(\tilde{u}_\epsilon), \partial_x \vp \rangle
			\, ds
		\end{align}
		and observe that for the individual summands
		\begin{align}&
			\tfrac{n^2}{4}
			\int_0^t \langle  \sigma_{k,\epsilon}^2
			\tilde{u}_\epsilon^{n-2} \partial_x \tilde{u}_\epsilon,\partial_x \vp\rangle\, ds \,\to\, 
			\tfrac{n^2}{4}
			\int_0^t \langle  \sigma_{k}^2
			\tilde{u}^{n-2} \partial_x \tilde{u},\partial_x \vp\rangle\, ds ,\\&
			\tfrac{n}{2}
			\int_0^t \langle
			\sigma_{k,\epsilon}\partial_x \sigma_{k,\epsilon} \tilde{u}_\epsilon^{n-1} , \partial_x \vp\rangle \, ds \,\to \,
			\tfrac{n}{2} \int_0^t \langle\sigma_k\partial_x \sigma_k \tilde{u}^{n-1} , \partial_x \vp\rangle\, ds,
		\end{align}
		as $\epsilon\searrow 0$, since $\tilde{u}_\epsilon \to \tilde{u}$ in $\XL$, $\XS$ and $\sigma_{k,\epsilon} \to \sigma_k$ in $C^1(\TT)$ for fixed $k\in \ZZ$ by \eqref{Eq_Def_sigma}, \eqref{Eq87} and \eqref{Eq88}.  Hence, by the dominated convergence theorem, it suffices to find a summable, dominating sequence of
		\[
		\biggl( 
		\int_0^t |\langle  \sigma_{k,\epsilon}^2
		\tilde{u}_\epsilon^{n-2} \partial_x \tilde{u}_\epsilon,\partial_x \vp\rangle| \,+\, |
		\langle
		\sigma_{k,\epsilon}\partial_x \sigma_{k,\epsilon} \tilde{u}_\epsilon^{n-1} , \partial_x \vp\rangle|
		\, ds\biggr)_{k\in \ZZ}
		\]
		independent of $\epsilon$ to conclude \ref{Item5}.
		To this end, we estimate using H\"older's inequality
		\begin{align}&
			\int_0^t 
			|\langle \sigma_{k,\epsilon}^2
			\tilde{u}_\epsilon^{n-2} \partial_x \tilde{u}_\epsilon,\partial_x \vp\rangle|\,+\, | \langle
			\sigma_{k,\epsilon}\partial_x \sigma_{k,\epsilon} \tilde{u}_\epsilon^{n-1} , \partial_x \vp\rangle| \, ds\\&\quad\le\,
			\|\partial_x \vp\|_{L^\infty(\TT)}  \int_0^t \bigl(\| \sigma_{k,\epsilon}\|_{L^\infty(\TT)}^2
			\|\tilde{u}_\epsilon\|_{L^{2(n-2)}(\TT)}^{n-2} \|\partial_x \tilde{u}_\epsilon\|_{L^{2}(\TT)}\,+\,  \|\sigma_{k}\|_{L^\infty(\TT)}
			\|\partial_x \sigma_{k}\|_{L^\infty(\TT)}
			\|\tilde{u}_\epsilon\|_{L^{n-1}(\TT)}^{n-1} \bigr)\, ds
			\\&\quad \le\,
			\|\partial_x \vp\|_{L^\infty(\TT)}  \bigl(
			\|\tilde{u}_\epsilon\|_{L^{2(n-2)}([0,T]\times\TT)}^{n-2} \|\partial_x \tilde{u}_\epsilon\|_{L^2([0,T]\times \TT)}\,+\, 	\|\tilde{u}_\epsilon\|_{L^{n-1}([0,T]\times\TT)}^{n-1}
			\bigr)\|\sigma_k\|_{C^1(\TT)}^2.
		\end{align}
		The prefactor on the right-hand side is uniformly bounded in $\epsilon$, since $\tilde{u}_\epsilon \to \tilde{u}$ in $\XL$, $\XS$ and the sequence $(\|\sigma_k\|_{C^1(\TT)}^2)_{k\in \ZZ}$ is summable by \eqref{Consequence_sigma_New}, which finishes the proof of \ref{Item5}. 
		For \ref{Item6}, we proceed analogously and observe first that for each summand
		\begin{align}&
			\biggl|\int_0^t 
			\langle
			\sigma_{k,\epsilon} q(\tilde{u}_\epsilon), \partial_x \vp
			\rangle^2 -
			\langle
			\sigma_{k} q(\tilde{u}), \partial_x \vp
			\rangle^2
			\,ds\biggr|
			\\&\quad\le \,
			\int_0^t |
			\langle
			\sigma_{k,\epsilon} q(\tilde{u}_\epsilon), \partial_x \vp
			\rangle -
			\langle
			\sigma_{k} q(\tilde{u}), \partial_x \vp
			\rangle|\,
			|
			\langle
			\sigma_{k,\epsilon} q(\tilde{u}_\epsilon), \partial_x \vp
			\rangle +  
			\langle
			\sigma_{k} q(\tilde{u}), \partial_x \vp
			\rangle|
			\,ds
			\\&\quad\le \,
			\|\partial_x \vp\|_{L^2(\TT)}^2
			\int_0^t \| 
			\sigma_{k,\epsilon} q(\tilde{u}_\epsilon)- 
			\sigma_{k} q(\tilde{u})
			\|_{L^2(\TT)}
			\| 
			\sigma_{k,\epsilon} q(\tilde{u}_\epsilon)+
			\sigma_{k} q(\tilde{u})
			\|_{L^2(\TT)}\, ds
			\\&\quad \le \,
			\|\partial_x \vp\|_{L^2(\TT)}^2
			\| 
			\sigma_{k,\epsilon} q(\tilde{u}_\epsilon)- 
			\sigma_{k} q(\tilde{u})
			\|_{L^2([0,T]\times\TT)}
			\| 
			\sigma_{k,\epsilon} q(\tilde{u}_\epsilon)+
			\sigma_{k} q(\tilde{u})
			\|_{L^2([0,T]\times\TT)},
		\end{align}
		which tends to $0$ as $\epsilon\searrow 0$, since $\tilde{u}_\epsilon\to \tilde{u}$ in $\XL$ and $\sigma_{k,\epsilon} \to \sigma_k$ in $L^\infty(\TT)$. We use again the dominated convergence theorem together with
		\begin{align} \int_0^t 
			\langle
			\sigma_{k,\epsilon} q(\tilde{u}_\epsilon), \partial_x \vp
			\rangle^2
			\,ds\,  &\le \,  \int_0^t 
			\| \sigma_{k,\epsilon} \|_{L^\infty(\TT)}^2  \| \tilde{u}_\epsilon \|_{L^n(\TT)}^n \| \partial_x \vp \|_{L^2(\TT)}^2
			\,ds
			\\&\le \,
			\| \partial_x \vp \|_{L^2(\TT)}^2
			\| \tilde{u}_\epsilon \|_{L^n([0,T]\times\TT)}^n 	\| \sigma_{k} \|_{L^\infty(\TT)}^2,
		\end{align}
		the convergence $\tilde{u}_\epsilon\to \tilde{u}$ in $\XL$ and \eqref{Consequence_sigma_New} to conclude \ref{Item6}.
		The convergence \ref{Item7} follows from $\tilde{u}_\epsilon \to \tilde{u}$ in $\XL$ and $\sigma_{k,\epsilon} \to \sigma_k$ in $L^\infty(\TT)$.
	\end{proof}
	\begin{lemma}\label{Lemma_SPDE}
		For every $\vp\in C^\infty(\TT)$, $j\in \NN$ and $t\in [0,T]$ we have that
		\begin{align}\begin{split}
				\label{Eq89}  \mathbbm{1}_{\tilde{A}_j}\left[
				\langle\tilde{u}(t), \vp\rangle\,-\, \langle\tilde{u}(0), \vp\rangle\right]\,=\, \mathbbm{1}_{\tilde{A}_j}\biggl[& \tfrac{n(n-1)}{2}\int_0^t \langle
				\tilde{u}^{n-2} (\partial_x \tilde{u})^3, \partial_x\vp\rangle\, ds
				\\&+\,
				\tfrac{3n}{2} \int_0^t\langle \tilde{u}^{n-1}(\partial_x\tilde{u})^2, \partial_x^2\vp\rangle\, ds\,+\,
				\int_0^t\langle \tilde{u}^{n}\partial_x \tilde{u}, \partial_x^3\vp \rangle\,ds
				\\&-\,\tfrac{1}{2}\sum_{k\in \ZZ}\int_0^t \langle
				\sigma_k q'(\tilde{u}) \partial_x (\sigma_k q(\tilde{u})) ,\partial_x \vp\rangle\, ds
				\\&-\,\sum_{k\in \ZZ} \int_0^t 
				\langle
				\sigma_k q(\tilde{u}), \partial_x \vp
				\rangle
				\, d\tilde{\beta}_s^{(k)}\biggr].
			\end{split}
		\end{align}
	\end{lemma}
	\begin{proof} 
		Throughout this proof, we fix $j\in \NN$ and $\vp\in C^\infty(\TT)$ and define the process
		\begin{align}
			\tilde{M}(t)\,=\, \mathbbm{1}_{\tilde{A}_j}\biggl[&	\langle\tilde{u}(t), \vp\rangle\,-\, \langle\tilde{u}(0), \vp\rangle\,-\, \tfrac{n(n-1)}{2}\int_0^t \langle
			\tilde{u}^{n-2} (\partial_x \tilde{u})^3, \partial_x\vp\rangle\, ds\\&
			-\,\tfrac{3n}{2} \int_0^t \langle \tilde{u}^{n-1}(\partial_x\tilde{u})^2, \partial_x^2\vp\rangle\, ds\,-\,
			\int_0^t\langle \tilde{u}^{n}\partial_x \tilde{u}, \partial_x^3\vp \rangle\,ds\\&+\,\tfrac{1}{2}\sum_{k\in \ZZ}\int_0^t \langle
			\sigma_k q'(\tilde{u}) \partial_x (\sigma_k q(\tilde{u})) ,\partial_x \vp\rangle\, ds\biggr]
		\end{align}
		and the approximating processes
		\begin{align}\begin{split}
				\label{Eq90}
				\tilde{M}_\epsilon(t)\,=\,\tilde{\chi}_\epsilon^{(j)} \biggl[& 	\langle\tilde{u}_\epsilon(t), \vp\rangle\,-\, \langle\tilde{u}_\epsilon(0), \vp\rangle\,-\, \tfrac{n(n-1)}{2}\int_0^t \langle
				\tilde{u}_\epsilon^{n-2} (\partial_x \tilde{u}_\epsilon)^3, \partial_x\vp\rangle \, ds\\&
				-\,
				\tfrac{3n}{2} \int_0^t\langle \tilde{u}_\epsilon^{n-1}(\partial_x\tilde{u}_\epsilon)^2, \partial_x^2\vp\rangle\, ds\,-\,
				\int_0^t\langle \tilde{u}_\epsilon^{n}\partial_x \tilde{u}_\epsilon, \partial_x^3\vp \rangle\,ds\\&+\,\tfrac{1}{2}\sum_{k\in \ZZ}\int_0^t \langle
				\sigma_{k,\epsilon} q'(\tilde{u}_\epsilon) \partial_x (\sigma_{k,\epsilon} q(\tilde{u}_\epsilon)) ,\partial_x \vp\rangle\, ds\biggr].
			\end{split}
		\end{align} As a consequence of Lemma \ref{Lemma_Limits} we have indeed $\tilde{M}_\epsilon(t)\to \tilde{M}(t)$ as $\epsilon\searrow 0$. Now let 
		\begin{equation}\label{Eq254}
		\phi \colon\,  \prod_{j=1}^{\infty} \RR \times C([0,s]; H^{-1}(\TT)) \times \prod_{k\in \ZZ} C([0,s]) \,\to\, \RR
		\end{equation}
		be continuous and bounded and define 
		\begin{align}&
			\tilde{\rho} \,=\, \phi\bigl(( \tilde{\chi}^{(j)})_{j\in \NN}, \tilde{u}, (\tilde{\beta}^{(k)})_{k\in \ZZ}  \bigr),\\&
			\tilde{\rho}_\epsilon \,=\, \phi\bigl((\tilde{\chi}_\epsilon^{(j)})_{j\in \NN}, \tilde{u}_\epsilon, (\tilde{\beta}_\epsilon^{(k)})_{k\in \ZZ} \bigr),
		\end{align}
		such that $\tilde{\rho}_\epsilon\to \tilde{\rho}$ as $\epsilon\searrow 0$ by \eqref{Eq83}. Defining $M_\epsilon$ on the original probability space $(\Omega, \mathfrak{A}, \PP)$ as the right-hand side of \eqref{Eq90} with $\tilde{u}_\epsilon$ replaced by $u_\epsilon$, we find that
		\[
		M_\epsilon(t)\,=\, - \mathbbm{1}_{A_j}\int_0^t\langle
		\Phi_\epsilon d\beta_s, \partial_x \vp
		\rangle,
		\]
		where $\Phi_\epsilon$ and $\beta$ are defined as in the proof of Lemma \ref{Lemma_reg_stoch_int_New}, because of Consequence \ref{Consequences_Ass_1} \ref{Item_Consequences_6} and  \eqref{Eq39}.
		The quadratic variation process of $M_\epsilon$ is given by 
		\begin{align}\begin{split}\label{Eq91}
				\mathbbm{1}_{A_j}\sum_{k\in \ZZ} \int_0^t \langle \sigma_{k,\epsilon} q(u_\epsilon), \partial_x \vp \rangle^2\, ds\, &\le \,
				\mathbbm{1}_{A_j}\sum_{k\in \ZZ} \int_0^t \|\sigma_{k,\epsilon}\|_{L^\infty(\TT)}^2 \|u_\epsilon \|_{L^n(\TT)}^n \|\partial_x \vp\|_{L^2(\TT)}^2\, ds
				\\&\lesssim_\Lambda\, \mathbbm{1}_{A_j} \|\partial_x \vp \|_{L^2(\TT)}^2 \|u_\epsilon\|_{L^n([0,T]\times \TT)}^n,
			\end{split}
		\end{align}where we used \eqref{Consequence_sigma_New} in the second inequality,
		which is integrable by Consequence \ref{Consequences_Ass_1} \ref{Item_Consequences_5} and thus $M_\epsilon$ is a square integrable martingale.	
		In particular, since $(\tilde{M}_\epsilon, (\tilde{\chi}_\epsilon^{(j)})_{j\in \NN}, \tilde{u}_\epsilon, (\tilde{\beta}_\epsilon^{(k)})_{k\in \ZZ} )$ has the same distribution as $({M}_\epsilon, (\mathbbm{1}_{A_j})_{j\in \NN}, {u}_\epsilon, ({\beta}^{(k)})_{k\in \ZZ} )$ by Corollary \ref{Cor_SS}, we obtain that
		\begin{align}\begin{split}
				\label{Eq93}&
				\tilde{\EE}\bigl[
				(\tilde{M}_\epsilon(t)-\tilde{M}_\epsilon(s)) \tilde{\rho}_\epsilon
				\bigr]\,=\, 0,\\&
				\tilde{\EE}\biggl[
				\biggl(\tilde{M}_\epsilon^2(t) -
				\tilde{M}_\epsilon^2(s) - \tilde{\chi}_\epsilon^{(j)} \sum_{k\in \ZZ}  \int_s^t 
				\langle
				\sigma_{k,\epsilon} q(\tilde{u}_\epsilon), \partial_x \vp
				\rangle^2
				\,d\tau  \biggr)\tilde{\rho}_\epsilon
				\biggr]\,=\, 0,\\&
				\tilde{\EE}\biggl[
				\biggl(\tilde{M}_\epsilon(t)\tilde{\beta}_\epsilon^{(k)}(t) -
				\tilde{M}_\epsilon(s)\tilde{\beta}_\epsilon^{(k)}(s) - \tilde{\chi}_\epsilon^{(j)} \int_s^t 
				\langle
				\sigma_{k,\epsilon} q(\tilde{u}_\epsilon), \partial_x \vp
				\rangle
				\,d\tau  \biggr)\tilde{\rho}_\epsilon
				\biggr]\,=\, 0. 
			\end{split}
		\end{align}
		Next, we note that by \eqref{Eq91} and the Burkholder-Davis-Gundy inequality
		\begin{align}\label{Eq94}
			\tilde{\EE} \Bigl[
			\sup_{0\le t\le T} |\tilde{M}_\epsilon(t)|^\rho
			\Bigr]\,=\,
			{\EE} \Bigl[
			\sup_{0\le t\le T} |{M}_\epsilon(t)|^\rho
			\Bigr]
			\,\lesssim_{\rho,\Lambda}\, 
			\|\partial_x \vp \|_{L^2(\TT)}^\rho \EE\Bigl[
			\mathbbm{1}_{A_j} \|u_\epsilon\|_{L^n([0,T]\times \TT)}^\frac{\rho n}{2}
			\Bigr],
		\end{align}
		which is uniformly in $\epsilon$ bounded by \eqref{Eq240} for $\rho\in (2,\frac{14}{n})$. Using \eqref{Eq91} and \eqref{Eq240} again, we also obtain that
		\begin{equation}\label{Eq95}
			\tilde\EE\biggl[\biggl(  \tilde{\chi}_\epsilon^{(j)} \sum_{k\in \ZZ}
			\int_0^T
			\langle
			\sigma_{k,\epsilon} q(\tilde{u}_\epsilon), \partial_x \vp
			\rangle^2
			\,dt\biggr)^\frac{\rho}{2}
			\biggr]
		\end{equation}
		and more simply
		\begin{align}\begin{split}&
				\label{Eq96}
				\tilde{\EE}\biggl[
				\biggl(  \tilde{\chi}_\epsilon^{(j)} \int_0^T |
				\langle
				\sigma_{k,\epsilon} q(\tilde{u}_\epsilon), \partial_x \vp
				\rangle|
				\,dt \biggr)^\rho
				\biggr] \\&\quad \le\,\bigl( 
				\|\sigma_{k,\epsilon}\|_{L^\infty(\TT)} \|\partial_x \vp\|_{L^\infty(\TT)}\bigr)^\rho\EE\Bigl[
				\mathbbm{1}_{A_j} \|u_\epsilon\|_{L^\frac{n}{2}([0,T]\times \TT)}^\frac{\rho n}{2}
				\Bigr] 
			\end{split}
		\end{align} 
		are uniformly in $\epsilon$ bounded for these $\rho$. Since also 
		\begin{equation}\label{Eq97}
			\tilde{\EE}\Bigl[
			\sup_{0\le t\le T} |\tilde{\beta}_\epsilon^{(k)}(t)|^\upsilon
			\Bigr]\,=\, \EE\Bigl[
			\sup_{0\le t\le T} |{\beta}^{(0)}(t)|^\upsilon
			\Bigr]\,<\, \infty
		\end{equation}
		for all $\epsilon$ and $\upsilon\in (1,\infty)$, we can use uniform integrability of the random variables in \eqref{Eq93} and the almost sure convergences $\tilde{M}_\epsilon(t)\to \tilde{M}(t)$, $\tilde{\rho}_\epsilon\to \tilde{\rho}$, Lemma \ref{Lemma_Limits} \ref{Item6} and \ref{Item7}, $\tilde{\beta}_\epsilon^{(k)}\to \tilde{\beta}^{(k)}$ in $\XBM$  as well as  $\tilde{\chi}_\epsilon^{(j)} \to \mathbbm{1}_{\tilde{A}_j}$ in $\Xind$ as $\epsilon\searrow 0$, to conclude 
		\begin{align}\begin{split}\label{Eq92}
				&
				\tilde{\EE}\bigl[
				(\tilde{M}(t)-\tilde{M}(s)) \tilde{\rho}
				\bigr]\,=\, 0,\\&
				\tilde{\EE}\biggl[
				\biggl(\tilde{M}^2(t) -
				\tilde{M}^2(s) - \mathbbm{1}_{\tilde{A}_j}\sum_{k\in \ZZ}  \int_s^t 
				\langle
				\sigma_{k} q(\tilde{u}), \partial_x \vp
				\rangle^2
				\,d\tau  \biggr)\tilde{\rho}
				\biggr]\,=\, 0,\\&
				\tilde{\EE}\biggl[
				\biggl(\tilde{M}(t)\tilde{\beta}^{(k)}(t) -
				\tilde{M}(s)\tilde{\beta}^{(k)}(s) - \mathbbm{1}_{\tilde{A}_j}\int_s^t 
				\langle
				\sigma_{k} q(\tilde{u}), \partial_x \vp
				\rangle
				\,d\tau  \biggr)\tilde{\rho}
				\biggr]\,=\, 0. 
			\end{split}
		\end{align}
		An application of the monotone class theorem \cite[Theorem 2.12.9]{Bogachev_2007} yields that \eqref{Eq92} holds for any $\tilde{\rho}$, which is bounded and measurable with respect to the sigma-field generated by random variables of the form
		\[\phi\bigl(( \tilde{\chi}^{(j)})_{j\in \NN}, \tilde{u}, (\tilde{\beta}^{(k)})_{k\in \ZZ}  \bigr)\]
		with $\phi$ as in \eqref{Eq254} continuous and bounded.
		Arguing as in \cite[Remark 5.11]{Sauerbrey_2021} one finds that this sigma-field coincides with $\tilde{\mathfrak{G}}_s$. Another application of Vitali's convergence theorem, using continuity in time of the random variables in \eqref{Eq92} and the moment estimates \eqref{Eq94}, \eqref{Eq95}, \eqref{Eq96} and \eqref{Eq97} once more yields that \eqref{Eq92} holds also for bounded, $\tilde{\mathfrak{F}}_s$-measurable $\tilde{\rho}$. Consequently, an application of \cite[Proposition A.1]{hofmanova2013} leads to \eqref{Eq89}.
	\end{proof}
Before completing the proof of Theorem \ref{thm_main}, we deduce versions of the a-priori estimates, which we derived for ${u}_\epsilon$,  also for $\tilde{u}$.
\begin{lemma}\label{Lemma_APRIORI}
We assume that $p\in  (n+4,7)$, $r\in(\frac{n+4}{2},\frac{7}{2})$,  $\gamma\in(0,\frac{1}{2})$, $\mu\in (\frac{n+4}{n+2},\frac{7}{n+2})$, $\nu\in (1,\frac{7}{n+4})$ and $\alpha\in (-1,2-n)$.
\begin{enumerate}[label=(\roman*)]
	\item \label{Item_AP_1} The estimates \eqref{Eq224} and \eqref{Eq225} hold, whenever their right-hand side is finite.
	\item \label{Item_AP_2} The estimate \eqref{Eq243} holds with $p_{\mu,\nu}=\max\{ \mu(n+2), \nu(n+4) \}$, if the right-hand side is finite.
	\item \label{Item_AP_3} We have almost surely $ \tilde{u}^\frac{\alpha+n+1}{4}\in L^4(0,T;W^{1,4}(\TT))$, $ \tilde{u}^\frac{\alpha+n+1}{2}\in L^2(0,T;H^2(\TT)) $ and it holds \eqref{Eq210}, whenever its right-hand side is finite.
\end{enumerate}
\end{lemma}
\begin{proof}
	For \ref{Item_AP_1}, we only verify \eqref{Eq225}, because \eqref{Eq224} can be derived analogously. Since almost surely $\tilde{u}_\epsilon\to \tilde{u}$ in $\XS$, we conclude
	\begin{align}
		\tilde{\EE}\bigl[
		\|\partial_x \tilde{u}\|_{L^r([0,T]\times \TT)}^r\bigr] \,&\le\,
		\tilde{\EE}\Bigl[ \liminf_{\epsilon\searrow 0}
		\|\partial_x \tilde{u}_\epsilon\|_{L^r([0,T]\times \TT)}^r\Bigr]
		\\&
		\le\, \liminf_{\epsilon\searrow 0} \tilde{\EE}\bigl[
		\|\partial_x \tilde{u}_\epsilon\|_{L^r([0,T]\times \TT)}^r
		\bigr]
		\\&
		=\, \liminf_{\epsilon\searrow 0} {\EE}\bigl[
			\EE\bigl[ \|\partial_x u_\epsilon\|_{L^r([0,T]\times \TT)}^r\,\big|\, \mathfrak{F}_0
		\bigr]
		\bigr]
		\\&
		\lesssim_{n, r, \Lambda, T}
		\, \liminf_{\epsilon\searrow 0}
		\EE\bigl[
		(\|u_0\|_{\MM(\TT)}+\epsilon)^{4-r} \bigl(\|u_0\|_{\MM(\TT)}^{2r-n-4}+ \|u_0\|_{\MM(\TT)}^{2r-4}+\epsilon^{2r-n-4}\bigr)\bigr]
		\\&
		= \, \EE\bigl[\|u_0\|_{\MM(\TT)}^{r-n}+ \|u_0\|_{\MM(\TT)}^{r}\bigr],
	\end{align}
	using Fatou's lemma and Lemma \ref{Lemma_int_r}. 
	
	For \ref{Item_AP_2}, we use  Lemma \ref{Lemma_temporal_regularity_New} to conclude 
	\begin{align}\begin{split}\label{Eq247}
			&
	\tilde{\EE}\Bigl[
	\bigl\|\tilde{\chi_{\epsilon}}^{(j)}
	\tilde{u}_\epsilon
	\bigr\|^{\nu}_{W^{\gamma, \frac{2\nu}{2-\nu}}(0,T;W^{-3,\mu}(\TT))}
	\Bigr]
	\\&\quad =\,\EE\Bigl[\mathbbm{1}_{A_j}
	\EE\Bigl[
	\left\|
	u_\epsilon
	\right\|^{\nu}_{W^{\gamma, \frac{2\nu}{2-\nu}}(0,T;W^{-3,\mu}(\TT))}
	\,\Big|\, \mathfrak{F}_0 \Bigr]\Bigr]\\&\quad \lesssim_{\gamma, n, \mu,\nu, \Lambda, T}\, \EE\Bigl[\mathbbm{1}_{A_j}\Bigl(\bigl(
	\|u_0\|_{\MM(\TT)}+\epsilon
	\bigr)^{(n-1-\frac{n^2}{p_{\mu,\nu}})\nu}\,+\, 
	\bigl(
	\|u_0\|_{\MM(\TT)}+\epsilon
	\bigr)^{(n+1)\nu}\Bigr)\Bigr],
		\end{split}
\end{align}
yielding a uniform bound on $\tilde{\chi_{\epsilon}}^{(j)}\tilde{u}_\epsilon$ in $L^\nu(\tilde{\Omega},W^{\gamma, \frac{2\nu}{2-\nu}}(0,T;W^{-3,\mu}(\TT)) )$. Hence, up to taking another subsequence, we can assume that $\tilde{\chi_{\epsilon}}^{(j)}\tilde{u}_\epsilon$ admits a weak limit in $L^\nu(\Omega,W^{\gamma, \frac{2\nu}{2-\nu}}(0,T;W^{-3,\mu}(\TT)) )$. Since almost surely $\tilde{\chi_{\epsilon}}^{(j)}\to \mathbbm{1}_{\tilde{A}_j}$ in $\Xind$ and $\tilde{u}_\epsilon\to \tilde{u}$ in $\Xp$, we also have $\tilde{\chi_{\epsilon}}^{(j)}\tilde{u}_\epsilon\to \mathbbm{1}_{\tilde{A}_j}\tilde{u} $ in $L^p([0,T]\times \TT)$. Hence, using Vitali's convergence theorem and \eqref{Eq240}, we deduce that
\[
\EE\bigl[ \bigl\|\tilde{\chi_{\epsilon}}^{(j)}\tilde{u}_\epsilon -  \mathbbm{1}_{\tilde{A}_j}\tilde{u} \bigr\|_{L^p([0,T]\times \TT)}^p\bigr]\,\to \, 0,
\]
such that the weak limit of  $\tilde{\chi_{\epsilon}}^{(j)}\tilde{u}_\epsilon$ has to be $\mathbbm{1}_{\tilde{A}_j}\tilde{u}$. By lower semicontinuity of the norm with respect to weak convergence, Fatou's lemma and \eqref{Eq247}, we conclude that
\begin{align}&
	\tilde{\EE}\Bigl[\mathbbm{1}_{\tilde{A}_j}
	\|
	\tilde{u}
	\|^{\nu}_{W^{\gamma, \frac{2\nu}{2-\nu}}(0,T;W^{-3,\mu}(\TT))}
	\Bigr]
	\\&\quad 
	\le \,\liminf_{\epsilon\searrow 0} 
		\tilde{\EE}\Bigl[
	\bigl\|\tilde{\chi_{\epsilon}}^{(j)}
	\tilde{u}_\epsilon
	\bigr\|^{\nu}_{W^{\gamma, \frac{2\nu}{2-\nu}}(0,T;W^{-3,\mu}(\TT))}
	\Bigr]
	\\&\quad \lesssim_{\gamma, n, \mu,\nu, \Lambda, T}\, \liminf_{\epsilon\searrow 0} \EE\Bigl[\mathbbm{1}_{A_j}\Bigl(\bigl(
	\|u_0\|_{\MM(\TT)}+\epsilon
	\bigr)^{(n-1-\frac{n^2}{p_{\mu,\nu}})\nu}\,+\, 
	\bigl(
	\|u_0\|_{\MM(\TT)}+\epsilon
	\bigr)^{(n+1)\nu}\Bigr)\Bigr]
	\\&\quad =\, \EE\Bigl[\mathbbm{1}_{A_j}\Bigl(
	\|u_0\|_{\MM(\TT)}^{(n-1-\frac{n^2}{p_{\mu,\nu}})\nu}\,+\, 
	\|u_0\|_{\MM(\TT)}
	^{(n+1)\nu}\Bigr)\Bigr].
\end{align}
It remains to sum over $j\in \NN$ to obtain \eqref{Eq243}. Analogously,  \ref{Item_AP_3} follows from Proposition \ref{Thm_Ito_s_formula_New}.
\end{proof}

	\begin{proof}[Proof of Theorem \ref{thm_main}] 
		We first show that $(\tilde{\Omega},\tilde{\mathfrak{A}}, \tilde{\PP})$, $\tilde{\mathfrak{F}}$, $(\tilde{\beta}^{(k)})_{k\in \ZZ}$ and $\tilde{u}$ indeed constitute a martingale solution to \eqref{Eq_STFE_rewritten} in the sense of Definition \ref{Defi_sol}. 
		To this end, we observe that	almost surely $\tilde{u}_\epsilon(t)\ge 0$ for all $t\in [0,T]$ since it is equidistributed to $u_\epsilon$. Consequently, its limit $\tilde{u}$ in $\Xc$ satisfies almost surely $\tilde{u}(t)\ge 0$ as well and in particular $\tilde{u}(t)$ is a non-negative measure on $\TT$ for all $t\in [0,T]$ by \cite[Theorem 6.22]{LiebLoss2001}. Moreover, since $\tilde{u}(t)$ is $\tilde{\mathfrak{F}}_t$-$\mathfrak{B}(H^{-1}(\TT))$-measurable by the definition of $\tilde{\mathfrak{F}}$, the real valued random variables $\langle \tilde{u}(t), \vp\rangle$ for $\vp\in C^\infty(\TT)$ are $\tilde{\mathfrak{F}}_t$-measurable. By approximation, the same holds when $\vp\in C(\TT)$, such that $\tilde{u}$ is $\tilde{\mathfrak{F}}$-adapted with values in $(\mathcal{M}(\TT), \Z)$. Moreover, by the divergence form of \eqref{Eq_1}, the total variation norm $\|\tilde{u}(t)\|_{\MM(\TT)}=  \langle\tilde{u}(t), \mathbbm{1}_{\TT}\rangle$ is constant in time. Using this, we show vague continuity of $\tilde{u}$ in $\MM(\TT)$ and assume for contradiction that $\tilde{u}$ is not continuous at some time $t$. Then, there is a vaguely open neighborhood $\mathcal{O}\subset \MM(\TT)$ of $\tilde{u}(t)$ and a sequence $t_n\to t$ with $\tilde{u}(t_n)\notin \mathcal{O}$ for all $n\in \NN$. However, since $\tilde{u}$ is bounded in $\MM(\TT)$, a subsequence of $\tilde{u}(t_n)$ converges vaguely to some $\nu\in \MM(\TT)$ by the Banach-Alaoglu theorem and since $\tilde{u}\in \Xc$ it must hold $\nu = \tilde{u}(t)$, contradicting $\tilde{u}(t_n)\notin \mathcal{O}$.	As demonstrated in Remark \ref{Rem_Intbility}, the integrability conditions \eqref{Eq255} follow from $\tilde{u}$ lying in $\XL$ and $\XS$. Lemma \ref{Lemma_BM} states that $(\tilde{\beta}^{(k)})_{k\in \ZZ}$ is  a family of independent $\tilde{\mathfrak{F}}$-Brownian motions. Since Lemma \ref{Lemma_new_partition} and Lemma \ref{Lemma_SPDE} imply that $\tilde{u}$ satisfies \eqref{Eq_1}, we showed that the quadruple $(\tilde{\Omega},\tilde{\mathfrak{A}}, \tilde{\PP})$, $\tilde{\mathfrak{F}}$, $(\tilde{\beta}^{(k)})_{k\in \ZZ}$, $\tilde{u}$ suffices Definition \ref{Defi_sol}.
		
		To verify that $\tilde{u}(0)\sim u_0$, we observe that the sets of the form
		\[
		\bigl\{
		\nu \in \MM(\TT)\,\big|\,(
		\langle \nu, \vp_1 \rangle,\dots, 
		\langle \nu, \vp_l \rangle)\,\in\, A
		\bigr\}
		\]
		for $\vp_1, \dots, \vp_l\in C^\infty(\TT)$ and Borel sets $A\subset \RR^l$ form an intersection stable generator of $\Z$ by density of $C^\infty(\TT)$ in $C(\TT)$. Since $\tilde{u}_\epsilon (0)\sim u_{0,\epsilon}$ as $H^{-1}(\TT)$-valued random variables, almost surely $u_{0,\epsilon}\to u_0$ vaguely and $\tilde{u}_\epsilon (0)\to \tilde{u}(0)$ in $H^{-1}(\TT)$ as $\epsilon \searrow 0$, it holds 
		\[
		(
		\langle u_0, \vp_1 \rangle,\dots, 
		\langle u_0, \vp_l \rangle)\,\sim\, 
		(
		\langle \tilde{u}(0), \vp_1 \rangle,\dots, 
		\langle \tilde{u}(0), \vp_l \rangle),
		\]
		yielding that the laws of $u_0$ and $\tilde{u}(0)$ on $(\MM(\TT), \Z)$ coincide. Due to Corollary \ref{Cor_SS}, we have $\tilde{u}\in \Xc\cap \XL\cap \XS$.
		Together with Lemma \ref{Lemma_APRIORI} \ref{Item_AP_1}, \ref{Item_AP_2} this leads to
		\ref{Main_Item2} and \ref{Main_Item6}.
		 The claim in \ref{Main_Item_Mass} was already checked at the beginning of this proof. Part \ref{Main_Item5} is the content of Lemma \ref{Lemma_APRIORI} \ref{Item_AP_3}.
	\end{proof}

	\appendix
	
	\section{Projective limits of locally convex vector spaces}\label{Appendix_lcs}
	We give a short summary of useful facts on topological vector spaces following \cite[Section II.4, Section II.5]{Schaefer_Wolff_99}. In the following, we consider vector spaces over $\mathbbm{R}$ equipped with a topology. Such a tuple is a topological vector space, if addition and scalar multiplication are  continuous mappings. A Hausdorff topological vector space $\X$ is called locally convex, if every neighborhood of a point $x\in \X$ contains a convex neighborhood of $x$.
	Since balls are convex, every normed vector space is locally convex. Moreover, if $\X$ is a normed vector space, and we denote its topological dual by $\X^*$, the weak topology on $\X$ admits the collection of sets
	\begin{equation}\label{Eq74}
		\bigl\{ y\in \X \,\big|\,| \langle y-x, x_i^*\rangle|<\delta,\;i\in \{1,\dots , j\} \bigr\}
	\end{equation}
	for
	$x_i^*\in \X^*$, $\delta>0$
	as a neighborhood basis at $x\in \X$. Since the set \eqref{Eq74} is convex, $\X$ with its weak topology is locally convex as well. Now, let $\X_l$ for $l\in \NN$ be a Hausdorff, locally convex space, such that $\X_{l+1}\hookrightarrow \X_{l}$ continuously. Then the projective limit of $(\X_l)_{l\in \NN}$ is the space $\X=\bigcap_{l\in \NN}\X_l $ equipped with the coarsest topology, such that each of the embeddings $\X\hookrightarrow \X_l$ is continuous, and is itself a Hausdorff, locally convex topological vector space again.
	\begin{lemma}\label{Lemma_compactness_in_lcs}
		Let $K$ be  a subset of the projective limit $\X$ of a sequence of Hausdorff, locally convex spaces $(\X_l)_{l\in \NN}$. Then $K$ is compact in the topology of $\X$, iff $K=\bigcap_{l\in \NN}K_l$, where $K_l$ is a compact subset of $\X_l$.
	\end{lemma}
	\begin{proof}
		If $K$ is compact with the topology of $\X$, it is also compact with the topology of $\X_l$ since the embedding $\X\hookrightarrow \X_l$ is continuous. The claim follows since $K= \bigcap_{l\in \NN} K$, trivially. For the converse implication, we note that $\X$ is homeomorphic to the subset 
		\[D\,=\, 
		\biggl\{  (x_l)_{l\in \NN} \in \prod_{l\in \NN}\X_l\, \bigg| \,\forall i,j: \,x_i=x_j\biggr\}
		\]
		of the topological product space $\prod_{l\in \NN}\X_l $ by \cite[p.52]{Schaefer_Wolff_99}. Denoting the homeomorphism by $f\colon D\,\to \, \X$, we notice that $K= f(D\cap \prod_{l\in \NN} K_l)$. Because $\prod_{l\in \NN} K_l$ is compact by Tychonoff's theorem, it suffices to show that $D\subset \prod_{l\in \NN}\X_l $ is closed, since then $K$ is compact as it is the image of a compact set under a continuous mapping. To do so, let $(x_l)_{l\in \NN}\in \prod_{l\in \NN}\X_l  $ not in $D$, i.e. we assume that there are indices $i<j$ with $x_i\ne x_j$. Since $\X_j\subset \X_i$ and $\X_i$ is Hausdorff, there exist disjoint open neighborhoods $B_i,B_j\subset \X_i$ of $x_i$ and $x_j$, respectively. By the continuity of the embedding $\X_j\hookrightarrow \X_i$, $B_j$ is also open in $\X_j$. Hence, denoting by $p_i$, $p_j$ the continuous projection from $\prod_{l\in \NN}\X_l$ onto the $i$-th and $j$-th component, we have constructed the open neighborhood $p_i^{-1}(B_i) \cap p_j^{-1}(B_j)$ of $(x_l)_{l\in \NN}$ which is disjoint from $D$. Hence, $D$ is closed and the proof is finished.
	\end{proof}
	
	\section{Tightness criterions}
	Let $(\Omega, \mathfrak{A},\PP)$ be a probability space. We recall that a family  $(Y_i)_{i\in \mathcal{I}}$ of mappings defined on $\Omega$ with values in a  topological  space $\X$ is called tight, if for every $\delta$ there is a compact set $K_\delta\subset \mathcal{X}$ such that
	\[
	\PP(\{ Y_i\notin K_\delta\})\,<\, \delta
	\]
	for all $i\in \mathcal{I}$. For this definition to make sense, we only require $\{ Y_i\notin K_\delta\}\in \mathfrak{A}$ for the compact sets $K_\delta$, which is in line with the setting in \cite{jakub98}. This is the case when $(Y_i)_{i\in \mathcal{I}}$ is a family of random variables and $\X$ is Hausdorff, since then the compact sets $K_\delta$ are closed and in particular Borel measurable.
	\begin{lemma}\label{Tightness_Proj_Limit}Let $\X$ be the projective limit of a sequence of Hausdorff, locally convex spaces $(\X_l)_{l\in \NN}$ and $Y_i\colon \Omega \to \X$ for $i\in \mathcal{I}$ a random variable in each of the spaces $\X_l$. The family $(Y_i)_{i\in \mathcal{I}}$ is tight on $\X$ iff it is tight on $\X_l$ for each $l\in \NN$.
	\end{lemma}
	
	\begin{proof}If $(Y_i)_{i\in \mathcal{I}}$ is tight on $\X$ it is also tight on $\X_l$ by continuity of the embedding $\X \hookrightarrow \X_l$. Conversely, if $(Y_i)_{i\in \mathcal{I}}$ is tight on $\X_l$ for each $l\in \NN$, we can for  $\delta>0$ choose  compact subsets $K_{\delta,l}\subset \X_l$ such that
		\[
		\PP(\{ Y_i\notin K_{\delta,l}\})\,<\, \tfrac{\delta}{2^l}
		\]
		for all $i\in \mathcal{I}$.
		The set
		$K_\delta=\bigcap_{l\in \NN}K_{\delta,l}$ 
		is compact with the topology of $\X$
		by Lemma \ref{Lemma_compactness_in_lcs}
		and since 
		\begin{equation}\label{Eq76}
			\PP(\{
			Y_i\notin K_\delta
			\})\,\le \, \sum_{l\in \NN} \PP(\{ Y_i\notin K_{\delta, l}\})\,<\, \delta,
		\end{equation}
		the family
		$(Y_i)_{i\in \mathcal{I}}$ is tight on $\X$.
	\end{proof}
	Using the same argument, one can also reduce tightness in a countable product of topological spaces to tightness in each of the separate spaces.
	\begin{lemma}\label{Lemma_prod_tightness}Let $\X^{(l)}$ be a topological space and $(Y_i^{(l)})_{i\in \mathcal{I}}$ be a family of $\X^{(l)}$-valued mappings defined on $\Omega$ for each $l\in \NN$. If $(Y_i^{(l)})_{i\in \mathcal{I}}$ is tight on $\X^{(l)}$ for each $l\in \NN$, then also the family $((Y_i^{(l)})_{l\in \NN})_{i\in \mathcal{I}}$ lies tight on the topological product $\prod_{l\in \NN} \X^{(l)}$.
	\end{lemma} 
	\begin{proof}For $\delta>0$, $l\in \NN$ there are compact sets $K_{\delta,l}\subset \X^{(l)}$ such that
		\[
		\PP(\{Y_i^{(l)}\notin K_{\delta,l}\})\,<\,\tfrac{ \delta}{2^l}
		\]
		for all $i\in \mathcal{I}$. The set $K_\delta= \prod_{l\in \NN}K_{\delta,l}$ is compact by Tychonoff's theorem and 
		\[
		\PP(\{
		(Y_i^{(l)})_{l\in \NN}\notin K_\delta
		\})\,\le \, \sum_{l\in \NN} \PP(\{ Y_i^{(l)}\notin K_{\delta, l}\})\,<\, \delta
		\]
		yields the claim.
	\end{proof}
	Lastly, we also show that it suffices to show tightness locally on $\Omega$.
	\begin{lemma}\label{Lemma_localized_tightness}Let $\X$ be a Hausdorff topological vector space, $(Y_i)_{i\in \mathcal{I}}$ a family of $\X$-valued random variables and $(A_j)_{j\in \NN}$  a measurable partition of $\Omega$. If $(\mathbbm{1}_{A_j}Y_i)_{i\in \mathcal{I}}$ lies tight on $\mathcal{X}$ for every $j\in \NN$, then $(Y_i)_{i\in \mathcal{I}}$ lies also tight on $\mathcal{X}$.
	\end{lemma}
	\begin{proof}
		For a given $\delta>0$ we choose $J_0\in \NN$ such that
		\[
		\sum_{j=J_0+1}^{\infty} \PP\left(
		A_j\right)\,<\, \tfrac{\delta}{2}.
		\]
		Since  $(\mathbbm{1}_{A_j}Y_i)_{i\in \mathcal{I}}$ lies tight on $\mathcal{X}$ for every $j\in \{
		1,\dots, J_0
		\}$, there exist compact sets $K_\delta^{(j)}\subset \mathcal{X}$ such that
		\[
		\PP\bigl(\bigl\{
		\mathbbm{1}_{A_j}Y_i\notin 
		K_\delta^{(j)}
		\bigr\}
		\bigr)\,<\, \tfrac{\delta}{2J_0}
		\]
		for all $i\in \mathcal{I}$.
		Then, defining the compact set
		\[
		K_\delta\,=\, \bigcup_{j=1}^{J_0} K_\delta^{(j)}\cup \{0\},
		\]
		we can calculate that
		\begin{align}
			\PP \left(
			\left\{
			Y_i\notin K_\delta
			\right\}
			\right)\,&= \,
			\PP \biggl(
			\bigcup_{j\in \NN}\bigl\{
			\mathbbm{1}_{A_j}Y_i \notin K_\delta
			\bigr\}
			\biggr)\,\le \, \sum_{j\in \NN}
			\PP \bigl(
			\bigl\{
			\mathbbm{1}_{A_j} Y_i\notin K_\delta
			\bigr\}
			\bigr)\\&\le \,
			\sum_{j=1}^{J_0}
			\PP \bigl(
			\bigl\{
			\mathbbm{1}_{A_j} Y_i\notin K_\delta^{(j)}
			\bigr\}
			\bigr)\,+\, 
			\sum_{j=J_0+1}^{\infty}
			\PP (
			A_j
			)\,<\, \delta
		\end{align}
		for every $i\in \mathcal{I}$.
	\end{proof}
	
	\section*{Acknowledgements}
	The authors thank G\"unther Gr\"un for communicating the proper approximation of the $\alpha$-entropy functional used in the proof of Proposition \ref{Thm_Ito_s_formula_New}.
	
	This work was partially funded by the Deutsche Forschungsgemeinschaft (DFG, German Research Foundation) via – SFB 1283/2 2021 – 317210226, and co-funded by the European Union (ERC, FluCo, grant agreement No. 101088488). Views and opinions expressed are however those of the authors only and do not necessarily reflect those of the European Union or of the European Research Council. Neither the European Union nor the granting authority can be held responsible for them. 
	
This publication is partially supported by the project \emph{Codimension two free boundary problems} (with project number VI.Vidi.223.019 of the research program ENW Vidi) which is financed by the Dutch Research Council (NWO).

\end{document}